\documentclass[11pt,a4paper,reqno]{amsart} 

\usepackage{amsfonts}
\usepackage{amssymb}
\usepackage{graphicx}
\usepackage{accents}

\usepackage{amstext}

\usepackage{mathtools}

\usepackage{bm}
\usepackage{dsfont}
\usepackage{verbatim}
\usepackage{enumerate}
\usepackage{fullpage}

\usepackage{tikz-cd}

\usepackage{times}

\allowdisplaybreaks

\DeclarePairedDelimiter{\abs}{\lvert}{\rvert}

\DeclarePairedDelimiter{\norm}{\lVert}{\rVert}

\DeclareMathOperator{\sign}{sign}

\numberwithin{equation}{section}

\theoremstyle{plain}
\newtheorem{theorem}{Theorem}
\newtheorem{lemma}[theorem]{Lemma}
\newtheorem{proposition}[theorem]{Proposition}
\newtheorem{corollary}[theorem]{Corollary}

\newtheorem{definition}[theorem]{Definition}

\newtheorem{remark}[theorem]{Remark}

\newtheorem*{notes*}{Notes}
\newtheorem*{theorem*}{Theorem}
\newtheorem*{definition*}{Definition}
\newtheorem*{notation*}{Notation}
\newtheorem*{remark*}{Remark}

\newtheorem{concluding.remark}{Remark}[section]

\newtheorem*{example*}{Example}

\newtheorem*{notation.polynomial}{The polynomial $\mathbf{\mathrm{p}}_{\pi}$}

\newtheorem*{reciprocity.law}{The reciprocity theorem for generalized Gauss sums}

\newtheorem*{mobius.inversion.formula}{M\"{o}bius inversion formula}

\newtheorem*{mobius.inversion.function}{M\"{o}bius function}

\newcommand{\covt}[2]{ \mathrm{Cov} \left[ \mathrm{Tr} (  {#1} ), \mathrm{Tr} (  {#2} ) \right] }

\newcommand{\covtn}[2]{ \mathrm{Cov} [ \mathrm{Tr} (  {#1} ), \mathrm{Tr} (  {#2} ) ] }

\newcommand{\cov}[2]{ \mathrm{cov} \left[  {#1}, {#2} \right] } 

\newcommand{\exptr}[1]{ \mathbb{E} \left[ {#1}\right] } 

\newcommand{\exptrn}[1]{ \mathbb{E} [ {#1} ] } 

\newcommand{\Tr}[1]{\mathrm{Tr} \left( {#1} \right) }

\newcommand{\Trn}[1]{\mathrm{Tr} ( {#1} ) }

\newcommand{\tr}[1]{\mathrm{tr} \left( {#1} \right) } 

\newcommand{\trn}[1]{\mathrm{tr} ( {#1} ) } 

\newcommand{\kernel}[1]{ \mathrm{ker} \left({#1} \right) }

\graphicspath{{figures/}}

\begin{document}

\title[Second order freeness]{Fluctuation moments induced by conjugation with asymptotically liberating random matrix ensembles}

\author[J. Vazquez-Becerra]{Josue Vazquez-Becerra}
\thanks{Supported by Mexican National Council of Science and Technology (CONACYT) ref. 579659/410386}
\address{Department of Mathematics and Statistics, Queen's University, Jeffery	Hall, Kingston, Ontario, K7L 3N6, Canada}
\email{13jdvb1@queensu.ca}
\keywords{Free probability, Random matrices, Fluctuation moments, Discrete Fourier Transform matrix}

\begin{abstract}
	G. Anderson and B. Farrel showed that conjugation of constant matrices by asymptotically liberating random unitary matrices give rise to asymptotic free independence.
	Independent Haar-unitary random matrices and independent Haar-orthogonal random matrices are examples of asymptotically liberating ensembles.
	In this paper, we investigate the fluctuation moments, and higher order moments, induced on constant matrices by conjugation with asymptotically liberating ensembles.
	In particular, we determine fluctuation moments induced by ensembles related to the Discrete Fourier Transform matrix. 
\end{abstract}

\maketitle

\section{Introduction}

\subsection{Background}\hspace*{\fill} \newline

Random matrices are matrix-valued random variables that were first investigated in mathematical statistics \cite{1928} and then in nuclear physics \cite{wigner1955characteristic}. 
Over the years, its study has evolved into a theory with applications to pure and applied sciences such as numerical analysis \cite{edelman2005random}, analytic number theory \cite{keating1993riemann}, and wireless communications \cite{tulino2004random}, to name some. 

One of the main topics in Random Matrix Theory is the study of limiting, or asymptotic, properties of random matrix ensembles. 
The term \textit{random matrix ensemble} 
is used in the literature to refer to a sequence of random matrices $\{X_{N}\}_{N=1}^{\infty}$, or a sequence of families of random matrices $\{\{X_{N,i}\}_{i \in I}\}_{N=1}^{\infty}$, where the considered random matrices increase in size with respect to $N$,   
their \textit{limiting properties} are then those arising from letting $N$ go to infinity. 
Joint eigenvalue distributions, eigenvalues spacing, concentration inequalities, large deviation principles, maximal eigenvalues, and central limit theorems are some examples of limiting properties, 
for an introduction on these subjects one can consult \cite{anderson2010introduction}.

Now, introduced by D. Voiculescu  in his research on von Neumann algebras in  \cite{voiculescu1132symmetries}, free probability theory has played a key role in the study of random matrices when multiple ensembles need to be considered. 
A main notion from free probability is that of \textit{asymptotic free independence}.  
%
\begin{definition}\label{def.asymptotic.freenes.alternative}
	Let $I$ be a non-empty set. 
	Suppose $\{X_{N,i}\}_{N=1}^{\infty}$ is a random matrix ensemble  for each $i\in I$ where each $X_{N,i}$ is a $N$-by-$N$ random matrix. 
	We say that  $\{X_{N,i}\}_{N=1}^{\infty}$ with $i\in I$ are \textbf{asymptotically freely independent} if the following two conditions are satisfied:
	\begin{enumerate}[({AF.}1)]
		\item\label{def.asymp.free.1.alternative} for each index $i \in I $ and every integer $m\geq1$ the limit  $$\displaystyle 
		\lim_{N\rightarrow \infty} \mathbb{E}\left[\tr{X_{N,i}^m}\right], $$ 
		where $\tr{\cdot}$ denotes the normalized trace $\frac{1}{N}\Tr{\cdot}$, exists and
		\item\label{def.asymp.free.2.alternative} for all integers $m\geq 1$, all indexes $i_1, i_2, \ldots, i_m \in I$ 
		satisfying $i_{1} \neq i_{2}, i_{2} \neq i_{3}, \ldots, i_{m-2} \neq i_{m-1}, i_{m-1} \neq i_{m}$, and \textcolor{black}{$i_{m} \neq i_{1}$} and all polynomials $\mathrm{p}_{1}, \mathrm{p}_{2}, \ldots ,\mathrm{p}_{m}$ in the algebra  
		$\mathbb{C}[\mathrm{x}]$, we have
		\[
		\lim_{N\rightarrow \infty} 
		\exptr{\tr{Y_{N,1} Y_{N,2} \cdots Y_{N,m}}} = 0 
		\]
		where $Y_{N,k} =  \mathrm{p}_{k} \left( X_{N,i_k} \right) -  
		\exptr{\tr{\mathrm{p}_{k} \left( X_{N,i_k} \right)}}I_{N}$. 
	\end{enumerate}
\end{definition}

The first connection between free probability and random matrices was established by D. Voiculescu when he shows in \cite{voiculescu1991limit} that independent Gaussian Unitary Ensembles converge to free semi-circular random variables, a result which generalizes Wigner's semicircular law and entails the asymptotic free independence of  independent Gaussian Unitary Ensembles.  
The list of random matrix ensembles exhibiting asymptotic free independence has been extended since then and it now includes: 
independent Wishart ensembles, independent Gaussian Orthogonal ensembles, independent Haar-unitary distributed ensembles, independent Haar-orthogonal distributed ensembles, among others. 
The monograph \cite{MR1217253} and the book \cite{nica2006lectures} are standard introductions to free probability and the recent monograph \cite{mingo-spiecher} is an excellent source presenting multiples directions in which the relation between free probability and random matrices has been extended.

Another result due to D. Voiculescu in \cite{voiculescu1991limit}, and subsequently generalized by other authors, states that conjugation by independent Haar-distributed random unitary matrices gives rise to asymptotic free independence.  
More concretely, assume $D_{N,i}$ is a self-adjoint $N$-by-$N$ deterministic matrix for each index $i \in I$ and each integer $N\geq 1$ and suppose that 
\begin{align}\label{assump.constnt.matrices}
\sup_{N \in \mathbb{N}}  \Vert D_{N,i} \Vert <  \infty 
\quad \text{ and } \quad
\lim_{N \rightarrow \infty} \mathrm{tr}(D_{N,i}^{m})  
\quad \text{ exists } 
\end{align}	 
for all $i \in I $ and $m\geq 1$; 
the random matrix ensembles $\{ D^{}_{N,i} \}_{N=1}^{\infty}$ with $i \in I$ might or might not be asymptotically freely independent, however, if $\{U_{N,i}\}_{i \in I}$ is a family of independent $N$-by-$N$ Haar-unitary distributed random matrices for each $N\geq 1$, then $\{ U^{}_{N,i} D^{}_{N,i} U^{*}_{N,i} \}_{N=1}^{\infty}$ with $i \in I$ are asymptotically freely independent. 
The same conclusion holds if each $U_{N,i}$ is Haar-orthogonal distributed, 
see \cite{mingo2013real}. 

Aiming to enclose all of those unitary random matrix ensembles that give rise to asymptotic free independence when used for conjugation, B. Farrell and G. Anderson introduced in \cite{anderson2014asymptotically} the notion of \textit{asymptotically liberating random matrix ensembles}. 
%
\begin{definition}\label{def.asymp.liberatin}
	Suppose $ U_{N,i} $ is an $N$-by-$N$ unitary random matrix for each index $i \in I$ and each integer $N\geq 1$. 
	The unitary random matrix ensemble $\left\{ \left\{ U_{N,i}  \right\}_{i\in I}\right\}_{N=1}^{\infty}$ is \textbf{asymptotically liberating} if
	for all indexes $i_{1},i_{2},\ldots,i_{m} \in I$ with $i_1 \neq i_2, i_2 \neq i_3, \ldots, i_{m-1}\neq i_m$, and $i_m \neq i_1$
	there exists a constant $C > 0$ depending only on the indexes $i_{1},i_{2},\ldots,i_{m}$ such that 
	\begin{equation*}
	\Big\lvert 
	\exptr{\Tr{U^{}_{N,i_1} A^{}_{N,1} U^{*}_{N,i_1} 
			U^{}_{N,i_2} A^{}_{N,2} U^{*}_{N,i_2} \cdots 
			U^{}_{N,i_m} A^{}_{N,m} U^{*}_{N,i_m} }}
	\Big\rvert  
	\leq 
	C \norm*{A_{N,1}} \norm*{A_{N,2}} \cdots \norm*{A_{N,m}}
	\end{equation*}
	for all integers $N \geq 1$ and all matrices $A_{N,1}, A_{N,2}, \ldots, A_{N,m} \in \text{Mat}_N (\mathbb{C})$ each of trace zero.  
\end{definition}

It follows immediately from the definition above that asymptotically liberating ensembles gives rise to asymptotic free independence when used for conjugation. 
Indeed, suppose $\{ \{ U_{N,i}  \}_{i\in I}\}_{N=1}^{\infty}$ is an asymptotically liberating ensemble and assume $\{ D^{}_{N,i} \}_{N=1}^{\infty}$ with $i \in I$ satisfy (\ref{assump.constnt.matrices}). 
Letting $X_{N,i} = U^{}_{N,i} D^{}_{N,i} U^{*}_{N,i}$, we have (\ref{assump.constnt.matrices}) implies  \textit{({AF.}\ref{def.asymp.free.1.alternative})} from Definition \ref{def.asymptotic.freenes.alternative}; 
moreover, if each  $Y_{N,k}$ is as in \textit{({AF.}\ref{def.asymp.free.2.alternative})}  from Definition \ref{def.asymptotic.freenes.alternative}, then 
\[
Y_{N,1} Y_{N,2} \cdots Y_{N,m} 
= 
(U^{}_{N,i_{1}} A^{}_{N,1} U^{*}_{N,i_{1}} ) ( U^{}_{N,i_{2}} A^{}_{N,2} U^{*}_{N,i_{2}} ) \cdots (U^{}_{N,i_{m}} A^{}_{N,m} U^{*}_{N,i_{m}} ) 
\]
where $A_{N,k}$ denotes the matrix  of trace zero $\mathrm{p}_{k} ( D^{}_{N,i_k}) -  \mathrm{tr} ( \mathrm{p}_{k} ( D^{}_{N,i_k} )) I_{N} $,
but (\ref{assump.constnt.matrices}) also implies that $\sup_{N} \norm{A^{}_{N,k}} < \infty$, and hence \textit{({AF.}\ref{def.asymp.free.2.alternative})} holds. 
As it was intended, independent Haar-unitary random matrix ensembles and independent Haar-orthogonal random matrix ensembles are among those unitary random matrix ensembles shown to be asymptotically liberating, see Theorem 2.8 in \cite{anderson2014asymptotically} or Lemma \ref{lemma.bounded.cumulant.property} below. 

A key feature of asymptotic free independence is that it provides us with universal rules to compute limiting mixed moments out of individual ones.
A \textit{limiting mixed moment} of the ensembles $\{X_{N,i}\}_{N=1}^{\infty}$ with $i\in I$  is a limit of the form 
\begin{align}\label{limiting.moment}
\lim_{N\rightarrow \infty} \exptr{\tr{  X^{}_{N,i_1} 
		X^{}_{N,i_2} \cdots X^{}_{N,i_m}  } }
\end{align}
where at least two of the indexes $i_1, i_2,\ldots,i_m \in I$ are distinct and none of them depend on $N$. 
Thus, if $\{X_{N,i}\}_{N=1}^{\infty}$ with $i\in I$ are asymptotically free independent and  $i_{1}, i_{2} \in I$ are distinct, one can show that
\begin{align*}
\lim_{N\rightarrow \infty} \exptr{\tr{  X^{1}_{N,i_{1}} X^{4}_{N,i_{2}} X^{7}_{N,i_{1}} X^{2}_{N,i_{2}}  } }
= 
\alpha^{(i_{1})}_{8}\alpha^{(i_{2})}_{4}\alpha^{(i_{2})}_{2}
+\alpha^{(i_{2})}_{6}\alpha^{(i_{1})}_{1}\alpha^{(i_{1})}_{7}
-\alpha^{(i_{1})}_{1}\alpha^{(i_{2})}_{4} \alpha^{(i_{1})}_{7} \alpha^{(i_{2})}_{2}
\end{align*}
where $ \alpha_{m}^{(i)}$ denotes $\lim_{N\rightarrow \infty} \mathbb{E}[\mathrm{tr}(X_{N,i}^m) ]$ and is called the \textit{m-th limiting individual moment} of $\{X_{N,i}\}_{N=1}^{\infty}$. 
The relation above, and any other derived from asymptotic free independence to compute mixed moments, is called \textit{universal} since it does not depend on any particular choice of $i_{1}$ and $i_{2}$ and it only requires $\{X_{N,i_{1}}\}_{N=1}^{\infty}$ and $\{X_{N,i_{2}}\}_{N=1}^{\infty}$ to be asymptotically freely independent.  

At this point, one might wonder if there are universal rules for computing limiting mixed moments of higher order out individual ones. 
A \textit{limiting moment of $n$-th order} of the ensembles $\{X_{N,i}\}_{N=1}^{\infty}$ with $i \in I$ is defined to be a limit of the form 
\begin{align}\label{limiting.higher.moment}
\lim_{N\rightarrow \infty}
N^{n-2} \mathfrak{c}_{n} [\Trn{\widetilde{X}_{N,1}},\Trn{\widetilde{X}_{N,2}},\ldots,\Trn{\widetilde{X}_{N,n}} ]
\end{align}
where $\mathfrak{c}_{n}[\cdot, \ldots, \cdot]$ denotes the $n$-th classical cumulant 
and each $\widetilde{X}_{N,k}$ is of the form 
\[
\widetilde{X}_{N,k}=X_{N,i^{(k)}_{1}}X_{N,i^{(k)}_{2}} \cdots X_{N,i^{(k)}_{m_{k}}}
\] 
for some integer $m_{k}\geq 1$ and some indexes $i^{(k)}_{1},i^{(k)}_{2},\ldots,i^{(k)}_{m_{k}} \in I$ not depending on $N$. 
The choice of the normalization factor $N^{n-2}$ appearing in \eqref{limiting.higher.moment} is due to what has been observed for the behavior of \eqref{limiting.higher.moment} when each $X_{N,i}$ is a Gaussian Unitary Ensemble. 
Since the limiting moment (\ref{limiting.higher.moment}) is just a generalization of (\ref{limiting.moment}), we call it \textit{mixed} if at least two of the indexes $i^{(1)}_{1},\ldots,i^{(1)}_{m_{1}},i^{(2)}_{1},\ldots,$ $i^{(2)}_{m_{2}},\ldots,i^{(n)}_{1},\ldots, i^{(n)}_{m_{n}}$ are distinct, and \textit{individual}, otherwise.

The most studied moments of higher order are moments of second order, also known as \textit{fluctuation moments}. 
A fluctuation moment of the ensembles $\{X_{N,i}\}_{N=1}^{\infty}$ with $i \in I$ is then a limit of the form
\begin{align}\label{fluctuation.moment}
\lim_{N \rightarrow \infty} \covtn{X_{N,i_{1}} X_{N,i_{2}} \cdots X_{N,i_{m_{1}}}}	{X_{N,i_{m_{1}+1}} X_{N,i_{m_{1}+2}} \cdots X_{N,i_{m_{1}+m_{2}}}}
\end{align}
for some integers $m_{1},m_{2}\geq 1$ and indexes $i_{1},i_{2},\ldots,i_{m_{1}},i_{m_{1}+1},i_{m_{1}+2},\ldots,i_{m_{1}+m_{2}} \in I$.  
A common practice in free probability theory to determine  (\ref{limiting.higher.moment}), or (\ref{fluctuation.moment}), combinatorially is that of calculating limiting moments of products of cyclically alternating and centered random matrices, as in \textit{(AF.\ref{def.asymp.free.2.alternative})} from Definition \ref{def.asymptotic.freenes.alternative}.
For fluctuation moments, this means one must consider limits of the form
\[
\lim_{N\rightarrow \infty} 
\covt{Y_{N,1} Y_{N,2}   \cdots Y_{N,m_{1}}}	{Z_{N,1} Z_{N,2}  \cdots Z_{N,m_{2}}} 
\]   
where $ Y_{N,k}$ and $Z_{N,l}$ are given by 
\begin{align}\label{yz.centered.matrices}
Y_{N,k} =  \mathrm{p}_{k} \left( X_{N,i_k} \right) -  \exptr{\tr{\mathrm{p}_{k} \left( X_{N,i_k} \right)}}I_{N} 
\ \text{ and } \
Z_{N,l} =  \mathrm{q}_{l} \left( X_{N,j_l} \right) -  
\exptr{\tr{\mathrm{q}_{l}  \left( X_{N,j_l} \right)}}I_{N}
\end{align}
for all polynomials  $\mathrm{p}_{1}, \mathrm{p}_{2},\ldots, \mathrm{p}_{m_{1}}, \mathrm{q}_{1},\mathrm{q}_{2}, \ldots , \mathrm{q}_{m_{2}} \in \mathbb{C}[\mathrm{x}]$ and all indexes $i_1, i_2, \ldots,  i_{m_{_1}}, j_{1},j_{2}, \ldots,  \allowbreak j_{m_{_2}} \in I$ satisfying the condition  
\begin{align}\label{cyclic.alternating.indexes}
i_{1} \neq i_{2}, i_{2} \neq i_{3}, \ldots, i_{m_{_1}-1} \neq i_{m_{_1}},i_{m_{_1}} \neq i_{1} ,  
j_{1} \neq j_{2}, j_{2} \neq j_{3}, \ldots, j_{m_{_2}-1} \neq 
j_{m_{_2}},j_{m_{_2}} \neq j_{1} .
\end{align}

Analyzing the fluctuation moments of complex Gaussian and complex Wishart random matrix ensembles, J. Mingo and R. Speicher found a relation between individual and mixed moments of first and second order and introduced in \cite{mingo2006second} the notion of \textit{asymptotic free independence of second order}. 
%
\begin{definition}\label{def.so.asymptotic.freenes} 
	We say that the random matrix ensembles $\{X_{N,i}\}_{N=1}^{\infty}$ with $i\in I$ are	\textbf{asymptotically freely independent of second order} if they are asymptotically freely independent and the following three conditions are satisfied:
	\begin{enumerate}[({ASOF.}1)]
		\item\label{def.asymp.so.free.1} for each index $i \in I$ and all integers $m,n \geq 1$ the limit  $$\displaystyle 
		\lim_{N\rightarrow \infty} \covt{X_{N,i}^m}{X_{N,i}^n}$$ exists, 
		\item\label{def.asymp.so.free.2} for all integers $m_{1},m_{2} \geq 1$, all indexes $i_1, i_2, \ldots, i_{m_{_1}}, 
		j_{1},j_{2},\ldots,j_{m_{_2}} \in I$ satisfying (\ref{cyclic.alternating.indexes}), and all polynomials $\mathrm{p}_{1}, \mathrm{p}_{2},\ldots, \mathrm{p}_{m_{1}}, \mathrm{q}_{1},\mathrm{q}_{2}, \ldots , \mathrm{q}_{m_{2}}$ in the algebra  $\mathbb{C}[\mathrm{x}]$, if we take 
		\[
		Y_{N}=Y_{N,1} Y_{N,2} \cdots Y_{N,m_{1}}
		\quad \text{ and } \quad
		Z_{N} = Z_{N,1} Z_{N,2} \cdots Z_{N,m_{2}}
		\]
		with $Y_{N,k}$ and $Z_{N,l}$ given by (\ref{yz.centered.matrices}) for $1 \leq k \leq m_{1}$ and $1 \leq l \leq m_{2}$, we have
		\begin{align}\label{fluctuation.moments.sof}
		\lim_{N\rightarrow \infty} 
		\covt{ Y_{N} }{ Z_{N} }
		=
		\delta_{m_{1},m_{2}} \lim_{N\rightarrow \infty}	
		\sum^{m_{1}}_{l=1}\prod^{m_{2}}_{k=1}\exptr{\tr{Y_{N,k}Z_{N,l-k}}}
		\end{align}
		where $l-k$ is taken modulo $m_{2}$, and 
		\item\label{def.asymp.so.free.3} for ever integer $n\geq 3$, all polynomials $\mathrm{p}_{1}, \mathrm{p}_{2},\ldots, \mathrm{p}_{n}$ in the algebra of non-commutative polynomials $\mathbb{C} \left\langle \mathrm{x}_{i} \mid i \in I \right\rangle$,
		letting $Y_{N,k} = \mathrm{p}_{k} \left( \{X_{N,i}\}_{i\in I} \right)$, we have
		\[
		\lim_{N\rightarrow \infty}   \mathfrak{c}_{n} \left[ \Tr{Y_{N,1}}, 
		\Tr{Y_{N,2}},\ldots ,\Tr{Y_{N,n}}\right]  = 0
		\]
	\end{enumerate}
\end{definition}

Similar to asymptotic free independence, asymptotic free independence of second order provides us with universal rules, via the conditions  \textit{(ASOF.\ref{def.asymp.so.free.1})} and \textit{(ASOF.\ref{def.asymp.so.free.2})} above, to calculate limiting mixed fluctuation moments out of individual ones. 
Moreover, independent Gaussian Unitary Ensembles are asymptotically freely independent of second order and conjugation by independent Haar-unitary random matrix ensembles leads to asymptotic free independence of second order, see \cite{mingo2006second} and \cite{mingo2007second}, respectively. 

However, in contrast to moments of first order, fluctuation moments induced by Haar-unitary random matrix ensembles and those induced by Haar-orthogonal random matrix ensembles differ. 
Investigating fluctuation moments of independent Gaussian Orthogonal Ensembles, E. Redelmeier proved in \cite{redelmeier2014real} that if each  $\{X_{N,i}\}_{i \in I}$ forms a family of independent Gaussian Orthogonal Ensembles for every $N\geq 1$, then the ensembles $\{X_{N,i}\}_{N=1}^{\infty}$ with $i\in I$  satisfy \textit{(ASOF.\ref{def.asymp.so.free.1})} and \textit{(ASOF.\ref{def.asymp.so.free.3})} from Definition \ref{def.so.asymptotic.freenes} but \textit{(ASOF.\ref{def.asymp.so.free.2})} has to be replaced by the following:
\begin{enumerate}[\textit{({ASOF.}2')}]
	\item\label{def.asymp.so.free.2.real}\textit{ for all integers $m_{1},m_{2} \geq 1$, all indexes $i_1, i_2, \ldots, i_{m_{_1}}, 
	j_{1},j_{2},\ldots,j_{m_{_2}} \in I$ satisfying (\ref{cyclic.alternating.indexes}), and all polynomials $\mathrm{p}_{1}, \mathrm{p}_{2},\ldots, \mathrm{p}_{m_{1}}, \mathrm{q}_{1},\mathrm{q}_{2}, \ldots , \mathrm{q}_{m_{2}}$ in the algebra  $\mathbb{C}[\mathrm{x}]$, if we take 
	\[
	Y_{N}=Y_{N,1} Y_{N,2} \cdots Y_{N,m_{1}}
	\quad \text{ and } \quad
	Z_{N} = Z_{N,1} Z_{N,2} \cdots Z_{N,m_{2}}
	\]
	with $Y_{N,k}$ and $Z_{N,l}$ given by (\ref{yz.centered.matrices}) for $1 \leq k \leq m_{1}$ and $1 \leq l \leq m_{2}$, we then have
	\begin{align}\label{fluctuation.moments.rsof}
	\lim_{N\rightarrow \infty}
	\covt{Y_{N}}{Z_{N} }  
	= 
	\delta_{m_{1},m_{2}} 
	\lim_{N\rightarrow \infty}	
	\sum^{m_{1}}_{l=1} \left(
	\prod^{m_{2}}_{k=1}\exptrn{\trn{Y_{N,k}Z_{N,l-k}}} + 
	\prod^{m_{2}}_{k=1}\exptrn{\trn{Y^{}_{N,k}Z^{T}_{N,l+k}}} \right)
	\end{align}
	where $l-k$ and $l+k$ are taken modulo $m_{2}$.}
\end{enumerate}
Asymptotically freely independent ensembles satisfying \textit{(ASOF.\ref{def.asymp.so.free.1})}, \textit{(ASOF.\ref{def.asymp.so.free.2}')}, and \textit{(ASOF.\ref{def.asymp.so.free.3})}  are called \textit{asymptotically freely independent of second order in the real sense}. 
Generalizing the findings of E. Redelmeier in \cite{redelmeier2014real}, it was showed by J. Mingo and M. Popa in \cite{mingo2013real} that independent orthogonally-invariant ensembles are asymptotically freely independent of second order in the real sense, and therefore,  the fluctuation moments induced by Haar-orthogonal ensembles are not described by (\ref{fluctuation.moments.sof}) but (\ref{fluctuation.moments.rsof}) instead. 

\subsection{Objectives and main results}\hspace*{\fill} \newline 

The aim of this paper is investigate the behavior of the fluctuation moments, and higher order moments, resulting from conjugation by  asymptotically liberating ensembles.
Since independent Haar-unitary and independent Haar-orthogonal are both asymptotically liberating but the fluctuation moments each of them induces are distinct, we already know that the induced fluctuation moments depend on the specific liberating ensemble used for conjugation.
However, it might well be the case that the relations in (\ref{fluctuation.moments.sof}) and (\ref{fluctuation.moments.rsof}) cover all possible behaviors for fluctuation moments induced by liberating ensembles, our first result shows that this is actually not the case, adding even more evidence that fluctuation moments are more intricate than its first order counterpart. 

It is illustrative and good for comparison to restate what the relations in (\ref{fluctuation.moments.sof}) and in (\ref{fluctuation.moments.rsof}) yield when Haar-unitary ensembles and Haar-orthogonal ensembles are used of conjugation. 
So, let us assume $X_{N,1} = U^{}_{N,1}  D^{}_{N,1}  U^{*}_{N,1}$ and $X_{N,2} = U^{}_{N,2}  D^{}_{N,2}  U^{*}_{N,2}$ for each integer $N \geq 1$ where each sequence $\{ D^{}_{N,i} \}_{N=1}^{\infty}$ satisfies (\ref{assump.constnt.matrices}) and $\{  U_{N,1} , U_{N,2} \}_{N=1}^{\infty}$ is an asymptotically liberating ensemble. 
Note that if the random matrices $Y_{N}$ and $Z_{N}$ are as in \textit{(ASOF.\ref{def.asymp.so.free.2})} from Definition \ref{def.so.asymptotic.freenes}, then we can write 
\[
Y_{N} = \big( U^{}_{N,i_{1}}A^{}_{N,1}U^{*}_{N,i_{1}} \big)
\big( U^{}_{N,i_{2}}A^{}_{N,2}U^{*}_{N,i_{2}} \big) \cdots
\big( U^{}_{N,i_{2m_{1}}}A^{}_{N,2m_{1}}U^{*}_{N,i_{2m_{1}}} \big)
\]
and
\[
Z_{N} =\big( U^{}_{N,j_{1}}B^{}_{N,1}U^{*}_{N,j_{1}} \big)
\big( U^{}_{N,j_{2}}B^{}_{N,2}U^{*}_{N,j_{2}} \big) \cdots
\big( U^{}_{N,j_{2m_{1}}}B^{}_{N,2m_{1}}U^{*}_{N,j_{2m_{1}}} \big)
\]
where $A^{}_{N,k}$ and $B^{}_{N,l}$ are deterministic matrices of trace zero given by 
\begin{align}\label{ab.centered.matrices.alternative}
A_{N,k} =  \mathrm{p}_{k} \left( D_{N,i_k} \right) -  \tr{\mathrm{p}_{k} \left( D_{N,i_k} \right)}I_{N} 
\quad \text{ and } \quad
B_{N,l} =  \mathrm{q}_{l} \left( D_{N,j_l} \right) -  
\tr{\mathrm{q}_{l} \left( D_{N,j_l} \right)}I_{N}
\end{align}
for $1 \leq k \leq 2m_{1}$ and $1 \leq l \leq 2m_{2}$. 
For simplicity, and without loss of generality, let us assume $i_{1} = j_{1}$. 
Now, if $U_{N,1}$ and $U_{N,2}$ are independent Haar-unitary ensembles, it follows from \textit{(AF.\ref{def.asymp.free.2.alternative})} in Definition \ref{def.asymptotic.freenes.alternative} and the relation in  (\ref{fluctuation.moments.sof}) that the covariance $\covt{ Y_{N} }{ Z_{N} } $ converges to  
\begin{align}\label{fluct.moments.second.order.free}
\lim_{N \rightarrow \infty} \delta_{m_{1},m_{2}} 
\sum^{m_{1}}_{l=1}\prod^{2m_{2}}_{k=1} \tr{A_{N,k} B_{N,2l-k}} 
\end{align}
as $N$ goes to infinity.  
On the other hand, if $U_{N,1}$ and $U_{N,2}$ are independent Haar-orthogonal ensembles, then \textit{(AF.\ref{def.asymp.free.2.alternative})} and  (\ref{fluctuation.moments.rsof}) imply that  $\covt{ Y_{N} }{ Z_{N} } $ converges to 
\begin{align}\label{fluct.moments.real.second.order.free}
\lim_{N \rightarrow \infty} \delta_{m_{1},m_{2}} 
\sum^{m_{1}}_{l=1} \left(
\prod^{2m_{2}}_{k=1} \tr{A_{N,k} B_{N,2l-k}} +
\prod^{2m_{2}}_{k=1} \tr{A^{}_{N,k} B^{T}_{N,2l+k}}
\right) 
\end{align}  
as $N$ goes to infinity. 
Note that (\ref{assump.constnt.matrices}) alone guarantees the existence of each of the limits above if each matrix $D_{N,i}$ is self-adjoint, regardless of what $U_{N,1}$ and $U_{N,2}$ are.  

Another ensemble shown to be asymptotically liberating, see Corollary 3.2 in \cite{anderson2014asymptotically}, and a main focus in this paper, is the unitary random matrix ensemble  $\{W_{N},  H_{N}W_{N} /\sqrt{N}, X_{N}H_{N}W_{N}/\sqrt{N} \}$ where $W_{N}$ is a random $N$-by-$N$ signed permutation matrix, $X_{N}$ is a random  $N$-by-$N$ signature matrix independent from $W_{N}$, and $H_{N}$ is the $N$-by-$N$ Discrete Fourier Transform matrix. 
Our first result shows that if we take pairs of distinct unitary matrices  $U_{N,1}$ and $U_{N,2}$ from  $\{W_{N},  H_{N}W_{N} /\sqrt{N}, X_{N}H_{N}W_{N}/\sqrt{N} \}$ and use them for conjugation, then the resulting fluctuation moments vary with each pair and differ from those in \eqref{fluct.moments.second.order.free} and in \eqref{fluct.moments.real.second.order.free}. 
%
%
%
%

\begin{theorem}\label{thm.second.order.asymptotics}
	Let $D_{N,1}$ and $D_{N,2}$ be $N$-by-$N$ self-adjoint matrices for each integer $N\geq 1$ so that each $\{ D^{}_{N,i} \}_{N=1}^{\infty}$ satisfies (\ref{assump.constnt.matrices}). 
	Suppose $X_{N,1} =  U^{}_{N,1}  D^{}_{N,2}  U^{*}_{N,1} $ and $X_{N,2} = U^{}_{N,2}  D^{}_{N,2}  U^{*}_{N,2} $ where $U_{N,1}$ and $U_{N,2}$ are distinct matrices from  $\{W_{N},  H_{N}W_{N} /\sqrt{N}, X_{N}H_{N}W_{N}/\sqrt{N} \}$.  
	If $Y_{N}$ and $Z_{N}$ are given by $Y_{N}=Y_{N,1} Y_{N,2} \cdots Y_{N,2m_{1}}$ and $Z_{N} = Z_{N,1} Z_{N,2} \cdots Z_{N,2m_{2}}$ where $Y_{N,k}$ and $Z_{N,l}$ are defined as in \eqref{yz.centered.matrices} for some polynomials $\mathrm{p}_{1}, \mathrm{p}_{2},\ldots, \mathrm{p}_{2m_{1}}, \mathrm{q}_{1},   \mathrm{q}_{2}, \ldots , \mathrm{q}_{2m_{2}}\in \mathbb{C}[\mathrm{x}]$ and some indexes $i_1, i_2, \ldots, i_{2m_{_1}}, j_{1},j_{2},\ldots,j_{2m_{_2}} \in \{1,2\}$ satisfying (\ref{cyclic.alternating.indexes}) and $i_{1}=j_{1}$, then the following holds:
	\begin{enumerate}
		\item\label{thm.second.order.asymptotics.1} $U_{N,1} = W_{N} $ and $U_{N,2}= H_{N}W_{N}/\sqrt{N}$ implies
		\begin{align*}
		\covt{ Y_{N} }{ Z_{N} } 
		=& 
		\delta_{m_{1},m_{2}} \sum_{l=1}^{m_{1}}	 \left( 
		\prod_{k=1}^{2m_{1}} \tr{A_{N,k} B_{N,2l-k}} 
		+ 		\prod_{k=1}^{2m_{1}} \tr{A^{}_{N,k} B_{N,2l+k-1}^{T}}	
		\right) 
		\\ &	+
		O \left( N^{-\frac{1}{2	}}\right)
		\end{align*}	
		\item\label{thm.second.order.asymptotics.2}  $U_{N,1} = W_{N} $ and $U_{N,2}= X_{N}H_{N}W_{N}/\sqrt{N}$ implies 
		\begin{align*}
		\covt{ Y_{N} }{ Z_{N} }  
		=& 
		\delta_{m_{1},m_{2}} \sum_{l=1}^{m_{1}} \left( 
		\prod_{k=1}^{2m_{1}} \tr{A_{N,k} B_{N,2l-k}} 
		+ 		\prod_{k=1}^{2m_{1}} \tr{A_{N,k} \circ B_{N,2l+k-1}}	
		\right)
		\\ &	+
		O \left( N^{-\frac{1}{2	}}\right)
		\end{align*}
		\item\label{thm.second.order.asymptotics.3}  $U_{N,1} = H_{N}W_{N}/\sqrt{N} $ and $U_{N,2}= X_{N}H_{N}W_{N}/\sqrt{N}$ implies 
		\begin{align*}
		\covt{ Y_{N} }{ Z_{N} }  
		=& 
		\sum_{l_{1}=1}^{2m_{1}} \sum_{l_{2}=1}^{2m_{2}} 
		\prod_{k_{1}=1}^{m_{1}} 
		\tr{A_{N,l_{1}+k_{1}-1} A_{N,l_{1}-k_{1}}}
		\cdot
		\prod_{k_{2}=1}^{m_{2}}		
		\tr{B_{N,l_{2}+k_{2}-1} B_{N,l_{2}-k_{2}}}
		\\	&		
		+ \delta_{m_{1},m_{2}} \sum_{l=1}^{2m_{1}} 
		\left( \prod_{k=1}^{2m_{1}} \tr{A_{N,k}B_{N,l-k}} \right)	
		+
		O \left( N^{-\frac{1}{2	}}\right) 
		\end{align*}
	\end{enumerate}
	with $A_{N,k}$ and $B_{N,l}$ defined as in \eqref{ab.centered.matrices.alternative}, $2l-k$, $2l+k-1$, $l_{1}+k_{1}-1$, $l_{1}-k_{1}$, and $l-k$  interpreted modulo $2m_{1}$, and  $l_{2}+k_{2}-1$ and $l_{2}+k_{2}$ interpreted module $2m_{2}$.
\end{theorem}
%
%
%

Evidence of the existence of second order behaviors, other than second order free independence and second order free independence in the real sense, is not new, at least, from an algebraic point of view.  
We mention in particular the papers \cite{hao2018combinatorial} and \cite{jiao2015fluctuations} where the authors analyze fluctuation moments of matrices with entries from a possibly non-commutative unital algebra and obtain different relations from those mentioned above.  
Now, notice (\ref{assump.constnt.matrices}) alone is not enough to guarantee the existence of limiting second order behaviors in Theorem \ref{thm.second.order.asymptotics}, in contrast to \eqref{fluct.moments.second.order.free} and \eqref{fluct.moments.real.second.order.free}. 
For instance, if we want to take the limit as $N$ goes to infinity in \textit{(\ref{thm.second.order.asymptotics.3})} from Theorem \ref{thm.second.order.asymptotics}, we need $\{D_{N,1}\}_{N=1}^{\infty}$ and $\{D_{N,2}\}_{N=1}^{\infty}$ to have a joint limiting distribution, i.e., we need that the limit
$\lim_{N \rightarrow \infty} \mathrm{tr}(D_{N,i_{1}}^{} D_{N,i_{2}}^{} \cdots D_{N,i_{m}}^{} )  $
exists for all integers $m\geq 1$ and all indexes $i_{1},i_{2},\ldots,i_{m} \in \{1,2\}$. 
This shows we can not expect a classification for universal products of second order, in the spirit of \cite{muraki2003five} or \cite{MR1426844}, encompassing all of the second order behaviors exhibited by random matrices. 

It would be desirable to have a master theorem encompassing all three cases from Theorem \ref{thm.second.order.asymptotics}, but the combinatorics, to which we arrive from our analysis of induced fluctuation moments, seems already too intricate when we consider each case separately.  
On this regard, although we make no use of the theory of traffic free independence of C. Male, see \cite{maletraffic}, it is very likely that our results will find a nice expression in terms of traffic algebras and we hope to return to this later.

Despite the fact that no pair of distinct unitary matrices $U_{N,1}$ and $U_{N,2}$ from the ensemble $\{W_{N}, \allowbreak H_{N}W_{N} /\sqrt{N},  X_{N}H_{N}W_{N}/\sqrt{N} \}$ leads to asymptotic free independence of second order when used for conjugation, 
it turns out not much more is needed to achieve this end, at least, partially. 
More concretely, if $U_{N,1}=W_{N,1}$ and $U_{N,2}=H_{N}W_{N,2}/\sqrt{N}$ where $W_{N,1}$ and $W_{N,2}$ are independent $N$-by-$N$ uniformly-distributed signed permutation matrices, then the fluctuation moments induced by $\{U_{N,1},U_{N,2}\}_{N=1}^{\infty}$ are the same as if $U_{N,1}$ and $U_{N,2}$ were independent Haar-unitary, i.e., the induced fluctuation moments are described by \eqref{fluct.moments.second.order.free}. 
Thus, we can think of $\{W_{N,1},H_{N}W_{N,2} /\sqrt{N} \}_{N=1}^{\infty}$ as an \textit{asymptotically liberating ensemble of second order}.
%
%
%
%

\begin{theorem}\label{thm.second.order.freeness}
	Let $D_{N,1}$ and $D_{N,2}$ be $N$-by-$N$ self-adjoint matrices for each integer $N\geq 1$ so that each $\{ D^{}_{N,i} \}_{N=1}^{\infty}$ satisfies (\ref{assump.constnt.matrices}). 
	Suppose $X_{N,1} =  U^{}_{N,1}  D^{}_{N,2}  U^{*}_{N,1} $ and $X_{N,2} = U^{}_{N,2}  D^{}_{N,2}  U^{*}_{N,2} $ where $U_{N,1}=W_{N,1}$ and $U_{N,2}=H_{N}W_{N,2}/\sqrt{N}$. 
	Then $\{X_{N,1}\}_{N=1}^{\infty}$ and $\{X_{N,2}\}_{N=1}^{\infty}$ are asymptotically freely independent and they satisfy \textit{(ASOF.\ref{def.asymp.so.free.1})} and  \textit{(ASOF.\ref{def.asymp.so.free.2})} from Definition \ref{def.so.asymptotic.freenes}.
	In particular, if $Y_{N}$, $Z_{N}$, $A_{N,k}$, and $B_{N,l}$ are given as in the previous theorem, then  
	\begin{align}\label{thm.second.order.freeness.moments}
	\covt{ Y_{N} }{ Z_{N} } 
	=& 
	\delta_{m_{1},m_{2}} \sum_{l=1}^{m_{1}}	 \left( 
	\prod_{k=1}^{2m_{1}} \tr{A_{N,k} B_{N,2l-k}} \right) 
		+
	O \left( N^{-\frac{1}{2	}}\right)
	\end{align}
\end{theorem}
%
%
%

%

Now, the main result in \cite{anderson2014asymptotically} gives sufficient conditions on a unitary random matrix ensemble to be asymptotically liberating. 
Using a different approach than that one in \cite{anderson2014asymptotically}, we have been able to prove that, under the same conditions, a unitary random matrix ensemble not only is asymptotically liberating but also satisfies a natural generalization of the boundedness condition in Definition \ref{def.asymp.liberatin} to cumulants of any order. 
More concretely, we have the following lemma. 
%
%
%
%

\begin{lemma}\label{lemma.bounded.cumulant.property}
	Let $ U_{N,i} $ be an $N$-by-$N$ unitary random matrix for each index $i \in I$ and each integer $N\geq 1$. 
	Suppose the unitary random matrix ensemble  $\mathcal{U}=\{\{U_{N,{i}}\}_{{i} \in I}\}_{N=1}^{\infty}$ satisfies the following two conditions: 
	\begin{enumerate}[(I)]
		\item\label{signed-permutation-invariance} the families of random matrices $ \{ U^{*}_{N,i_{1}}U^{}_{N,i_{2}} \}_{i_{1}, i_{2} \in I} $ and  $ \{ \mathrm{W}^{*} U^{*}_{N,i_{1}}U^{}_{N,i_{2}} \mathrm{W} \}_{i_{1}, i_{2} \in I} $ are equal in distribution for every $N$-by-$N$ signed permutation matrix $\mathrm{W}_{N}$, and

		\item\label{pnorms-uniformly-bounded} for every positive integer $m$ there is a constant $C_m$ independent from $N$ such that 
		\[
		\Big\lVert \left( U^{*}_{N,i_{1}}U^{}_{N,i_{2}} \right) (j_{1},j_{2})\Big\rVert_{m} \leq C_m N^{-1/2} 
		\]
		for all integers $ j_{1},j_{2} \in \{1,2,\ldots, N\}$ and indexes $i_{1}, i_{2} \in I$ with $i_{1} \neq i_{2} $,
	\end{enumerate}
	Now, given positive integers $m_1, m_2, \ldots, m_n$,
	take $m'_{k}=m'_{k-1}+m^{}_{k-1}$ for $k=2,3,\ldots,n$ with $m'_{1}=0$ 
	and consider the permutation $\gamma = (1,2,\ldots, m'_{1}+ m_{1})(m'_{2}+1,m'_{2}+2, \ldots, m'_{2}+ m_{2}) \cdots (m'_{n}+1, \ldots, m'_{n}+ m_{n})$. 
	If some indexes ${i}_1,{i}_2, \ldots, {i}_m \in I$ are such that ${i}_{k} \neq {i}_{\gamma(k)}$ for $k=1, 2, \ldots, m$ where  $m = m_1 + m_2 + \cdots + m_n$,  
	then there exists a constant $C({i}_1, {i}_2, \ldots, {i}_m)$ such that for 
	\[
	Y_{N,k} = 	\big( U_{N,{i}_{m'_k+1}} A_{m'_k+1}U^{*}_{N,{i}_{m'_k+1}} \big)
	\big(	U_{N,{i}_{m'_k+2}} A_{m'_k+2}U^{*}_{N,{i}_{m'_k+2}}  \big) \cdots 
	\big( U_{N,{i}_{m'_k+m_k}} A_{m'_k+m_k}U^{*}_{N,{i}_{m'_k+m_k}} \big) 
	\]
	with $A_{1},A_{2},\ldots,A_{m} \in \mathrm{M}_N(\mathbb{C})$ each of trace zero, we have
	\[
	\abs*{\mathfrak{c}_n \left[  \Tr{Y_{N,1}}, \Tr{Y_{N,2}} , \ldots,\Tr{Y_{N,n}} \right]}
	\leq 
	C({i}_1, {i}_2, \ldots, {i}_m) \norm*{A_1} \norm*{A_2} \cdots \norm*{A_m}
	\]
\end{lemma}

Thus, the fact that a unitary random matrix ensemble $ \{ \{ U_{N,i}  \}_{i\in I} \}_{N=1}^{\infty}$ satisfying  \textit{(\ref{signed-permutation-invariance})} and \textit{(\ref{pnorms-uniformly-bounded})} above is asymptotically liberating can now be seen as a particular case of the previous lemma. 
Moreover, if $U_{N,i_{1}}$ and $U_{N,i_{2}}$ are independent Haar-unitary (resp. Haar-orthogonal), then $U^{*}_{N,i_{1}}U^{}_{N,i_{2}}$ is also Haar-unitary (resp. Haar-orthogonal), and hence, $U^{*}_{N,i_{1}}U^{}_{N,i_{2}}$ satisfies \textit{(\ref{signed-permutation-invariance})} and \textit{(\ref{pnorms-uniformly-bounded})} above.   
Therefore, independent Haar-unitary (Haar-orthogonal) random matrix ensembles are asymptotically liberating. 

The approach we take to prove Theorem \ref{thm.second.order.asymptotics} and Lemma \ref{lemma.bounded.cumulant.property} relies on the examination of graph sums of square matrices, see \cite{mingo2012sharp} or Section \ref{sec.graph.sums} below, and gives the expressions in \eqref{cumm.lemma.4} and \eqref{eqn.fluct.moments.c2.chi.chi.1} as intermediate steps. 
We expect these expressions can be used to determine the higher order moments induced by Haar-unitary and Haar-orthogonal ensembles via the  Weingarten Calculus from \cite{MR1959915} and \cite{MR2217291}. 

The customary definition of asymptotic free independence for random matrix ensembles involves the convergence of a sequence of linear functionals on non-commutative polynomials, see Proposition \ref{prop.asymptotic.centering.1} and the comment right after its proof. 
In a similar way, multi-linear functionals on non-commutative polynomials can be used to analyze the behavior of moments of higher order, allowing us to show that unitary random matrix ensembles satisfying \textit{(\ref{signed-permutation-invariance})} and \textit{(\ref{pnorms-uniformly-bounded})} above induce the bounded cumulants property when used for conjugation. 
%
%
%
%

\begin{theorem} \label{thm.bounded.cumulants.property}
	Let $D_{N,i}$ be a self-adjoint $N$-by-$N$ deterministic matrix and let $ U_{N,i} $ be an $N$-by-$N$ unitary random matrix for each index $i \in I$ and each integer $N\geq 1$. 
	Suppose the unitary random matrix ensemble $ \{ \{ U_{N,i}  \}_{i\in I} \}_{N=1}^{\infty}$ satisfies \textit{(\ref{signed-permutation-invariance})} and \textit{(\ref{pnorms-uniformly-bounded})} from the previous lemma and  (\ref{assump.constnt.matrices}) holds. 
	Then the ensemble  $ \{ \{ U^{}_{N,i}D^{}_{N,i}U_{N,i}^{*} \}_{i \in I}  \}_{N=1}^{\infty}$ has the  \textbf{bounded cumulants property}, namely, for all polynomials $\mathrm{p}_{1}, \mathrm{p}_{2}, \mathrm{p}_{3},  \ldots$ in the algebra of non-commutative polynomials $\mathbb{C}\left< \mathrm{x}_{i} \mid {i \in I} \right>$ taking $Y_{N,k}= \mathrm{p}_{k} (\{U^{}_{N,i}D^{}_{N,i}U^{*}_{N,i} \}_{i \in I} )$ we have
	\begin{align}\label{eqn.bnd.cumulants.prpty}
	\sup_{N}  \Big\lvert \mathfrak{c}_{n} \left[ 
	\Tr{Y_{N,1}}, \Tr{Y_{N,2}},\ldots ,\Tr{Y_{N,n}}
	\right] \Big\rvert 
	< \infty  
	\end{align}
	for every integer  $n \geq 1$.
\end{theorem}
%
%
%
%

The term \textit{bounded cumulants property} is borrowed from \cite{mingo2016freeness} where it is used to prove several results concerning the limiting behavior of unitarily-invariant random matrix ensembles and some other random matrix ensemble with this property. 

\subsection{Organization of this paper}\hspace*{\fill} \newline 

%
%
The rest of this paper is organized as follows. 
In Section \ref{sec.preliminaries}, we introduce the main definitions and the main notation for partitions, classical cumulants, matrices, and non-commutative polynomials; we also establish the distribution of random signed permutation matrices and random signature matrices. 
In Section \ref{sec.graph.sums}, we review and prove multiple results on graph sums of square matrices, providing the central tools our proofs rely on. 
Roughly speaking, a graph sums of square matrices is a sum of products of entries of square matrices with the constraint that some of the entries from distinct matrices are indexed by the same summation variable. 
Then, Section \ref{sec.proof.main.thm.1} and Section \ref{sec.proof.main.thm.2} are devoted to the proofs of our main results, more concretely, Lemma \ref{lemma.bounded.cumulant.property} and Theorem \ref{thm.bounded.cumulants.property} are proved in  Section \ref{sec.proof.main.thm.1} whereas Theorem \ref{thm.second.order.asymptotics} and Theorem \ref{thm.second.order.freeness} are proved in Section \ref{sec.proof.main.thm.2}. 
Finally, in Section \ref{sec.concluding}, we give some concluding remarks including open questions and further research projects. 

\section{Preliminaries}\label{sec.preliminaries}

\subsection{Set partitions, the M\"{o}bius inversion function, and classical cumulants}\label{subsec.part.mob.cum}\hspace*{\fill}\newline

A \textit{partition} of non-empty set $S$ is a set of non-empty and pair-wise disjoint subsets of $S$ whose union is $S$,
i.e., a set $\pi$ is a partition of $S$ if $B \subset S$ and $B \neq \emptyset$  for every $B \in \pi$, $B\cap B' \neq \emptyset$ implies $B=B'$ for all $B,B'\in \pi$, and $\cup_{B \in \pi } B = S$. 
The elements of a partition are called \textit{blocks}, a block is said to be \textit{even} if it has even cardinality, and similarly, a block is said to be \textit{odd} if it has odd cardinality. 
A partition containing only even blocks is called \textit{even}, but if all of its blocks have exactly two elements, we refer to it as a 
\textit{pairing}.  
The total number of block in partition $\pi$ is denoted by $\#(\pi)$ and 
we let $P(S)$, $P_{\text{even}}(S)$, and $P_{2}(S)$ denote the set of all partitions of $S$, the set of all even partitions of $S$, and the set of all pairing partitions of $S$, respectively. 

\begin{example*}
	The sets $\pi_{1} = \{\{-1,-3,-2,2\},\{1,3\}\}$, $\pi_{2} = 
	\{\{-1,-2\},\{2\},\{1,-3,3\}\}$, and 	$\pi_{3} 
	= \{\{-1,-3\},\{1,3\},\{-2,2\}\}$ are all partitions of 
	$\{-1,1,2,-2,-3,3\}$. 
	The partitions $\pi_{1}$ and $\pi_{3}$ are both even, but while $\pi_{3}$ is a paring, $\pi_{1}$ is not. 
	The partition $\pi_{2}$ is neither even nor odd since it contains two odd	blocks, $\{2\}$ and $\{1,-3,3\}$, and one even block, $\{-1,2\}$. 
\end{example*} 

We let $[m]$ and $[\pm m ]$ denote the sets of integers $\{1,2,\ldots, m\}$ and $\{-1,1,-2,2,\ldots, -m, m \}$, respectively. 
The sets $[m]$ and $[\pm m ]$ are used extensively in this paper, so we will omit the square brackets when referring to any of their sets of partitions.
Thus, for instance, we write $P_{\text{even}}(\pm m)$ instead of $P_{\text{even}}([\pm m])$.

Every partition $\pi \in P(S)$ defines an \textit{equivalence relation}, denoted by $\sim_{\pi}$, that has the blocks of $\pi$ as equivalence classes. 
Thus, given elements $k,l \in S$, we write $k \sim_{\pi} l$ only if $k$ and $l$ belong to the same block of $\pi$. 
With this notation in mind, a partition $\pi \in P( \pm m )$ is called \textit{symmetric} if $k \sim_{\pi} l$ implies $-k \sim_{\pi} -l$.

The set of partitions $P(S)$ becomes a \textit{partially ordered set} with the partial order $\leq$ defined as follows: given partitions $\pi$ and $\theta$ in $P(S)$, we write $\pi \leq \theta$, and say that $\pi$ is a \textit{refinement} of $\theta$, if every block of $\pi$ is contained in some block of $\theta$. 
Note that $\pi \leq \theta$ if and only if $k \sim_{\pi} l$ implies $k \sim_{\theta} l $ for all $k,l \in S$. 
In the previous example, the partition $\pi_{3}$ is a refinement of $\pi_{1}$, and there is no other refinement between $\pi_{1}$, $\pi_{2}$, and $\pi_{3}$.  

Consider now the function  $\zeta : P(S) \times P(S) \rightarrow \{1,0\}$ defined by 
\[
\zeta(\theta,\eta)=\left\{\begin{array}{cl}
1 & \text{if } \theta \leq \eta \\
0 & \text{otherwise}.
\end{array}\right. 
\]
This function is called the \textit{zeta function} of $P(S)$. It turns out that if $S$ is a finite set, then the system of equations 
%
%
\begin{align}\label{eqn.defining.mobius.func.r}
\sum_{\substack{ \eta \leq \pi \leq \theta } } 
\zeta(\eta,\pi)
\mu(\pi,\theta) 
=	
\left\{\begin{array}{cl}
1 & \text{if } \eta = \theta \\
0 & \text{otherwise}
\end{array}\right. \text{ for } \eta,\theta \in P(S)
\end{align}
determines a function  $\mu : P(S) \times P(S) \rightarrow \mathbb{Z}$ called the \textit{M\"{o}bius function} of $P(S)$ which can be explicitly computed, but first, let us establish the convention that whenever we write $\eta=\{B_{i_1},B_{i_2},\ldots,B_{i_r}\}$ for a partition $\eta$, it is always assumed that blocks $B_{i_k}$ and $B_{i_l}$ are the same only if $i_k=i_l$.
Suppose now we are given partitions $\pi$ and $\theta$ in $P(S)$. 
If $\pi \leq \theta$, we can write $\theta = \{B_1, B_2,\ldots, B_{r}\}$ and $\pi=\{B_{1,1},B_{1,2}, \ldots, B_{1,m_1},\ldots, B_{n,m_r} \}$ with $B_{k} = \cup_{l=1}^{m_k}B_{k,l}$ for each $k$, and, in this case, we have 
%
%
\begin{align}\label{mobius.inversion.function}
\mu(\pi, \theta) =  \prod_{k=1}^{n}(-1)^{m_k-1}(m_k-1)! \ \  .
\end{align}
On the other hand, if $\pi$ is not a refinement of $\theta$, we have $\mu(\pi,\theta) = 0$.
The \textit{M\"{o}bius inversion formula} states that given arbitrary functions $f,g:P(S) \rightarrow \mathbb{C}$, we have the relation
%
%
\begin{align}\label{mobius.inversion.formula.0}
\forall \theta \in P(S) \quad
f(\theta) = \sum_{\substack{ \pi \in P(S) \\ \pi \geq \theta }} 
			g(\pi) 
 \quad \Longleftrightarrow  \quad
\forall \pi \in P(S) \quad
g(\pi) = \sum_{\substack{ \theta \in P(S) \\ \theta \geq \pi}} 
	\mu(\pi,\theta) f(\theta)
\end{align}
The computation of M\"{o}bius function, Equation (\ref{mobius.inversion.function}), and the M\"{o}bius inversion formula, Equation (\ref{mobius.inversion.formula.0}),  are well-known and their proofs can be found in \cite[Lecture 10]{nica2006lectures}. 

\subsubsection*{Classical cumulants}\label{sec.classical.cumulants}
Let $(\Omega,\mathcal{F},P)$ be a classical probability space and let $L^{-\infty}(\Omega,\mathcal{F},P)$ denote the set of complex-valued random variables on $(\Omega,\mathcal{F},P)$ with finite moments of all orders. 
The classical $n$-th cumulant on $L^{-\infty}(\Omega,\mathcal{F},P)$ is the $n$-linear functional $\mathfrak{c}_{n} : L^{-\infty}(\Omega,\mathcal{F},P) \times \cdots  \times  L^{-\infty}(\Omega,\mathcal{F},P) \rightarrow \mathbb{C} $ defined by 
%
%
\begin{align}\label{eqn.moment.cumulants.relation}
\mathfrak{c}_{n}[x_1, x_2, x_3 ,\ldots, x_n] 
=& 
\sum_{\substack{ \pi \in P(n) } }
\mu(\pi, 1_n) 
\prod_{B \in \pi} 
\mathbb{E}_{}\left[ \prod_{b \in B} x_b \right]
\end{align}
for random variables $x_1, x_2, x_3 ,\ldots, x_n \in L^{-\infty}(\Omega,\mathcal{F},P)$ and where $\mathbb{E}[ \cdot ]$ denotes the corresponding expected value. 
Note that if $x_{k}$ is a constant for some $k \in [n]$ and $n \geq 2$, then 
\(\mathfrak{c}_{n}[x_1, x_2,\ldots, x_n]  = 0 \). 

\subsection{The kernel notation, tuples, and permutations}\label{sec.kernel.tuples.permutations}\hspace*{\fill}\newline 

Let $S_{1}$ and $S_{2}$ be non-empty sets. 
We make the convention that for a function $\mathbf{j}: S_{1} \rightarrow S_{2}$, we take $j_{k} = \mathbf{j}(k)$ for every $k \in S_{1}$; 
additionally, if  $S_{1} = [\pm m]  $ for some integer $m \geq 1$, we identify the  function  $\mathbf{j}:  S_{1} \rightarrow S_{2}$ with the tuple $(j_{-1},j_{1},j_{-2},\ldots,j_{-m},j_{m})$. 
Moreover, the \textit{kernel} of a function $\mathbf{j} : S_{1} \rightarrow S_{2}$, denoted by $\kernel{\mathbf{j}}$, is defined as the partition of $S_{1}$ whose blocks are all of the non-empty pre-images of $\mathbf{j}$, i.e., 
%
%
\begin{align*}
\kernel{\mathbf{j}} 
= \{ \mathbf{j}^{-1}(s) \neq \emptyset \mid s \in S_{2}  \} 
= \{ \{k \in S \mid j_{k} = s\} \neq \emptyset \mid s \in S_{2}  \} .
\end{align*}
%
%
Additionally, if we are given permutations $\sigma_{l} \in \{ f: S_{l} \rightarrow S_{l} \mid f \text{ is bijective} \}$ for $l=1,2$,
we let $\mathbf{j} \circ \bm{\sigma}_{1}  : S_{1} \rightarrow S_{2}$ and $ \bm{\sigma}_{2} \circ \mathbf{j} : S_{1} \rightarrow S_{2}$ be given by the usual composition of functions, so we have
\begin{align*}
\mathbf{j} \circ \bm{\sigma}_{1} (k) =  j_{\sigma_{1}  (k) } 
\text{\quad and \quad} 
\bm{\sigma}_{2} \circ \mathbf{j} (k) = \sigma_{2}(j_{k}) 
\quad \quad \forall k \in S_{1} .
\end{align*}
\begin{example*} 
	The function $\mathbf{j}:[\pm 3] \rightarrow [4]$ given by
	\[
	\mathbf{j}(-1) = \mathbf{j}(2) = \mathbf{j}(3) = 4\text{,}\quad 
	\mathbf{j}(1)  = \mathbf{j}(-3) = 1\text{,\quad and\quad}
	\mathbf{j}(-2) = 3, 
	\]
	or, equivalently, $ (j_{-1},j_{1},j_{-2},j_{2},j_{-3},j_{3}) = (4,1,3,4,1,4)$, has kernel 
	\[\kernel{\mathbf{j}} = \{\{-1,2,3\},\{1,-3\},\{-2\}\}.	\]
	Additionally, if $\sigma_{1}:[\pm 3] \rightarrow [\pm 3]$ is given $\sigma_{1} (k)=-k$ for every $k \in [\pm 3]$ and $\sigma_{2}:[4] \rightarrow [4]$ is the cyclic permutation   $(1,2,3,4)$, then 
	$ \mathbf{j} \circ \bm{\sigma}_{1} 	= (1,4,4,3,4,1) $ and $	\bm{\sigma}_{2} \circ \mathbf{j} = (1,2,4,1,2,1) $.
\end{example*} 

Now, the group of permutations $\{f : S_{} \rightarrow S_{} \mid f \text{ is  bijective}\} $ acts on the set of partitions $P(S_{})$ as follows: given a permutation $\sigma \in \{ f: S_{} \rightarrow S_{} \mid f \text{ is bijective} \}$ we let $\sigma \circ \pi$ be given by 
%
%
\begin{align*}
\sigma \circ \pi
& = \{ \sigma(B)  \mid B \in \pi  \}  = \{ \{ \sigma(k) \mid k \in B \} \mid B \in \pi  \} 
\end{align*}
The map $\pi \mapsto \sigma \circ \pi$ is a poset automorphism, in particular, it is order-preserving, so for all partitions $\pi, \theta \in P(S_{})$ we get
\begin{align*}
\pi \leq \theta 
\quad \Longleftrightarrow \quad 
\sigma \circ \pi \leq \sigma \circ \theta.
\end{align*}

\begin{remark*}
	Note that a partition $\pi \in P(S_{1})$ and a function  $\mathbf{j}: S_{1} \rightarrow S_{2}$ satisfy $\pi \leq \kernel{\mathbf{j}}$ if only if the function $\mathbf{j}$ is constant when restricted to each of the blocks of $\pi$, i.e., $j_{k} = j_{l}$ whenever $k,l \in B$ for some block $B \in \pi$. 
	Moreover, for permutations $\sigma_{1}: S_{1} \rightarrow S_{1}$ and  $\sigma_{2}: S_{2} \rightarrow S_{2}$, we have that $\kernel{ \mathbf{j} \circ \bm{\sigma}_{1} } = \sigma^{-1}_{1} \circ \kernel{\mathbf{j}}$ and  $\kernel{\mathbf{j}} = \kernel{\bm{\sigma}_{2} \circ \mathbf{j}} $.  
\end{remark*}

\subsection{Some random matrices and the joint distribution of their entries}\hspace*{\fill}\newline

Let $I$ be a non-empty set. 
Suppose $\{X_{i}\}_{i \in I}$ and $\{Y_{i}\}_{i \in I}$ are two families of $N$-by-$N$ random matrices defined on the same probability space. 
We say that $\{X_{i}\}_{i \in I}$ and $\{Y_{i}\}_{i \in I}$ are \textit{equal in distribution} if we have 
\begin{align*}
\exptr{ \prod_{k=1}^{m} X_{i_{k}}(j_{-k},j_{k}) }
=
\exptr{ \prod_{k=1}^{m} Y_{i_{k}}(j_{-k},j_{k}) }
\end{align*}
for all integers $m\geq 1$, indexes $i_{1},i_{2}, \ldots, i_{m} \in I$, and functions $\mathbf{j}: [\pm m] \rightarrow [N]$.  

A matrix $\mathrm{X} \in \mathrm{Mat}_{N}(\mathbb{C})$ is a \textit{signature matrix} if there exists signs $\epsilon_{1},\ldots,\epsilon_{N} \in \{-1,1\}$ such that
\[
\mathrm{X}(i,j)=\left\{ \begin{array}[c]{cl}
							\epsilon_{i} & \text{ if } i=j \\
							0			 & \text{ otherwise }.
				\end{array}	\right. 
\]
An $N$-by-$N$  random matrix $X$ is a \textit{uniformly distributed signature matrix} if  it is uniformly distributed on the set of $N$-by-$N$ signature matrices; in this case, for all functions $\mathbf{i,j}: S \rightarrow [N]$ we have 
%
%
%
\begin{align}\label{dist.entries.signature.matrix}
\exptr{\prod_{k \in S} X (i_{k} , j_{k}) }
= 
\left\{ 
\begin{array}{cl}												
1, & \text{ if } \mathbf{i} = \mathbf{j} \text{ and } \kernel{\mathbf{i}}_{} \text{ is an even partition} \\
0, &\text{ otherwise} \\
\end{array}
\right.
\end{align}

A matrix $ \mathrm{W} \in \mathrm{Mat}_{N}(\mathbb{C})$ is a \textit{signed permutation matrix} if there exists signs $\epsilon_{1},\ldots,\epsilon_{N} \in \{-1,1\}$ and a permutation $\sigma \in \{f:[N]\rightarrow [N]:f \text{ is bijective}\}$ such that
\[
\mathrm{W}(i,j)	=	\epsilon_{i}\delta_{i,\sigma(j)} 
				= 	\left\{ \begin{array}[c]{cl}
							\epsilon_{i} & \text{ if } i=\sigma(j) \\
							0			 & \text{ otherwise }
					\end{array}	\right. \text{ .}
\]
An $N$-by-$N$  random matrix $W$ is a \textit{uniformly distributed signed permutation matrix} if  it is uniformly distributed on the set of $N$-by-$N$ signed permutation matrices; if that is the case, for all functions $\mathbf{i,j}: S \rightarrow [N]$ we get  
%
%
%
\begin{align}\label{dist.entries.signed.perm.matrix}
\exptr{\prod_{s \in S} W(i_s , j_s) }
= 
\left\{ \begin{array}{cl}														
		\frac{(N-\#(\pi))!}{N!} & 
							\text{, if } \pi = \kernel{\mathbf{i}} = \kernel{\mathbf{j}} \in P_\text{even}(S) \\
		0						&
							\text{, otherwise} 
\end{array}
\right.
\end{align}

\begin{remark*} 
Suppose $\{V_{i}\}_{i \in I}$ is a family of $N$-by-$N$ random matrices distribution-invariant under conjugation by signed permutation matrices, i.e., the families $\{V_{i}\}_{i \in I}$ and $\{W^{*} V_{i} W \}_{i \in I}$ are equal in distribution for every signed permutation matrix $W$. 
Then, for all integers $m\geq 1$, indexes $i_{1},i_{2}, \ldots, i_{m} \in I$, and functions $\mathbf{j}: [\pm m] \rightarrow [N]$, we have 
	\begin{align}
	\exptr{\prod_{k=1}^{m} V_{i_{k}}(j_{-k},j_{k})} = 
	\prod_{k=1}^{m} \epsilon_{\sigma(j_{-k})} \epsilon_{\sigma(j_{k})} \exptr{\prod_{k=1}^{m} V_{i_{k}}(\sigma(j_{-k}),\sigma(j_{k}))}
	\end{align}
for all signs $\epsilon_{1},\ldots,\epsilon_{N} \in \{-1,1\}$ and permutations $\sigma \in \{f:[N]\rightarrow [N]:f \text{ is bijective}\}$ . 
\end{remark*}

\subsection{Non-commutative polynomials and their evaluation on families of random matrices}\hspace*{\fill}\newline 

Let $I$ be a non-empty set. 
We denote by $\mathbb{C}\left\langle \mathrm{x}_{i} \mid i \in I \right\rangle$ the algebra of non-commutative polynomials on the family of variables $\{\mathrm{x}_{i} \mid i \in I\}$. 
Let us recall that $\mathbb{C}\left\langle \mathrm{x}_{i} \mid i \in I \right\rangle$ is the algebra over $\mathbb{C}$ with a basis consisting of all the words in the alphabet $\{\mathrm{x}_{i} \mid i \in I\}$, including the empty word which acts as multiplicative identity, and the product of two basis elements is given by concatenation. 
Thus, a basis element is a word of the form
\[\mathrm{x}_{i_{1}}\mathrm{x}_{i_{2}}\cdots\mathrm{x}_{i_{r}}\]
for some integer $r\geq 0$ and some indexes $i_{1},i_{2},\ldots, i_{r} \in I$, and if $\mathrm{x}_{j_{1}}\mathrm{x}_{j_{2}}\cdots\mathrm{x}_{j_{r}}$ is another basis element, we have 
\[(\mathrm{x}_{i_{1}}\mathrm{x}_{i_{2}}\cdots\mathrm{x}_{i_{r}})(\mathrm{x}_{j_{1}}\mathrm{x}_{j_{2}}\cdots\mathrm{x}_{j_{s}}) = \mathrm{x}_{i_{1}}\mathrm{x}_{i_{2}}\cdots\mathrm{x}_{i_{r}}\mathrm{x}_{j_{1}}\mathrm{x}_{j_{2}}\cdots\mathrm{x}_{j_{s}} .\]
Given polynomials $\mathrm{p}_{1},\mathrm{p}_{2}, \ldots, \mathrm{p}_{m}$ in the algebra $\mathbb{C}\left\langle \mathrm{x}_{i} \mid i \in I \right\rangle$ and a set  $S=\{ k_{1} < k_{2} < \cdots < k_{n} \} \subset [m]$, we let 
\begin{align}
\vec{\prod_{k \in S }} \mathrm{p}_{k}: = \mathrm{p}_{k_{1}} \mathrm{p}_{k_{2}} \cdots \mathrm{p}_{k_{n}} . 
\end{align}
Suppose we are given random matrix ensembles $\{X_{N,i}\}_{N=1}^{\infty}$ with $i\in I$ where each $X_{N,i}$ is a $N$-by-$N$ random matrix.
For each non-commutative polynomial $\mathrm{p} \in \mathbb{C}\left\langle \mathrm{x}_{i} \mid i \in I \right\rangle $, we denote by 
\[\mathrm{p} \left( \{X_{N,i} \}_{i \in I} \right) \]
the random matrix obtained from replacing each $\mathrm{x}_{i}$ appearing in the polynomial $\mathrm{p}$ with the random matrix $X_{N,i}$ for every $i\in I$ and the constant term of $\mathrm{p}$, say $\alpha$, with the scalar multiple of the identity matrix $\alpha I_{N}$. 
For instance, if $\mathrm{p}(\mathrm{x}_{1},\mathrm{x}_{2}) = \mathrm{x}_{1} \mathrm{x}_{2} -  \mathrm{x}^{2}_{2} + 4 $, then 
\[\mathrm{p}(\{X_{N,i}\}_{i \in \{1,2\}}) = X_{N,1} X_{N,2} -  X^{2}_{N,2} + 4 I _{N}  .\]

\section{Graph Sums of Square Matrices}\label{sec.graph.sums}

In this section we review and prove some useful results on graph sums of square matrices. 
A \textit{graph sum} of given matrices $A_{1},A_{2},\ldots,A_{m} \in \mathrm{Mat}_N(\mathbb{C})$ is a sum of the form
%
%
\begin{align}\label{def.graph.sums}
\sum_{\substack{ \mathbf{j} :[\pm m]\rightarrow [N] \\  \kernel{\mathbf{j}} \geq \pi } }A_{1}( j_{-1},j_{1}) A_{1}( j_{-2},j_{2}) \cdots A_{m}( j_{-m},j_{m})
\quad =
\sum_{\substack{ 
		\mathbf{j} :[\pm m]\rightarrow [N] \\  
		\kernel{\mathbf{j}} \geq \pi } }
\prod_{k=1}^{m} A_{k}(j_{-k},j_{k})
\end{align}
for some partition $\pi \in P(\pm m)$. 
Note that the condition $\kernel{\mathbf{j}} \geq \pi $ in the sum above is simply a restatement of a set of equalities between the indexes $j_{\pm k}$. 
For example, if we let $\pi = \{\{1,-2\},\{2,-3\},\ldots,\allowbreak \{m-1,-m\},\{m,-1\}\}$, then $\kernel{\mathbf{j}} \geq \pi $ only if  $j_{1} = j_{-2},j_{2}=j_{-3},\ldots, j_{m-1}=j_{-m}$, and $j_{m}=j_{-1}$, and thus we get
\begin{align*}
\sum_{\substack{ 
		\mathbf{j} :[\pm m]\rightarrow [N] \\  
		\kernel{\mathbf{j}} \geq \pi } }
\prod_{k=1}^{m} A_{k}(j_{-k},j_{k})
=
\Tr{A_{1} A_{2} \cdots A_{m}}. 
\end{align*}
It is worth mentioning that although the labeling of the entries of $A_{k}$ in (\ref{def.graph.sums}) is not customary, it has proven to be suitable for many of our calculations;
moreover, for a bijection $\sigma :  [\pm m ] \rightarrow S $, the relation 
\begin{align}\label{graph.sums.re-labeling.eqn}
\sum_{\substack{ 
		\mathbf{j} : S \rightarrow [N] \\  
		\kernel{\mathbf{j}} \geq \hat\pi } }
\prod_{k=1}^{m} A_{k}(j_{\sigma(-k)},j_{\sigma(k)})
=
\sum_{\substack{ 
		\mathbf{j} :[\pm m]\rightarrow [N] \\  
		\kernel{\mathbf{j}} \geq \sigma^{-1} \circ \hat\pi } }
\prod_{k=1}^{m} A_{k}(j_{-k},j_{k})
\qquad \forall \hat\pi \in P(S)
\end{align}
provides the link between the labeling of the entries of $A_{k}$ in (\ref{def.graph.sums}) and any other labeling. 
For instance, if  $\sigma :  [\pm m ] \rightarrow [2m] $ is given by $\sigma(-k)=2k-1$ and $\sigma(k)=2k$ for $1 \leq k \leq m$, then 
\begin{align*}
\sum_{\substack{ 
		\mathbf{j} :[2m]\rightarrow [N] \\  
		\kernel{\mathbf{j}} \geq \hat\pi } }
\prod_{k=1}^{m} A_{k}(j_{2k-1},j_{2k})
=
\Tr{A_{1}} \Tr{A_{2}} \cdots \Tr{A_{m}} 
=
\sum_{\substack{ 
		\mathbf{j} :[\pm m]\rightarrow [N] \\  
		\kernel{\mathbf{j}} \geq \pi } }
\prod_{k=1}^{m} A_{k}(j_{-k},j_{k})
\end{align*}
where $\hat\pi = \{\{1,2\},\{3,4\},\ldots,\{2m-1,2m\}\}$ and $\pi = \sigma^{-1}\circ \hat\pi = \{\{-1,1\},\{-2,2\},\ldots,\{-m,m\}\}$. 
The type of sums above are named graph sums because they can be associated to certain graphs that, as we will see next, help us analyze the corresponding sums.

\subsection{Bounds of graph sums of general square matrices}\label{sec.general.graph.sums}
\hspace*{\fill}\newline 

The main result in \cite{mingo2012sharp} concerns more general graph sums, allowing the matrices $A_{k}$ in (\ref{def.graph.sums})	 to be rectangular and not necessarily square. For graph sums of square matrices, however, the result takes the following form. %
%
\begin{theorem}\label{mingo-speicher}
	Suppose $\pi$ is a partition in $P(\pm m)$.
	Then there exists a rational number $\tau_{\pi} \in \{ 1, \frac{3}{2}, 2 
	,\ldots \}$ depending only on the partition $\pi$ such that for 
	every integer $N\geq 1$  the following two conditions hold:
	\begin{enumerate}
		\item[(a)]for all matrices $ A_{1}, A_{2}, \ldots,A_{m}\in \mathrm{Mat}_N(\mathbb{C}) $ we have
		\[
		\abs*{ \sum_{\substack{ 
					\mathbf{j} :[\pm m]\rightarrow [N] \\  \kernel{\mathbf{j}} \geq \pi } }
				\prod_{k=1}^{m} A_{k}( j_{-k},j_{k}) 
		}
		\leq
		N^{\tau_{\pi}} \prod_{k=1}^{m} \norm*{A_{k}}
		\] 
		\item[(b)]there are some non-zero matrices $B_{1}, B_{2}, \ldots,B_{m} \in \mathrm{Mat}_N(\mathbb{C}) $ satisfying
		\[ \Bigg|
			 \sum_{\substack{ 
					\mathbf{j} :[\pm m]\rightarrow [N] \\  \kernel{\mathbf{j}} \geq 
					\pi } }
			\prod_{k=1}^{m} B_{k}( j_{-k},j_{k}) 
		\Bigg|
		=
		N^{\tau_{\pi}} \prod_{k=1}^{m} \norm*{B_{k}}
		\]
	\end{enumerate}
	Note that $\tau_{\pi}$ is uniquely determined by (a) and (b). 
	We call $\tau_{\pi}$ the \textbf{graph sum exponent} of $\pi$.
\end{theorem}

It is also shown in \cite{mingo2012sharp} that the graph sum exponent $\tau_{\pi}$ can be algorithmically computed analyzing the two-edge connectedness of a  graph associated to $\pi$. 
For the reader's convenience, we recount such algorithm next. 
%
%
%
\begin{enumerate}[\textit{Step }1.]
	\item Given a partition $\pi\in P(\pm m)$, consider the undirected graph  $\mathcal{G}_{\pi}$ resulting  from, first, taking edges $E_{1}, E_{2},\ldots,E_{m}$ with endpoints $-1,+1,-2,+2,\ldots,-m,+m$, respectively, and, then, identifying endpoints when they belong to the same block of $\pi$. 
	\item Identify the cutting-edges and the two-edge connected components of $\mathcal{G}_{\pi}$. Recall that a \textit{cutting-edge} of a graph, also known as a \textit{bridge}, is an edge whose removal increases the number of connected components. 
	Moreover, a graph is \textit{two-edge connected} if it is connected and has no cutting-edges, and, consequently, a \textit{two-edge connected component} of a graph is a sub-graph that is maximal, under the usual graph inclusion, in the set of all two-edge connected sub-graphs 
	\item Letting $\mathcal{F}_{\pi}$ denote the graph with vertex set given by the set of all two-edge connected components of $\mathcal{G}_{\pi}$ and edge set given by the set of all cutting-edges of $\mathcal{G}_{\pi}$, the graph sum exponent $\tau_{\pi}$ is given by
	\begin{align}\label{algrthm.sum.exponent}
	\tau_{\pi} = \sum_{v \text{ vertex in } \mathcal{F}_{\pi} } \mathfrak{l}(v)
	\quad \text{where} \quad
	\mathfrak{l}(v) := \left\{\begin{array}{cl}
	\frac{1}{2} & \text{ if } \deg(v) = 1, \\
	1			& \text{ if } \deg(v) = 0, \\
	0			& \text{otherwise}.
	\end{array}\right.
	\end{align}
	and $\deg(v)$ denotes the degree of the vertex $v$ in the graph $\mathcal{F}_{\pi}$. 
\end{enumerate}

\begin{example*} The undirected graph $\mathcal{G}_{\pi}$ associated to the partition 
	\[
	\pi =\left\{
	\begin{array}{c}
	\{-3\},\{+3,+1,-2\},\{-5,-1,-7,-4\},\{+7\},
	\{+2,+4\}, \{+6\}, \\ \{-6,+5,+8\},\{-8\}, 
	\{-10,+12\},\{+10,-12\},\{-11,+11,-9\},\{+9\}
	\end{array}	
	\right\} \in P(\pm 12)
	\]
	can be represented as
	\begin{center}\includegraphics[scale=1]{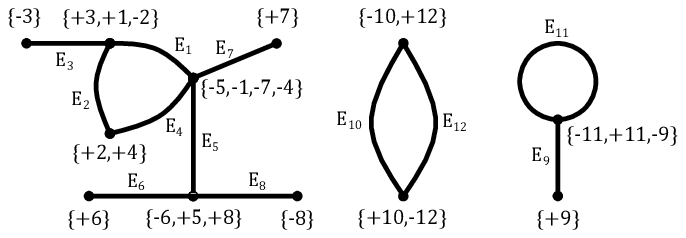}\end{center}
	Hence, the cutting-edges of $\mathcal{G}_{\pi}$ are $E_{3}$,  $E_{5}$, $E_{6}$, $E_{7}$,  $E_{8}$, and $E_{9}$; moreover, the two-edge connected components of $\mathcal{G}_{\pi}$ are exactly what remains of $\mathcal{G}_{\pi}$ after removing all of its cutting-edges. 
	The graph $\mathcal{F}_{\pi}$ can be obtained from $\mathcal{G}_{\pi}$  by shrinking each of the two-edge connected components of $\mathcal{G}_{\pi}$ to a vertex, and thus, if we represent the cutting-edges of $\mathcal{G}_{\pi}$ with dashed lines, we obtain 
	\begin{center}\includegraphics[scale=1]{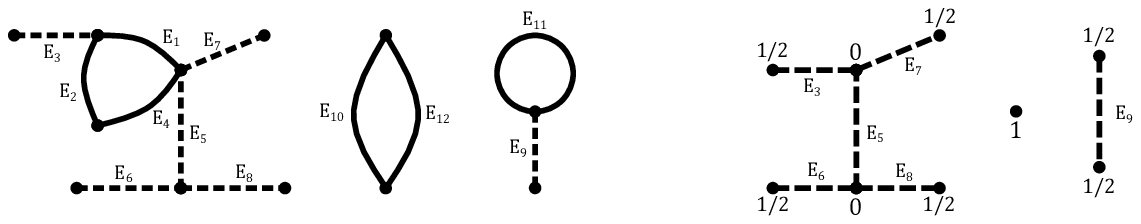}\end{center}
	where $\mathcal{F}_{\pi}$ is the graph on the right and next to each of its vertexes we have placed the corresponding contribution $\mathfrak{l}(v)$ to the graph sum exponent $\tau_{\pi}$.  
	Therefore, we have $\tau_{\pi} = 4$. 
\end{example*}

Having described the algorithm to compute  $\tau_{\pi}$, we can now show that graph sum exponents of even partitions can be easily calculated.
%
\begin{proposition}\label{prop.graph.sum.exp.bound1}
	If $\pi \in P( \pm m) $ is an even partition, 
	then the graph sum exponent $\tau_{\pi}$ equals the number of connected 
	components of $\mathcal{G}_{\pi}$.
\end{proposition}
%
%
\begin{proof}
	By Equation (\ref{algrthm.sum.exponent}), it suffices to show that the graph $\mathcal{G}_{\pi}$ has no cutting-edges. 
	Suppose $\mathcal{G}_{\pi}$ has a cutting-edge. 
	If we remove such cutting-edge, we get two disjoint graphs, each of which has one single vertex of odd degree and the other vertexes of even degree. 
	But, this contradicts the handshaking lemma that in any graph the sum of degrees over all its vertices must be even. 
	Thus,  $\mathcal{G}_{\pi}$ has no cutting-edges, and hence all its connected components are two-edge connected. 
\end{proof}
%
%
%
%
%
%

Now, resulting from endowing each edge $E_{k}$ in the graph $\mathcal{G}_{\pi}$ with the direction that goes from $+k$ to $-k$, the directed graph $\vec{\mathcal{G}}_{\pi}$ can sometimes be used to describe the corresponding graph sum. 
In particular, a graph sum factors as a product of traces of matrices when all connected components of $\mathcal{G}_{\pi}$ are bouquets, to which we refer as multiple-loops, or cycles, each connected component gives rise to a trace. 
For example, for the partition 
\[\pi=\{\{1,-6\},\{6,5\},\{-5,7\},\{-7,-1,\},\{-2,3\},\{-3,2\},\{-4,4\} \}\]
and given matrices  $A_{1},A_{2},\ldots,A_{7} \in \mathrm{Mat}_N(\mathbb{C})$, we have the graph sum 
\begin{align}\label{exmp.graph.sum.cycles.1}
\sum_{\substack{ \mathbf{j} :[\pm 7]\rightarrow [N] \\
		  		\kernel{\mathbf{j}} \geq \pi } } 
	  	\prod_{k=1}^{7} A_{k}(j_{-k},j_{k}) 
=
\Tr{ A^{}_{1} A^{}_{6} A^{T}_{5} A^{T}_{7} } \Tr{ A^{}_{2} A^{}_{3} } \Tr{ A^{}_{4} \circ A_{8} } 
\end{align}
where the right hand side can be deduced from analyzing the directed graph $\vec{\mathcal{G}}_{\pi}$ as follows:
\begin{enumerate}
  	\item The corresponding directed graph $\vec{\mathcal{G}}_{\pi}$ has exactly three connected components, two cycles and one double-loop, and can be represented as 
  	\begin{center}\includegraphics[scale=1]{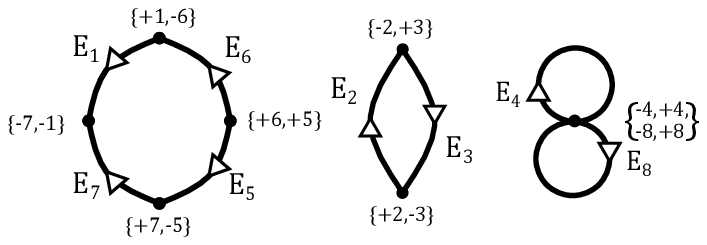}\end{center}
  	Each cycle and each one multiple-loop gives rise to a trace in the right hand side of (\ref{exmp.graph.sum.cycles.1}). 
  	\item If a connected component of $\vec{\mathcal{G}}_{\pi}$ is a cycle, we unfold it to obtain a horizontal line and replace each edge $E_{k}$ by the matrix $A_{k}$ if the direction of $E_{k}$ goes from right to left in the horizontal line, otherwise, we replace $E_{k}$ by $A^{T}_{k}$, the transpose of $A_{k}$. 
  	We then put the matrices $A_k$ or $A_{k}^{T}$ in a trace $\Tr{\cdot}$ as they appear when we read the resulting horizontal line from left to right.   
  	For instance, the longest cycle of $\vec{\mathcal{G}}_{\pi}$ gives 
  	\begin{center}\includegraphics[scale=1]{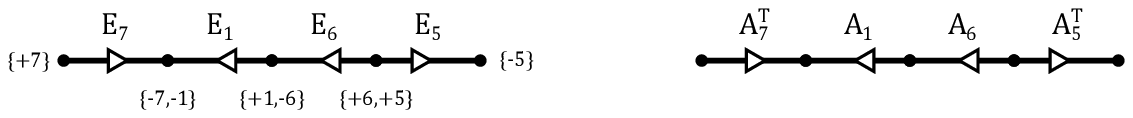}\end{center}
  	And so, we obtain the trace $\Tr{	A_{7}^{T}	A_{1}^{}	A_{6}^{}	A_{5}^{T}	}$ in (\ref{exmp.graph.sum.cycles.1}). 
  	Note that $\Tr{ A_{1}^{}	A_{6}^{}	A_{5}^{T} 	A_{7}^{T}	}$  and $\Tr{ A_{2}^{}	A_{2}^{} }$  do not depend on how the cycles in  $\vec{\mathcal{G}}_{\pi}$ are unfolded since for any matrices $A,B \in  \mathrm{Mat}_N(\mathbb{C})$ we have $\Tr{AB} = \Tr{BA}$, $\Tr{A} = \Tr{A^{T}}$, and $(AB)^{T} = B^{T}A^{T}$. 
  	\item On the other hand, a multiple-loop in $\vec{\mathcal{G}}_{\pi}$ with edges $E_{k_{1}},E_{k_{2}},\ldots,E_{k_{n}}$ yields to the trace of the Hadamard product of $A_{k_{1}},A_{k_{2}},\ldots,A_{k_{n}}$. 
  	This way, we get  $\Tr {A_{4} \circ A_{8}}$ in (\ref{exmp.graph.sum.cycles.1}). 
\end{enumerate}
Thus, if $\pi$ is now given by 
\[ \pi = \{\{-1,+6\},\{+1,-6\},\{-2,-7\},\{+2,+7\},\{-3,+3,-5,+5\},\{-4,+4\}\}, \]
the corresponding directed graph $\vec{\mathcal{G}}_{\pi}$ can be represented as 
\begin{center}\includegraphics[scale=1]{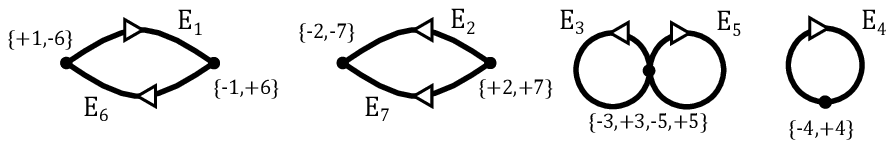}\end{center}
and hence, we obtain 
\[
\sum_{\substack{ \mathbf{j} :[\pm 7]\rightarrow [N] \\  		\kernel{\mathbf{j}} \geq \pi } } \prod_{k=1}^{7} A_{k}(j_{-k},j_{k}) = 
\Tr{A_{1} A_{6}}  \Tr{A^{}_{2} A^{T}_{7}}  \Tr{A^{}_{3} \circ A^{}_{5}} \Tr{A^{}_{4} } .
\]

\subsection{Bounds of graph sums of The Discrete Fourier Transform matrix}\label{subsec.graph.sums.dft}\hspace*{\fill}\newline 

Although the bound for graph sums given by Theorem \ref{mingo-speicher} is optimal in the set of all square matrices, it might not optimal for some graph sums involving the Discrete Fourier Transform matrix. 
Let us recall that the $N$-by-$N$ Discrete Fourier Transform matrix is the symmetric matrix $H$ with entries given by 
\begin{align}\label{def.DFT.matrix}
H(j_{1},j_{2}) =  \omega^{(j_{1}-1)(j_{2}-1)} 
\end{align}
where $\omega = \exp (-\frac{2 \pi}{N} \sqrt{-1} ) $ is a primitive $N$-th root of unity. 
Now, letting $\mathbf{h}(\mathbf{j})$ be given by 
\begin{align}\label{dfn.function.h}
\mathbf{h}(\mathbf{j}) 
	&=
		\prod_{k = 1 }^{m} 
			H(j_{-2k+1},j_{2k-1}) 	  H^{*}(j_{-2k},j_{2k}) 
\end{align}
for each function $\mathbf{j} : [\pm 2m] \rightarrow [N]$, Theorem \ref{mingo-speicher} gives us that
\begin{align}\label{eqn.DFT.graph.exponent}
\abs*{ \sum_{\substack{ 
			\mathbf{j} :[\pm 2m]\rightarrow [N] \\  
			\kernel{\mathbf{j}} \geq \pi } }
	\mathbf{h}(\mathbf{j}) 
}
\leq 
N^{\tau_{\pi}}  \prod_{k=1}^{m} \norm{H}  \prod_{k=1}^{m} \norm{H^{*}} =
N^{m+\tau_{\pi}} 
\end{align}
for any partition $\pi \in P(\pm 2m)$; on the other hand, since $	\mathbf{h}(\mathbf{j})$ has absolute value $1$, we also obtain
\begin{align}\label{eqn.DFT.number.blocks}
\abs*{ \sum_{\substack{ 
			\mathbf{j} :[\pm 2m]\rightarrow [N] \\  
			\kernel{\mathbf{j}} \geq \pi } }
	\mathbf{h}(\mathbf{j}) 
}
\quad 
\leq  \sum_{\substack{ 
			\mathbf{j} :[\pm 2m]\rightarrow [N] \\  
			\kernel{\mathbf{j}} \geq \pi } }
			1
\quad  = \quad 
		N^{\#(\pi)} .
\end{align} 
Thus, if $\pi$ is the partition $\{\{2k-1,-2k+1,2k,-2k\} \mid k = 1,2,\ldots, m\}$, then the graph sum exponent $\tau_{\pi}$ equals $m$, and hence $\tau_{\pi} + m = 2m$, but also $\#(\pi) = m$, so (\ref{eqn.DFT.number.blocks}) is a sharper bound than (\ref{eqn.DFT.graph.exponent}) in this case. 
In general, we prefer (\ref{eqn.DFT.number.blocks}) over (\ref{eqn.DFT.graph.exponent}) since (\ref{eqn.DFT.number.blocks}) is invariant under re-labeling of the entries of $H$ and $H^{*}$ in (\ref{dfn.function.h}), namely, if $\sigma:[\pm 2m] \rightarrow [\pm 2m]$ is bijective and we take
\begin{align*}
\mathbf{h}(\mathbf{j}\circ\bm{\sigma})
= 
\prod_{k = 1 }^{m} 
H(j_{\sigma(-2k+1)},j_{\sigma(2k-1)}) 	  H^{*}(j_{\sigma(-2k)},j_{\sigma(2k)})
\end{align*} 
for any function $\mathbf{j} : [\pm 2m] \rightarrow [N]$, then the inequality in (\ref{eqn.DFT.number.blocks}) implies
\begin{align*}
\abs*{
	\sum_{\substack{	\mathbf{j} :[\pm 2 m]\rightarrow [N] \\  
			\kernel{\mathbf{j}} \geq \pi } }
	\mathbf{h}(\mathbf{j}\circ\bm{\sigma})
} 
\leq  
N^{\#(\sigma^{-1} \circ \pi)}  = N^{\#(\pi)} 
\end{align*}
since we have the relation  	 
\begin{align}\label{eqn.shifted.DFT.graph.sums.alternative}
\sum_{\substack{	\mathbf{j} :[\pm 2 m]\rightarrow [N] \\  
		\kernel{\mathbf{j}} \geq \pi } }
\mathbf{h}(\mathbf{j}\circ\bm{\sigma})
\quad =
\sum_{\substack{	\mathbf{j} :[\pm 2 m]\rightarrow [N] \\  
		\kernel{\mathbf{j}} \geq \sigma^{-1} \circ \pi 
} }
\mathbf{h}(\mathbf{j}). 
\end{align}
Moreover,  in the proof of Theorem \ref{thm.second.order.asymptotics}, we will need to consider sums of the form 
\begin{align}\label{eqn.shifted.DFT.injective.graph.sums}
\sum_{\substack{	\mathbf{j} :[\pm 2 m]\rightarrow [N] \\  
		\kernel{\mathbf{j}} = \pi } }
\mathbf{h}(\mathbf{j}\circ\bm{\sigma})
\quad =
\sum_{\substack{	\mathbf{j} :[\pm 2 m]\rightarrow [N] \\  
		\kernel{\mathbf{j}} = \sigma^{-1} \circ \pi 
} }
\mathbf{h}(\mathbf{j})
\end{align}
where $m=m_{1}+m_{2}$ for some integers $m_{1},m_{2} \geq 1$ and $\sigma:[\pm 2m] \rightarrow [\pm 2m]$ is the permutation with cycle decomposition given by 
\begin{align}\label{def.permutation.sigma}
\sigma = (-1,1,-2,2,\ldots,-2m_{1},2m_{1})(-2m_{1}-1,2m_{1}+1,\ldots,-2m_{1}-2m_{2},2m_{1}+2m_{2}).   
\end{align}
Although the sum in (\ref{eqn.shifted.DFT.injective.graph.sums}) is not a graph sum, it can be determined up to a term of order $N^{\#{\pi}-1}$ analyzing (\ref{eqn.shifted.DFT.graph.sums.alternative}) since for every partition $ \pi \in P(\pm 2m)$ we have 
\begin{align}\label{eqn.graph.sums.h.injective.vs.non-injective}
\sum_{\substack{
		\mathbf{j}:[\pm {2m} ] \rightarrow [N] \\ 
		\kernel{\mathbf{j}}=\pi}} 
\mathbf{h}(\mathbf{j}) 
\quad = 
\sum_{\substack{
		\mathbf{j}:[\pm {2m} ] \rightarrow [N] \\ 
		\kernel{\mathbf{j}} \geq \pi}} 
\mathbf{h}(\mathbf{j})
\quad -
\sum_{\substack{
		\theta \in P( \pm {2m} ) \\ 
		\theta > \pi }}
\sum_{\substack{
		\mathbf{j}:[\pm {2m} ] \rightarrow [N] \\ 
		\kernel{\mathbf{j}}=\theta}} 
\mathbf{h}(\mathbf{j}). 
\end{align}
The rest of this section is devoted to find and classify partitions $\pi$ such that (\ref{eqn.DFT.number.blocks}) becomes an equality. 
To do that, let us first associate a polynomial to each partition $\pi \in P (\pm 2m)$. 

\subsubsection*{\textbf{The polynomial} $\bm{\mathrm{p}_{\pi}}$}
%
%
%
Given a partition $\pi=\{B_{1},B_{2},\ldots,B_{r}\} \in P(\pm 2m)$, we let   $\mathrm{p}_{\pi}(\mathrm{x}_{1},\mathrm{x}_{2},\ldots,\mathrm{x}_{r})$, or simply $\mathrm{p}_{\pi}$,  be the polynomial obtained from the expression
\begin{align}\label{eqn.initial.poly}
-\mathrm{x}_{-1}\mathrm{x}_{1} + \mathrm{x}_{-2}\mathrm{x}_{2} - \mathrm{x}_{-3}\mathrm{x}_{3} + \cdots + \mathrm{x}_{-2m}\mathrm{x}_{2m} 
\end{align}
after replacing each variable $\mathrm{x}_{k}$ by $\mathrm{x}_{l}$ whenever $k$ belongs to the block $B_{l}$.  
For instance, if $\pi=\{B_{1}=\{-1,3\},B_{2}=\{-3,1\},B_{3}=\{-2,2\},B_{4}=\{-4,4\}\}$, then 
\[\mathrm{p}_{\pi}(\mathrm{x}_{1},\mathrm{x}_{2},\mathrm{x}_{3},\mathrm{x}_{4})
=
-\mathrm{x}_{1}\mathrm{x}_{2} + \mathrm{x}_{3}\mathrm{x}_{3} - \mathrm{x}_{2}\mathrm{x}_{1} + \mathrm{x}_{4}\mathrm{x}_{4} . \]
Equivalently, the polynomial  $\mathrm{p}_{\pi}(\mathrm{x}_{1},\mathrm{x}_{2},\ldots,\mathrm{x}_{r})$  is the image of (\ref{eqn.initial.poly}) under the unique homomorphism from 
$\mathbb{Z}\left[\mathrm{x}_{-1},\mathrm{x}_{1},\ldots,\mathrm{x}_{-2m},\mathrm{x}_{2m}\right]$ to $\mathbb{Z}[\mathrm{x}_{1}, \mathrm{x}_{2}, \ldots, \mathrm{x}_{r}]$
such that $\mathrm{x}_{k} \mapsto \mathrm{x}_{l} $ whenever $k \in B_{l}$. 
Note that $\mathrm{p}_{\pi}(\mathrm{x}_{1},\mathrm{x}_{2},\ldots,\mathrm{x}_{r})$  has degree either $0$ or $2$ and can also be explicitly defined as
\begin{align} \label{quotient.polynomial}
\mathrm{p}_{\pi} (\mathrm{x}_1, \mathrm{x}_2,\ldots, \mathrm{x}_r) = 
\sum_{1 \leq t \leq s \leq r} a_{t,s} \mathrm{x}_{t} \mathrm{x}_{s}
\end{align}
where
\begin{align}\label{quotient.polynomial.coefficients}
 a_{t,t} =	\sum_{\substack{ k \in [2m] \\ -k , k \in  B_t }}
(-1)^{k} \quad \text{and} \quad
a_{t,s} = \sum_{\substack{ k \in [2m] \\ -k \in B_t , k \in B_s }}
(-1)^{k} +
\sum_{\substack{ l \in [2m] \\  l \in B_t , -l \in 	B_s }}
(-1)^{l} \quad \text{ for } t\neq s ;  
\end{align}
moreover, $\mathrm{p}_{\pi}(\mathrm{x}_{1},\mathrm{x}_{2},\ldots,\mathrm{x}_{r})$ satisfies the relation 
\begin{align}\label{eqn.function.h.gauss.sum.poly}
\sum_{\substack{		\mathbf{j} :[\pm 2 m]\rightarrow [N] \\  
		\kernel{\mathbf{j}} \geq \pi } }
\mathbf{h}(\mathbf{j})	
=
\sum_{j_1,j_2,\ldots,j_r=0}^{N-1} 
e^{\frac{2 \pi \sqrt{-1}}{N} \mathrm{p}_{\pi}(j_1,j_2,\ldots,j_r) }. 
\end{align}
Therefore, (\ref{eqn.DFT.number.blocks}) becomes an equality precisely when $\mathrm{p}_{\pi}(\mathrm{x}_{1},\mathrm{x}_{2},\ldots,\mathrm{x}_{r})$ is the zero polynomial. 
On the other hand, if $\mathrm{p}_{\pi}(\mathrm{x}_{1},\mathrm{x}_{2},\ldots,\mathrm{x}_{r})$ is a non-zero polynomial, we can then find a sharper bound than (\ref{eqn.DFT.number.blocks}) via the reciprocity theorem for generalized Gauss sums, see \cite[Section 1.2]{berndt1998gauss} for a proof of this theorem.

\begin{reciprocity.law}
	Suppose $a,b,c$ are integers with $a, c\neq 0$ and $ac+b$ even. Then
	\begin{align}\label{def.gauss.sums}
	S(a,b,c) 
	: =\sum_{j=0}^{\abs*{c}-1} e^{ \pi \sqrt{-1} \frac{a j^2+bj}{c}}
	=
	\abs*{\frac{c}{a}}^{\frac{1}{2}} e^{\pi \sqrt{-1} \frac{\abs*{ac} -  b^2}{4ac}}
	\sum_{j=0}^{\abs*{a}-1} e^{ \pi \sqrt{-1} \frac{-cj^2 - bj}{a}}
	\end{align}
\end{reciprocity.law}
%
%
%
%

\begin{proposition}\label{bound.gauss.sum.1}
	If $\mathrm{p}(\mathrm{x}_{1},\mathrm{x}_{2},\ldots,\mathrm{x}_{r})$ is a non-zero polynomial of degree at most 2 in $\mathbb{Z} \left[  \mathrm{x}_1, \mathrm{x}_2, \ldots, \mathrm{x}_r \right] $,  
	then there exist a constant $C_{\mathrm{p}}$ independent of $N$ such that
	\begin{align*}
		\abs*{ 	\sum_{j_1,j_2,\ldots,j_r=0}^{N-1} 
			e^{-\frac{2 \pi \sqrt{-1}}{N} \mathrm{p}(j_1,j_2,\ldots,j_r) } }
		\leq
		C_\mathrm{p} N^{r-\frac{1}{2}}
		\text{ .}
	\end{align*}
\end{proposition}

\begin{proof}
Suppose
$\mathrm{p}(\mathrm{x}_{1},\ldots,\mathrm{x}_{r}) \in \mathbb{Z} \left[  \mathrm{x}_1,  \ldots, \mathrm{x}_r \right]$ 
is a non-zero polynomial of degree at most 2. 
Without loss of generality, we can assume that there is a non-zero linear polynomial 
$\mathrm{q}_{1}(\mathrm{x}_{1},\ldots,\mathrm{x}_{r})
= \alpha_{1} \mathrm{x}_{1} + \alpha_{2} \mathrm{x}_{2} +\cdots+ \alpha_{r} \mathrm{x}_{r}  \in \mathbb{Z} \left[  \mathrm{x}_{1}, \mathrm{x}_{2}, \ldots, \mathrm{x}_{r} \right] $
and a polynomial 
$\mathrm{q}_{2}(\mathrm{x}_{2},\ldots,\mathrm{x}_{r}) \in \mathbb{Z} \left[  \mathrm{x}_{2}, \mathrm{x}_{3}, \ldots, \mathrm{x}_{r} \right]  $ of degree at most 2 
such that
\begin{align*}
\mathrm{p}(\mathrm{x}_{1},\ldots,\mathrm{x}_{r})
=
\mathrm{x}_{1} \mathrm{q}_{1}(\mathrm{x}_{1},\ldots,\mathrm{x}_{r}) +
\mathrm{q}_{2}(\mathrm{x}_{2},\ldots,\mathrm{x}_{r}).
\end{align*}
Since we have the inequality
\begin{align*}
\abs*{ 	\sum_{j_1,j_2,\ldots,j_r=0}^{N-1} 
	e^{-\frac{2 \pi \sqrt{-1}}{N} \mathrm{p}(j_1,j_2,\ldots,j_r) }  }
\leq
\sum_{j_2,\ldots,j_r=0}^{N-1}
\abs*{\sum_{j_1=0}^{N-1} 
	e^{-\frac{2 \pi \sqrt{-1}}{N}j_1 
		\mathrm{q}_{1}(j_1,j_2,\ldots,j_r)}	},
\end{align*}
we only need to show that there is a constant $C_{\mathrm{p}}$ independent from $N$ such that
\begin{align*}
\sum_{j_2,\ldots,j_r=0}^{N-1}
\abs*{\sum_{j_1=0}^{N-1} 
	e^{-\frac{2 \pi \sqrt{-1}}{N}j_1 \mathrm{q}_{1}(j_1,j_2,\ldots,j_r)}	} 
\leq 
C_{\mathrm{p}} N^{r-\frac{1}{2}} .
\end{align*}
Suppose $\alpha_{1} \neq 0$. 
Then, we have that
\begin{align*}
\sum_{j_1=0}^{N-1} e^{-\frac{2 \pi \sqrt{-1}}{N} j_1 q(j_1,j_2,\ldots,j_m)}
=	
\sum_{j_1=0}^{N-1} e^{ \pi \sqrt{-1} \frac{-2 \alpha_1 j_1^2 -2 \sum_{k=2}^{r}\alpha_k j_k j_1}{N}}
=
S\left(-2\alpha_1,- 2 \sum_{k=2}^{r}\alpha_k j_k, N \right) 
\end{align*}
where $S(a,b,c)$ denotes the generalized Gauss quadratic sum as in (\ref{def.gauss.sums}). 
Thus, by the reciprocity theorem for generalized Gauss sums, we get 
\begin{align*}
\abs*{ S \left(-2\alpha_1,- 2 \sum_{k=2}^{r}\alpha_k j_k, N \right) }
\leq
\abs*{\frac{N}{-2a_{1}}}^{\frac{1}{2}} \abs*{-2a_{1}} 
= 
\abs*{2 \alpha_{1} N}^{\frac{1}{2}},
\end{align*}
and therefore, we obtain
\begin{align*}
\sum_{j_2,\ldots,j_r=0}^{N-1}
\abs*{\sum_{j_1=0}^{N-1} 
	e^{-\frac{2 \pi \sqrt{-1}}{N}j_1 \mathrm{q}_{1}(j_1,j_2,\ldots,j_r)}	} 
\leq 
\abs*{2 \alpha_{1}}^{\frac{1}{2}} N^{r-\frac{1}{2}}.
\end{align*}
Now, suppose $\alpha_{1} = 0$. 
Recall that 
\begin{align*}
\sum_{j_1=0}^{N-1} e^{-\frac{2 \pi \sqrt{-1}}{N} j_1 \mathrm{q}_{1}(j_1,j_2,\ldots,j_r)}	
=
\left\{\begin{array}{cl}
N, & \text{ if } \mathrm{q}_{1}(j_1,\ldots,j_r) 
= \sum_{l=2}^{r} \alpha_{l} j_{l} 
\equiv 0 \mod N \\
0, & \text{ otherwise}
\end{array}\right. 
\end{align*}
So, we have
\begin{align*}
& \sum_{j_2,\ldots,j_r=0}^{N-1}
\abs*{\sum_{j_1=0}^{N-1} 
	e^{-\frac{2 \pi \sqrt{-1}}{N}j_1 \mathrm{q}_{1}(j_1,\ldots,j_r)}	} 
= \nonumber \\ &
N \cdot \#\left\{ (j_2,\ldots,j_r) \in [0,N-1]^{r-1} :
\sum_{l=2}^{r} \alpha_{l} j_{l} \equiv 0  \mod N  \right\} .
\end{align*}
But, since the polynomial $\mathrm{q}_{1}(\mathrm{x}_{1},\ldots,\mathrm{x}_{r}) = \alpha_{1} \mathrm{x}_{1}  +\cdots+ \alpha_{r} \mathrm{x}_{r}$ is non-zero, we must have $\alpha_{k} \neq 0$ for some $k \neq 1$, and hence, the equation 
\begin{align*}
\alpha_{k} \mathrm{x} + \beta \equiv 0 \mod N
\end{align*}
has at most $\abs*{\alpha_{k}}$ solutions in the set $\{0,1,\ldots,N-1\}$ for any given integer $\beta$. 
Thus, we have
\begin{align*}
\#\left\{ (j_2,\ldots,j_r) \in [0,N-1]^{r-1} : 
\alpha_{k} j_{k} + \sum_{\substack{l=2 \\ l \neq k }}^{r} \alpha_{l} j_{l} \equiv 0  \mod N  \right\}
\leq 	
\abs*{\alpha_{k}} N^{r-2},
\end{align*}
and therefore, we get
\begin{align*}
\sum_{j_2,\ldots,j_r=0}^{N-1}
\abs*{\sum_{j_1=0}^{N-1} 
	e^{-\frac{2 \pi \sqrt{-1}}{N}j_1 \mathrm{q}_{1}(j_1,j_2,\ldots,j_r)}	} 
\leq 
\abs*{\alpha_{k}}  N^{r-1}.
\end{align*}
\end{proof}
%
%
%
%
%
As an immediate consequence from  (\ref{eqn.graph.sums.h.injective.vs.non-injective}), (\ref{eqn.function.h.gauss.sum.poly}), and Proposition \ref{bound.gauss.sum.1}, we have the following. 
%
\begin{corollary}\label{cor.bound.gauss.sum}
	If $\mathrm{p}_{\pi}(\mathrm{x}_{1},\mathrm{x}_{2},\ldots,\mathrm{x}_{r})$ is a non-zero polynomial for some partition $\pi = \{B_{1},B_{2},\ldots,B_{r}\} \in P(\pm 2m)$, then there is a constant $C_{}$ independent from $N$ so that  
	\begin{align*}
	\abs*{ \sum_{\substack{
				\mathbf{j}:[\pm {2m} ] \rightarrow [N] \\ 
				\kernel{\mathbf{j}} = \pi}} 
		\mathbf{h}(\mathbf{j}) }
	\leq 
	C_{} N ^{\#(\pi) - \frac{1}{2}} .
	\end{align*}
\end{corollary}
%
%
%
%
%

The next two propositions establish necessary and sufficient conditions for $\mathrm{p}_{\pi}(\mathrm{x}_{1},\mathrm{x}_{2},\ldots,\mathrm{x}_{r})$ to be the  zero polynomial.
Roughly speaking, the polynomial $\mathrm{p}_{\pi}(\mathrm{x}_{1},\mathrm{x}_{2},\ldots,\mathrm{x}_{r})$ is zero if only and if the blocks of the partition $\pi$ group the elements of the set $[\pm 2m]$ in such a way that the positive and negative signs appearing in (\ref{eqn.initial.poly}) cancel each other out. 
%
%
%

\begin{proposition}\label{necessary.condt.special.partitions}
	Suppose $\pi = \{B_{1}, B_{2}, \ldots, B_{2m}\}$ is a pairing partition in $P( \pm 2 m )$.
	Then the polynomial $\mathrm{p}_{\pi} (\mathrm{x}_{1},\mathrm{x}_{2},\ldots,\mathrm{x}_{2m})$
	is zero if and only if $\pi$ is a symmetric partition such that $k \sim_\pi l$ implies $k + l$ odd for all integers $k,l \in [\pm 2 m]$.
\end{proposition}
%
%
%
\begin{proof}
	Suppose  $\mathrm{p}_{\pi} (\mathrm{x}_{1},\mathrm{x}_{2},\ldots,\mathrm{x}_{2m})$ is the zero polynomial and take $a_{t,s}$  as (\ref{quotient.polynomial.coefficients}) for $1 \leq t \leq s \leq 2m$. 
	To prove $\pi$ is a symmetric partition such that $k \sim_\pi l$ implies $k + l$ odd for all integers $k,l \in [\pm 2 m]$, it suffices to show that for every integer $k \in [2m]$ there exist an integer $l \in [2m]$ such that $k+l$ is odd and either 
	$k \sim_\pi l$ and $-k \sim_\pi -l$ or 
	$k \sim_\pi - l$ and $-k \sim_\pi l$.
	Fix $k \in [2m]$ and 
	let $t',s' \in [2m]$ such that $k \in B_{t'}$ and $-k \in B_{s'}$. 
	Since $\mathrm{p}_{\pi} (\mathrm{x}_1, \ldots, \mathrm{x}_{2m}) = \sum_{1 \leq t \leq s \leq r} a_{t,s} \mathrm{x}_{t} \mathrm{x}_{s}$ is the zero polynomial, we must have $a_{t',s'}= 0 $. 
	Now, if $t' \neq s'$, from (\ref{quotient.polynomial.coefficients}) we get that
	\begin{align*}
	a_{t',s'} = 
	(-1)^{k}
	+
	\sum_{\substack{ 	l \in [2m] \setminus \{k\} \\ 
			l \in B_{t'} , -l \in B_{s'} }}
	(-1)^{l} 
	+
	\sum_{\substack{ 	l \in [2m] \\ 
			-l \in B_{t'} , l \in B_{s'} }}
	(-1)^{l} 
	=
	0,
	\end{align*}
	which implies there exists $l \in [2m] $ such that $ (-1)^{k} + (-1)^{l}$ is zero and either $l \in B_{t'}$ and $-l \in B_{s'}$ or  $-l \in B_{t'}$ and $l \in B_{s'}$. But this is equivalent to the desired conclusion. A similar argument works for the case $s=t$.

	Now, if the partition $\pi$ is a symmetric pairing in $P(\pm 2m)$ such that $k \sim_\pi l$ implies $k + l$ odd for all integers $k,l \in [\pm 2 m]$, 
	we can write $\pi=\{B_{1},B_{2},\ldots,B_{2m}\}$ with $B_{1}=\{-k_{1},-l_{1}\},B_{2}=\{k_{1},l_{1}\}, B_{3}=\{-k_{2},-l_{2}\},B_{2}=\{k_{2},l_{2}\},\ldots,B_{2m}=\{k_{m},l_{2m}\}$
	and $k_{1},l_{1},k_{2},l_{2},\ldots,l_{m} \in [\pm 2m]$ satisfying $k_{i}+l_{i}$ odd for $i=1,2,\ldots, m$. 
	Moreover, since $\bigcup_{i=1}^{2m} B_{i}= [\pm 2m]$ and $(-1)^{k} = (-1)^{-k}$ for $k \in [\pm 2m]$, we have
	\[
	-\mathrm{x}_{-1} \mathrm{x}_{1} + \mathrm{x}_{-2} \mathrm{x}_{2} -\mathrm{x}_{-3} \mathrm{x}_{3} + \cdots  +\mathrm{x}_{-2m} \mathrm{x}_{2m}
	=
	\sum_{i=1}^{m} (-1)^{k_{i}} \mathrm{x}_{-k_{i}} \mathrm{x}_{k_{i}} +  (-1)^{l_{i}} \mathrm{x}_{-l_{i}} \mathrm{x}_{l_{i}}.  
	\]
	Therefore, from the definition of $\mathrm{p}_{\pi}(\mathrm{x}_{1},\mathrm{x}_{2},\ldots,\mathrm{x}_{2m})$ and the fact that $k_{i}+l_{i}$ is odd for $i=1,2,\ldots, m$, we get 
	\[
	\mathrm{p}_{\pi}(\mathrm{x}_{1},\mathrm{x}_{2},\ldots,\mathrm{x}_{2m})
	=
	\sum_{i=1}^{m} (-1)^{k_{i}} \mathrm{x}_{2i-1} \mathrm{x}_{2i} +  (-1)^{l_{i}} \mathrm{x}_{2i-1} \mathrm{x}_{2i} = 0.   
	\]
\end{proof}
%
%
%

\begin{proposition}\label{minimal.special.polynomials}
	Let $\pi = \{B_1, B_2, \ldots, B_n\}$ be a partition in $P( \pm 2 m )$.
	If there is a partition $\theta \in P( \pm 2 m )$ such that $\theta \leq \pi$ and $\mathrm{p}_{\theta}$ is the zero polynomial, then $\mathrm{p}_{\pi}$ is also the zero polynomial. 
	Conversely, if $\mathrm{p}_{\pi}$ is the zero polynomial, then there is symmetric pairing partition $\theta \leq \pi $ such that $\mathrm{p}_{\theta}$ is the zero polynomial.
\end{proposition}
%
%
%
\begin{proof}
Suppose $\theta \leq \pi$ and $\mathrm{p}_{\theta}$ is the zero polynomial. 
Write $\theta =\{B_{1,1},B_{1,2}, \ldots, B_{1,m_1},\ldots, B_{n,m_n} \}$ with $B_{i} = \cup_{j=1}^{m_i}B_{i,j}$ for $i=1,2,\ldots, n$.
Take 
$\mathcal{A} = \mathbb{Z}\left[\mathrm{x}_{1},\mathrm{x}_{-1},\ldots,\mathrm{x}_{2m},\mathrm{x}_{-2m}\right]$, 
$\mathcal{B}=\mathbb{Z}[\mathrm{x}_{1,1}, \mathrm{x}_{1,2}, \ldots, \allowbreak  \mathrm{x}_{1,m_1}, \ldots, \mathrm{x}_{n,m_n}]$, and 
$\mathcal{C} = \mathbb{Z}[\mathrm{x}_{1}, \mathrm{x}_{2}, \ldots, \mathrm{x}_{n}]$ and let  $\Phi :  \mathcal{A}  \rightarrow \mathcal{B}$ and $\Psi :  \mathcal{B}  \rightarrow \mathcal{C}$ be the unique homomorphisms such that $ \Phi(\mathrm{x}_{k})= \mathrm{x}_{i,j}$ if $k \in B_{i,j}$ and $\Psi(\mathrm{x}_{i,j}) = \mathrm{x}_{i}$. 
Note that $(\Psi \circ \Phi) (\mathrm{x}_{k}) = \mathrm{x}_{l}$ only if $k \in B_{l}$, and thus, by definition of $\mathrm{p}_{\pi}$ and $\mathrm{p}_{\theta}$, we have that 
\begin{align*}
\mathrm{p}_{\pi}
=
\Psi \circ \Phi \left(\sum_{k=1}^{m} (-1)^{k} \mathrm{x}_{-k} \mathrm{x}_{k}\right)
=
\Psi \left( \mathrm{p}_{\theta} \right) .
\end{align*}
Hence, if $\mathrm{p}_{\theta}$ is the zero polynomial, so is $\mathrm{p}_{\pi}$.
%
%
%

Suppose now $\mathrm{p}_{\pi}$ is the zero polynomial and let $\theta$ be a minimal element of the set 
\(  \{  \widehat{\pi}  \in P(\pm 2m) : \widehat{\pi} \leq \pi \text { and } \mathrm{p}_{ \widehat{\pi}} = 0 \} \) endowed with the partial order inherited from $P(\pm 2 m)$. 
By Proposition \ref{necessary.condt.special.partitions}, the partition $\theta$ has no singletons, and thus, either $\theta$ is a pairing partition or $\theta$ has a block with at least three elements.
%
%
Let us assume $\theta = \{C_1, C_2, \ldots, C_n\}$ has a block with at least three elements, say $C_{n}$. 
By Proposition \ref{necessary.condt.special.partitions}, there are integers $k, l \in [2m]$ such that $k+l$ is odd and at least one of the following conditions holds:
\begin{enumerate}
\item\label{finer.partition} $+k , +l \in C_{n}$ and $-k \sim_\theta -l$
\item\label{finer.partition.2} $+k , - l \in C_{n}$ and $-k \sim_\theta +l$
\item\label{finer.partition.3} $-k, -l \in C_{n}$ and $+k  \sim_\theta +l$
\item\label{finer.partition.4} $-k , +l \in C_{n}$ and $+k \sim_\theta -l$
\end{enumerate}
Assume (\ref{finer.partition}) holds. 
Then, $ C_{n} \setminus\{k,l\}$ is not empty, and hence,
letting $\widehat{C}_{i} = C_{i}$ for $i=1,2,\ldots n-1$,  $\widehat{C}_{n} = C_{n} \setminus\{k,l\}$, and $\widehat{C}_{n+1}= \{k,l\}$, we have $\widehat\theta = \{\widehat{C}_{1},  \widehat{C}_{2}, \ldots, \widehat{C}_{n+1} \}$ is a partition of $[\pm 2 m ]$ such that $\widehat\theta \lneq \theta$, i.e., $\theta \geq \widehat\theta$ but $\theta \neq \widehat\theta$. 
Let us show that $\mathrm{p}_{\widehat\theta}$ must be the zero polynomial, contradicting the minimality of $\theta$. 
Take
$\mathcal{A} = \mathbb{Z}\left[\mathrm{x}_{1},\mathrm{x}_{-1},\ldots,\mathrm{x}_{m},\mathrm{x}_{-m}\right]$, 
$\mathcal{B}=\mathbb{Z}[\mathrm{x}_{1}, \mathrm{x}_{2}, \ldots, \mathrm{x}_{n}]$, and 
$\widehat{\mathcal{B}} = \mathbb{Z}[\mathrm{x}_{1}, \mathrm{x}_{2}, \ldots, \mathrm{x}_{n+1}]$ and let  $\Phi :  \mathcal{A}  \rightarrow \mathcal{B}$ and $\widehat\Phi :  \mathcal{A}  \rightarrow \widehat{\mathcal{B}}$ be the unique homomorphisms such that $ \Phi(\mathrm{x}_{i})= \mathrm{x}_{j}$ if $i \in {C}_{j}$ and $\widehat\Phi(\mathrm{x}_{i}) = \mathrm{x}_{j}$ if $i \in \widehat{C}_{j}$. 
Since ${\Phi}(\mathrm{x}_{i}) = \widehat{\Phi}(\mathrm{x}_{i})$ for $i \in [\pm 2 m] \setminus \{k,l\}$, we have
\begin{align*}
\sum_{\substack{i=1 \\ i\neq k,l}}^{2m} (-1)^{i} 
{\Phi}(\mathrm{x}_{-i}) {\Phi}( \mathrm{x}_{i} ) 
=
\sum_{\substack{i=1 \\ i \neq k,l}}^{2m} (-1)^{i} 
\widehat{\Phi}(\mathrm{x}_{-i}) \widehat{\Phi}( \mathrm{x}_{i} ).
\end{align*}
Moreover, since  $-k \sim_\theta -l$ we have $\Phi( \mathrm{x}_{-k} ) =  \Phi( \mathrm{x}_{-l} ) = \widehat\Phi( \mathrm{x}_{-k} ) = \widehat\Phi( \mathrm{x}_{-l} ) $, so we get 
\begin{align*}
0
&=
(-1)^{k} 	\Phi(\mathrm{x}_{-k}) 
\Phi( \mathrm{x}_{k} )
+ 
(-1)^{l} 	\Phi(\mathrm{x}_{-l})	
\Phi( \mathrm{x}_{l} ) \\
&=
(-1)^{k} 	\widehat\Phi(\mathrm{x}_{-k}) 
\widehat\Phi( \mathrm{x}_{k} )
+ 
(-1)^{l} 	\widehat\Phi(\mathrm{x}_{-l})	
\widehat\Phi( \mathrm{x}_{l} ) 
\end{align*}
since $k+l$ is odd, $\Phi(\mathrm{x}_{k}) = \Phi(\mathrm{x}_{l}) = \mathrm{x}_{n}$, and  $\widehat\Phi(\mathrm{x}_{k}) = \widehat\Phi(\mathrm{x}_{l}) = \mathrm{x}_{n+1}$.
Thus, we obtain
\begin{align*}
\mathrm{p}_{\widehat\theta}
=
\sum_{\substack{i=1 }}^{2m} (-1)^{i} 
\widehat\Phi(\mathrm{x}_{-i}) \widehat\Phi( \mathrm{x}_{i} ) 
=
\sum_{\substack{i=1 }}^{2m} (-1)^{i} 
\Phi(\mathrm{x}_{-i}) \Phi( \mathrm{x}_{i} ) 
= 
\mathrm{p}_{\theta} = 0
\end{align*}
But then, $\theta$ is not minimal, and therefore, (\ref{finer.partition}) does not hold. Similar arguments show that neither (\ref{finer.partition.2}), nor (\ref{finer.partition.3}), nor (\ref{finer.partition.4}) hold. 
Therefore, the partition $\theta$ must be a pairing, and, in fact, a symmetric pairing by Proposition \ref{necessary.condt.special.partitions}. 
\end{proof}
%
%
%
%
%
As mentioned earlier, in proving Theorem \ref{thm.second.order.asymptotics}, we need to consider sums as in (\ref{eqn.shifted.DFT.injective.graph.sums}). 
Note that if $\sigma:[\pm 2m] \rightarrow [\pm 2m]$ is bijective and  $\mathrm{p}_{\sigma^{-1} \circ \pi}$ is the zero polynomial, then $\mathbf{h}(\mathbf{j}) = 1$ for any function  $\mathbf{j} :[\pm 2 m]\rightarrow [N] $ satisfying $\kernel{\mathbf{j}} \geq \sigma^{-1} \circ \pi$, and hence, we would get 
\[
\sum_{\substack{	\mathbf{j} :[\pm 2 m]\rightarrow [N] \\  
		\kernel{\mathbf{j}} = \pi } }
\mathbf{h}(\mathbf{j}\circ\bm{\sigma})  
\quad =
\sum_{\substack{	\mathbf{j} :[\pm 2 m]\rightarrow [N] \\  
		\kernel{\mathbf{j}} = \sigma^{-1} \circ \pi 
} }
\mathbf{h}(\mathbf{j})
\quad = \quad \frac{N!}{(N-\#(\pi))!} . 
\]
On the other hand, if the polynomial $\mathrm{p}_{\sigma^{-1} \circ \pi}$ is non-zero, we have that (\ref{eqn.shifted.DFT.injective.graph.sums}) is of order $N^{\#(\pi) - 1/2}$ by Corollary \ref{cor.bound.gauss.sum}. 
We will now use the previous results to classify all symmetric pairing partitions so that $\mathrm{p}_{\sigma^{-1} \circ \pi}$ is the zero polynomial.

\begin{lemma}\label{corollary.special.pairing.partitions}
	Let $m=m_{1} + m_{2}$ for some integers $m_{1}, m_{2}\geq 1$ and let $\sigma$ be the permutation given by (\ref{def.permutation.sigma}). 
	Suppose $k$ and $l$ are integers in $[2m]$ and $\pi$ is a symmetric pairing partition in $P(\pm 2 m)$ 
	such that  $\mathrm{p}_{\sigma^{-1} \circ \pi}$ 
	is the zero polynomial.
	If $-k \sim_\pi l$, then $\sigma^{-t}(-k) \sim_\pi \sigma^{t}(l)$ for every integer $t \geq 0$. 
	On the other hand, if $-k \sim_\pi -l$, then $\sigma^{t}(-k) \sim_\pi \sigma^{t}(-l)$ for every integer $t \geq 0$.
\end{lemma}
%
%
%
\begin{proof}
	Note that   $\hat{k} \sim_{\pi} \hat{l} $ implies $\sigma (- \sigma^{-1} (\hat{k}) ) \sim_{\pi} \sigma (- \sigma^{-1} (\hat{l}) )$.
	Indeed, by Proposition \ref{minimal.special.polynomials}, the partition $\sigma^{-1} \circ \pi$ is symmetric since $\mathrm{p}_{\sigma^{-1} \circ \pi}$ is the zero polynomial, and hence $- \sigma^{-1} (\hat{k})  \sim_{\sigma^{-1} \circ \pi} - \sigma^{-1} (\hat{l}) $ provided $\hat{k} \sim_{\pi} \hat{l} $, but in that case we must have $\sigma (- \sigma^{-1} (\hat{k}) ) \sim_{\pi} \sigma (- \sigma^{-1} (\hat{l}) )$. 
	Note also that for every integer $\hat{k} \in [\pm 2m]$ we have
	\begin{align*}
	\sigma ( -\sigma^{-1}(\hat{k}) ) = \left\{\begin{array}{cl}
	\sigma(\hat{k}) 		& \text{if } \hat{k} > 0 \\
	\sigma^{-1}(\hat{k}) 	& \text{if } \hat{k} < 0 \\
	\end{array}
	\right. 
	\end{align*}
	since for $1 \leq k \leq 2m $ we have $\sigma^{-1}(k)=-k$, $-\sigma^{-1}(-k)<0$, and  $\sigma(-k)=k$.

	Now, suppose $\sigma^{-t}(-k) \sim_{\pi} \sigma^{t}(l)$ for some integer  $t \geq 0$.  
	If $t$ is even, then $\hat{k}=\sigma^{-t}(-k) < 0 < \sigma^{t}(l)=\hat{l}$, and hence $ \sigma^{-t-1}(-k) = \sigma (- \sigma^{-1} (\hat{k}) ) \sim_{\pi} \sigma (- \sigma^{-1} (\hat{l}) ) = \sigma^{t+1}(l)$. 
	On the other hand, if $t$ is odd, we have $\sigma^{-t}(-k) > 0 > \sigma^{t}(l)$, and hence $ \sigma^{-t-1}(-k) = - \sigma^{-t}(-k) \sim_{\pi} -  \sigma^{t}(l) = \sigma^{t+1}(l)$ since $\pi$ is symmetric. 
	Thus, $-k \sim_{\pi} l$ implies $\sigma^{-t}(-k) \sim_\pi \sigma^{t}(l)$ for every integer $t \geq 0$ by induction on $t$. 
	Similarly, assuming $-k \sim_{\pi} -l$, we get $\sigma^{t}(-k) \sim_{\pi} \sigma^{t}(-l)$ for all $t\geq0$.
\end{proof}
%
%
%

%
%
\begin{proposition}\label{prop.minimal.special.partitions}
   	Let $m=m_{1} + m_{2}$ for some integers $m_{1}, m_{2}\geq 1$ and let $\sigma$ be the permutation given by (\ref{def.permutation.sigma}). 
	Suppose ${\pi}$ is a symmetric pairing partition of $[\pm 2m ]$ and denote by 
	${\pi}_1$ and ${\pi}_2$ 
	the restrictions of ${\pi}$ to $[\pm 2m_1]$ and $[\pm 2m ] \setminus [\pm 2m_1]$, respectively.
	Then $\mathrm{p}_{\sigma^{-1} \circ {\pi}}$ is the zero polynomial if and only if one of the following conditions holds:
	\begin{enumerate}
		\item\label{prop.minimal.special.partitions.1} ${\pi} \neq {\pi}_1 \sqcup {\pi}_2$, $m_{1}=m_{2}$,  and there are  integers $1 \leq k \leq 2m_{2}$  and $2m_{1}+1 \leq l \leq 2m_{1}+2m_{2}$ such that $k+l$ is even and 
		\[
		{\pi} = 
		\left\{ 
		\{\sigma^{t}(-k),\sigma^{-t}(l)\} \mid t=1,2,\ldots, 4m_{1} 
		\right\} .
		\]

		\item\label{prop.minimal.special.partitions.2} ${\pi} \neq {\pi}_1 \sqcup {\pi}_2$, $m_{1}=m_{2}$,  and there are integers $1 \leq k \leq 2m_{2}$  and $2m_{1}+1 \leq l \leq 2m_{1}+2m_{2}$ such that $k+l$ is odd and 
		\[
		{\pi} = \left\{ 
		\{\sigma^{t}(-k),\sigma^{t}(-l)\}	\mid t=1,2,\ldots, 4m_{1} 
		\right\} .
		\]
		
		\item\label{prop.minimal.special.partitions.3} ${\pi} = {\pi}_1 \sqcup {\pi}_2$ and there are integers $1 \leq k \leq 2m_{1}$ and $2m_{1}+1 \leq l \leq 2m_{1}+2m_{2}$ such that 
		\[
		{\pi}_{1} = \left\{	
		\{\sigma^{t_{1}}(-k),\sigma^{-t_{1}}(k)
		\}
		\mid t_{1}=1,2,\ldots, 2m_{1} \right\}
		\] and 
		\[
		{\pi}_{2} = \left\{	
		\{\sigma^{t_{2}}(-l),\sigma^{-t_{2}}(l)
		\}
		\mid t_{2}=1,2,\ldots, 2m_{2} \right\} . 
		\]
		
		\item\label{prop.minimal.special.partitions.4} ${\pi} = {\pi}_1 \sqcup {\pi}_2$, $m_{1}$ and $m_{2}$ are odd integers, 
		\[ 
		{\pi}_{1} = \left\{
		\{\sigma^{t_{1}}(-1),\sigma^{t_{1}}(-m_1-1)\} \mid t_{1}=1,\ldots, 2m_{1}
		\right\} ,
		\] and 
		\[{\pi}_{2} = \left\{
		\{\sigma^{t_{2}}(-2m_1-1),\sigma^{t_{2}}(-2m_1-m_2-1)\} \mid t_{2}=1,\ldots, 2m_{2}
		\right\} .
		\]
		
		\item\label{prop.minimal.special.partitions.5} ${\pi} = {\pi}_1 \sqcup {\pi}_2$, $m_{2}$ is odd, there is an integer $1 \leq k \leq 2m_{1}$ such that 
		\[
		{\pi}_{1} = \left\{	
		\{\sigma^{t_{1}}(-k),\sigma^{-t_{1}}(k)
		\}
		\mid t_{1}=1,2,\ldots, 2m_{1} \right\},
		\] and  
		\[
		{\pi}_{2} = \left\{
		\{\sigma^{t_{2}}(-2m_1-1),\sigma^{t_{2}}(-2m_1-m_2-1)\} \mid t_{2}=1,\ldots, 2m_{2}.
		\right\} 
		\] 
		
		\item\label{prop.minimal.special.partitions.6} ${\pi} = {\pi}_1 \sqcup {\pi}_2$, $m_{1}$ is odd, 
		\[
		{\pi}_{1} = \left\{
		\{\sigma^{t_{1}}(-1),\sigma^{t_{1}}(-m_1-1)\} \mid t_{1}=1,\ldots, 2m_{1}
		\right\},
		\] and there is an integer $2m_{1}+1 \leq l \leq 2m_{1}+2m_{2}$ such that  
		\[
		{\pi}_{2} = \left\{	
		\{\sigma^{t_{2}}(-l),\sigma^{-t_{2}}(l)
		\}
		\mid t_{2}=1,2,\ldots, 2m_{2} \right\}. 
		\]

	\end{enumerate}
\end{proposition}
%
%
%
\begin{proof} 	
Put ${\hat\pi} = \sigma^{-1} \circ \pi$. 
Suppose ${\hat\pi} = \{B_{1}, B_{2},\ldots, B_{r}\}$ and let $\Phi$ be the unique homomorphism from $\mathbb{Z}[\mathrm{x}_{-1},\mathrm{x}_{1},\ldots,\mathrm{x}_{-2m},\mathrm{x}_{2m}]$ to $\mathbb{Z}[\mathrm{x}_{1},\mathrm{x}_{2},\ldots,\mathrm{x}_{r}]$ such that $\Phi(\mathrm{x}_{i}) = \mathrm{x}_{j}$ if $i \in B_{j}$. 
If condition \textit{(\ref{prop.minimal.special.partitions.1})} holds, 
then ${\hat\pi} = \left\{ \{\sigma^{t}(-k),\sigma^{-t-2}(l)\} \mid t = 1,2,\ldots, 4m_{1} \right\}  $ and  $\sigma^{t}(-k) + \sigma^{-t-2}(l)$ is odd for $t=1,2,\ldots, 4m_{1}$. 
Thus, since we can write
\begin{align*}
2 \cdot \sum_{i=1}^{2m} (-1)^{i} \mathrm{x}_{-i}\mathrm{x}_{i}
= 
\sum_{t=1}^{2m_{1}} (-1)^{\sigma^{2t}(-k)} 
\mathrm{x}_{\sigma^{2t}(-k)}
\mathrm{x}_{\sigma^{2t+1}(-k)} 
+
\sum_{t=1}^{2m_{1}} (-1)^{\sigma^{-2t-2}(l)} 
\mathrm{x}_{\sigma^{-(2t+1)-2}(l)}
\mathrm{x}_{\sigma^{-2t-2}(l)},   
\end{align*}
we get 
$
\mathrm{p}_{{\hat\pi}} = \Phi ( \sum_{i=1}^{2m} (-1)^{i} \mathrm{x}_{-i} \mathrm{x}_{i} )  = 0.
$
It follows from similar arguments that $\mathrm{p}_{{\hat\pi}}$ is the zero polynomial if \textit{(\ref{prop.minimal.special.partitions.2})}
holds. 
Now, if \textit{(\ref{prop.minimal.special.partitions.4})} holds and we take  $l=2m_{1}+1$, we have $\sigma^{t}(-k) + \sigma^{-t-2}(k)$ and $\sigma^{t}(-l) + \sigma^{t}(-m_{2}-l)$ are odd 
and ${\hat\pi} = \left\{ \{\sigma^{t}(-k),\sigma^{-t-2}(k)\} , \{\sigma^{t}(-l),\sigma^{t}(-m_{2}-l)\}\mid t \geq 0 \right\}$. 
Thus, since we can write
\begin{align*}
\sum_{i=1}^{2m} (-1)^{i} \mathrm{x}_{-i}\mathrm{x}_{i}
= & 
\frac{1}{2}	\sum_{t=1}^{2m_{1}} (-1)^{\sigma^{2t}(-k)} 	
\mathrm{x}_{\sigma^{2t}(-k)}
\mathrm{x}_{\sigma^{2t+1}(-k)} 
+
\frac{1}{2}	\sum_{t=1}^{2m_{1}} (-1)^{\sigma^{-2t-2}(k)} 
\mathrm{x}_{\sigma^{-(2t+1)-2}(k)}
\mathrm{x}_{\sigma^{-2t-2}(k)}
\\ & + 
\sum_{t=1}^{m_2} (-1)^{\sigma^{2t}(-l)} 	
\mathrm{x}_{\sigma^{2t}(-l)}
\mathrm{x}_{\sigma^{2t+1}(-l)} 
+ 
\sum_{t=1}^{m_{2}} (-1)^{\sigma^{2t}(-m_2 - l)} 
\mathrm{x}_{\sigma^{2t}(-m_2-l)}
\mathrm{x}_{\sigma^{2t+1}(-m_2-l)} 
\end{align*}
we get
$\displaystyle
\mathrm{p}_{\hat\pi}= 0$. Similar arguments show that if either \textit{(\ref{prop.minimal.special.partitions.3})}, 
\textit{(\ref{prop.minimal.special.partitions.5})}, or \textit{(\ref{prop.minimal.special.partitions.6})} holds, then $\mathrm{p}_{\hat\pi}$ is the zero polynomial.

Suppose now $\mathrm{p}_{{\hat\pi}}$ is the zero polynomial and let ${\hat\pi}_1$ and ${\hat\pi}_2$ be the restrictions of ${\hat\pi}$ to $[\pm 2m_1]$ and $[\pm (2m_1 + 2m_2)] \setminus [\pm 2m_1]$, respectively.
We will consider two cases ${\hat\pi} \neq {\hat\pi}_{1} \sqcup {\hat\pi}_{2}$ and ${\hat\pi} = {\hat\pi}_{1} \sqcup {\hat\pi}_{2}$. 
Assume first ${\hat\pi} \neq {\hat\pi}_{1} \sqcup {\hat\pi}_{2}$. 
By Proposition \ref{necessary.condt.special.partitions}, there are integers $1 \leq k \leq 2m_{1}$ and $2m_{1}+1 \leq l \leq 2m_{1}+2m_{2}$ such that $k+l$ is odd and one of the following holds:
\begin{enumerate}[\quad (1')]
	\item\label{prop.minimal.special.partitions.1.proof} $k \sim_{\hat\pi} - l$ and $-k \sim_{\hat\pi} l$.  
	\item\label{prop.minimal.special.partitions.2.proof} $k \sim_{\hat\pi} l$ and $-k \sim_{\hat\pi} -l$.
\end{enumerate}	
Suppose (\ref{prop.minimal.special.partitions.2.proof}') holds. Then, $-k=-\sigma(-k) \sim_{\pi} -\sigma(-l) = -l$, and by Lemma \ref{corollary.special.pairing.partitions}, we have that $\sigma^{t}(-k) \sim_{\pi} \sigma^{t}(-l)$ for every integer $t \geq 0$.
Moreover, since $-k=\sigma^{4m_1}(-k)$, $\sigma^{4m_1}(-k) \sim_{\pi} \sigma^{4m_1}(-l)$, and ${\pi}$ is a pairing, we must have $-l=\sigma^{4m_1}(-l)$. 
But, the equation $-l = \sigma^{t}(-l)$ holds only if $t$ is an integer multiple of $4m_2$, and hence, $4m_1$ is a multiple of $4m_2$. 
Similarly, $4m_2$ is a multiple of $4m_1$, and therefore, $4m_1 = 4m_2$, and the partition ${\hat\pi}$ satisfies condition (\ref{prop.minimal.special.partitions.2}). 
A similar argument shows that ${\hat\pi}$ satisfies condition (\ref{prop.minimal.special.partitions.1}) if we suppose (\ref{prop.minimal.special.partitions.1.proof}') holds.
Assume now ${\hat\pi} = {\hat\pi}_{1} \sqcup {\hat\pi}_{2}$. By Proposition \ref{necessary.condt.special.partitions}, there is an integer $1 < \hat{k} \leq 2m_{1}$ satisfying one of the following:
\begin{enumerate}[\quad (a)]
\item\label{prop.minimal.special.partitions.3a.proof} 
$1 \sim_{\hat\pi} \hat{k}$, $-1 \sim_{\hat\pi} -\hat{k}$, and $1+\hat{k}$ is odd.

\item\label{prop.minimal.special.partitions.3b.proof} 
$1 \sim_{\hat\pi} - \hat{k}$, $-1 \sim_{\hat\pi} \hat{k}$, and $1+\hat{k}$ is odd.
\end{enumerate}	
and there is an integer $2m_1+1 < \hat{l} \leq 2m_{1} + 2m_2 $ satisfying one of the following:
\begin{enumerate}[\quad (A)]
\item\label{prop.minimal.special.partitions.4a.proof} 
$2m_1 +1 \sim_{\hat\pi} \hat{l}$, $-2m_1 -1 \sim_{\hat\pi} -\hat{l}$, and $2m_1 +1+\hat{l}$ is odd.

\item\label{prop.minimal.special.partitions.4b.proof} 
$2m_1 +1 \sim_{\hat\pi} -\hat{l}$, $-2m_1 -1 \sim_{\hat\pi} \hat{l}$, and $2m_1 +1+\hat{l}$ is odd. 
\end{enumerate}	
If (\ref{prop.minimal.special.partitions.3a.proof}) holds, we know that $\sigma^{t}(-1) \sim_{\pi} \sigma^{t}(-\hat{k})$ for every integer $t \geq 0$ by Lemma \ref{corollary.special.pairing.partitions}.
But then, since $-\hat{k} = \sigma^{2\hat{k}-2}(-1)$,  $\sigma^{2\hat{k}-2}(-1) \sim_{{\pi}} \sigma^{2\hat{k}-2}(-\hat{k})$, and ${\pi}$ is a pairing, we must have $\sigma^{2\hat{k}-2}(-\hat{k})	 = -1$, 
or, equivalently, $4\hat{k} - 4$ is a multiple of $4m_{1}$.
Therefore,  $m_{1}=\hat{k}-1$ is odd, and $ \sigma^{t_1}(-1) \sim_{\hat\pi} \sigma^{t_1}(-m_1 - 1)$, and hence
\[ 
{\pi}_{1} = \left\{
			\{ \sigma^{t_{1}}(-1),\sigma^{t_{1}}(-m_1-1)\} \mid t_{1}=1,2,\ldots, 4m_{1}
	\right\} 
\]
On the other hand, if (\ref{prop.minimal.special.partitions.3b.proof}) holds, it follows from Lemma \ref{corollary.special.pairing.partitions} that $\sigma^{t}(1) \sim_{\pi} \sigma^{-t}(-\hat{k}')$ for every integer $t \geq 0$. 
Moreover, since $\hat{k}'$ has the same parity as $ 1+ \hat{k}$, we have $\hat{k}'=2k-1$ for some integer $k \geq 1$.
But then, since $k = \sigma^{2k-2}(1) =  -\sigma^{2-2k}(2k-1)$ and $\sigma^{2k-2}(1) \sim_{{\pi}} \sigma^{2-2k}(2k-1)$, we have $k \sim_{\pi} -k$, and therefore,
\[
{\pi}_{1} = \left\{	
\{\sigma^{t_{1}}(k),-\sigma^{-t_{1}}(-k)
\}
\mid t_{1}=1,2,\ldots, 2m_{1} \right\}
\]
Similar arguments show that if (\ref{prop.minimal.special.partitions.4a.proof}) holds, then $m_{2}$ is odd and
\[ 
{\pi}_{2} = \left\{
\{ \sigma^{t_{2}}(-2m_{1}-1),\sigma^{t_{2}}(-2m_{1}-m_{2}-1)\} \mid t_{2}=1,2,\ldots, 4m_{2}
\right\},
\]
and if (\ref{prop.minimal.special.partitions.4b.proof}) holds, then there is an integer $l$ such that 
\[
{\pi}_{2} = \left\{	
\{\sigma^{t_{2}}(l),\sigma^{-t_{2}}(-l)
\}
\mid t_{2}=1,2,\ldots, 2m_{2} \right\}
\]
This completes the proof that if $\mathrm{p}_{{\hat\pi}}$ is the zero polynomial, then ${\hat\pi}$ must satisfy either
\textit{(\ref{prop.minimal.special.partitions.1})}, 
\textit{(\ref{prop.minimal.special.partitions.2})},
\textit{(\ref{prop.minimal.special.partitions.3})}, \textit{(\ref{prop.minimal.special.partitions.4})},
\textit{(\ref{prop.minimal.special.partitions.5})}, or \textit{(\ref{prop.minimal.special.partitions.6})}.
\end{proof}
%
%
%
%
%

\begin{remark}\label{rmk.minimal.polynomial.graphs}\label{rmk.single.loops} 
	Notice the results regarding the polynomials $\mathrm{p}_{\pi}$ and $\mathrm{p}_{\sigma^{-1} \circ {\pi}}$ being zero can be restated in terms of the graphs $\mathcal{G}_{\pi}$ and  $\vec{\mathcal{G}}_{\pi}$ from Section \ref{sec.general.graph.sums}. 
	For instance, Proposition \ref{prop.minimal.special.partitions} states that if $\pi \in P(\pm 2m)$ is a symmetric partition, then the polynomial $\mathrm{p}_{\sigma^{-1} \circ {\pi}}$ is zero if and only if one of the following conditions for the directed graph $\vec{\mathcal{G}}_{\pi}$, where $F_{t}$ denotes the edge $E_{2m_{1}+t}$ for $t=1,2,\ldots,2m_{2}$,  holds:
	\begin{enumerate}
		\item\label{graph.minimal.special.partitions.1}$m_{1}=m_{2}$  and there is an integer $1 \leq l \leq m_{2}$ so that the graph $\vec{\mathcal{G}}_{\pi}$ can be represented as 
\begin{center}\includegraphics[scale=1]{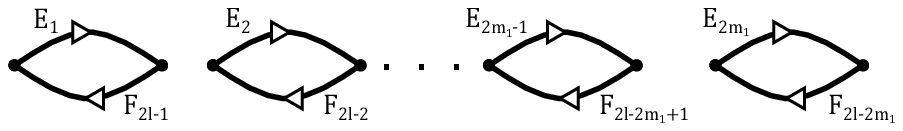}\end{center}
		%
		%

		\item\label{graph.minimal.special.partitions.2} $m_{1}=m_{2}$  and there is an integer $1 \leq l \leq m_{2}$ so that the graph $\vec{\mathcal{G}}_{\pi}$ can be represented as 
\begin{center}\includegraphics[scale=1]{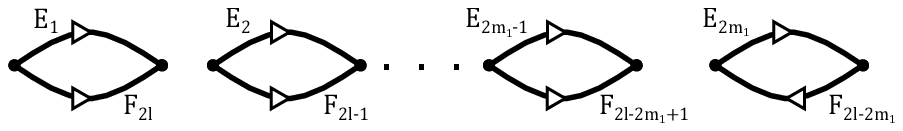}\end{center}
		%
		%

		\item\label{graph.minimal.special.partitions.3} $\vec{\mathcal{G}}_{\pi}$ is the disjoint union of  $\vec{\mathcal{G}}_{\pi_{1}}$ and $\vec{\mathcal{G}}_{\pi_{2}}$, there is an integer $1 \leq k \leq 2m_{1}$ so that $\vec{\mathcal{G}}_{\pi_{1}}$ can be represented as 
\begin{center}\includegraphics[scale=1]{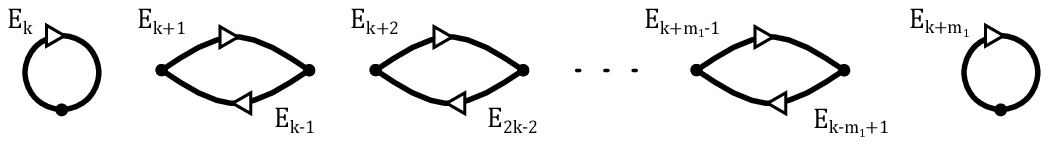}
		\end{center}
		and there is an integer $2m_{1}+1 \leq l \leq 2m_{1} + 2m_{2}$  so that $\vec{\mathcal{G}}_{\pi_{2}}$  can be represented as
	\begin{center}
		\includegraphics[scale=1]{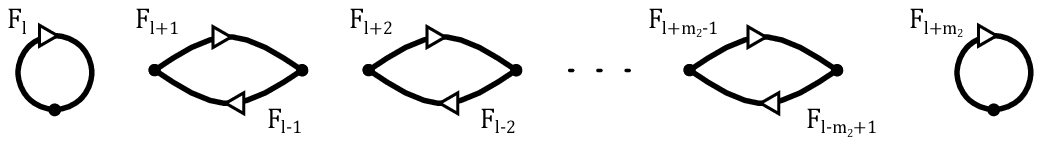}
	\end{center}

		\item\label{graph.minimal.special.partitions.6}  $m_{1}$ and $m_{2}$ are odd integers, the graph $\vec{\mathcal{G}}_{\pi}$ is the disjoint union of $\vec{\mathcal{G}}_{\pi_{1}}$ and $\vec{\mathcal{G}}_{\pi_{2}}$, the graph $\vec{\mathcal{G}}_{\pi_{1}}$ can be represented as 
		\begin{center}\includegraphics[scale=1]{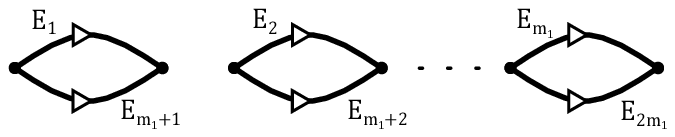}\end{center} 
		and the graph $\vec{\mathcal{G}}_{\pi_{2}}$ can be represented as 
		\begin{center}\includegraphics[scale=1]{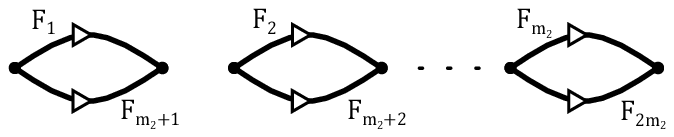}\end{center}

		\item\label{graph.minimal.special.partitions.4} $m_{2}$ is odd, $\vec{\mathcal{G}}_{\pi}$ is the disjoint union of  $\vec{\mathcal{G}}_{\pi_{1}}$ and $\vec{\mathcal{G}}_{\pi_{2}}$, there is an integer $1 \leq k \leq 2m_{1}$ so that $\vec{\mathcal{G}}_{\pi_{1}}$ can be represented as 
\begin{center}\includegraphics[scale=1]{special_pairings_cross_3a}
		\end{center}
		and the graph $\vec{\mathcal{G}}_{\pi_{2}}$ can be represented as 
\begin{center}\includegraphics[scale=1]{special_pairings_cross_6b}\end{center}

		\item\label{graph.minimal.special.partitions.5} $m_{1}$ is odd, the graph $\vec{\mathcal{G}}_{\pi}$ is the disjoint union of $\vec{\mathcal{G}}_{\pi_{1}}$ and $\vec{\mathcal{G}}_{\pi_{2}}$, the graph $\vec{\mathcal{G}}_{\pi_{1}}$ can be represented as 
		\begin{center}\includegraphics[scale=1]{special_pairings_cross_6a}\end{center}
		and there is an integer $1 \leq l \leq 2m_{1}$ so that $\vec{\mathcal{G}}_{\pi_{2}}$ can be represented as 
		\begin{center}\includegraphics[scale=1]{special_pairings_cross_3b}
		\end{center}

	\end{enumerate}
	In the graphs above, $2l-t$, $2l+t-1$, and $k \pm t $ are taken modulo $2m_{1}$ for $t=1,2,\ldots,2m_{1}$ and $l \pm t $ is taken modulo $2m_{2}$ for $t=1,2,\ldots,2m_{2}$. 
\end{remark}

\section{The Bounded Cumulants Property}\label{sec.proof.main.thm.1}
In this section, we first prove Lemma \ref{lemma.bounded.cumulant.property}, and then, before we can apply it to get the conclusion in Theorem \ref{thm.bounded.cumulants.property}, we need to establish the relations between the notion of asymptotic free independence, the bounded cumulants property, and linear functionals on an algebra of non-commutative polynomials. 
%
%
%
%
%
\begin{proof}[\textbf{Proof of Lemma \ref{lemma.bounded.cumulant.property}}] 
	%
	%
	%
	Put $V_k= U^{*}_{N,{i}_k} U_{N,{i}_{\gamma(k)}}$ for $k=1,2, \ldots, m$ and note that 
	\[
	\Tr{Y_{k}} = \Tr{	A_{m'_{k}+1} V_{m'_{k}+1} A_{m'_{k}+2} V_{m'_{k}+2} \cdots 
		A_{m'_{k}+m_{k}} V_{m'_{k}+m_{k}}}. 
	\]
	Now, letting $\mathbf{a(j)}, \mathbf{v}_1(\mathbf{j}), \mathbf{v}_2(\mathbf{j}), \ldots, \mathbf{v}_n(\mathbf{j})$ be given by
	\begin{align*}
	\mathbf{a(j)} = \prod_{k=1}^{m} A_{k}(j_{-k},j_{k})
	\text{ and } 
	\mathbf{v}_{k}(\mathbf{j}) = \prod_{l=m'_{k}+1}^{m'_{k}+m_{k}} V_{l}(j_{l},j_{-\gamma(l)})
	\text{ for } k=1,2,\ldots,n 
	\end{align*}
	for each function $\mathbf{j}:[\pm m] \rightarrow [N]$, 
	we have that
	\begin{align}\label{cumm.lemma.1}
	\mathfrak{c}_{n} \left[ \Tr{Y_1}, \ldots, \Tr{Y_n}	\right] 
	=&
	\sum_{ \mathbf{j}: [\pm m]\rightarrow [N] } 
	\mathbf{a}(\mathbf{j}) 
	\mathfrak{c}_{n} \left[ 
	\mathbf{v}_{1}(\mathbf{j}), \mathbf{v}_{2}(\mathbf{j})
	,\ldots,					\mathbf{v}_{n}(\mathbf{j})
	\right] 
	\end{align}
	since the matrices $A_k$ are deterministic and the classical cumulants are multi-linear.

	By hypothesis, 
	the family of random matrices $\left\{V_{l}\right\}_{l=1}^{m}$ 
	is distribution-invariant under conjugation by signed permutation matrices, 
	thus given a function $\mathbf{j}:[\pm m] \rightarrow [N]$ we have
	\begin{align*}
	\mathfrak{c}_{n} 
	\left[ 		
	\mathbf{v}_{1}(\mathbf{j}),\ldots,\mathbf{v}_{n}(\mathbf{j})
	\right]
	=
	\prod_{k=1}^{m} 	
	\epsilon_{\sigma(j_k)}\epsilon_{\sigma(j_{-k})}
	\mathfrak{c}_n \left[ 
	\mathbf{v}_1(\sigma \circ \mathbf{j}), \ldots ,
	\mathbf{v}_n(\sigma \circ \mathbf{j})
	\right]
	\end{align*}
	for all signs $\epsilon_{1},\epsilon_{2},\ldots,\epsilon_{N} \in \{\pm 1 \}$ 
	and permutations $\sigma \in  \{ f:[N] \rightarrow [N] \mid f \text{ is bijective}\}$. 
	This implies that
	\begin{align*}
	\mathfrak{c}_{n} \left[\mathbf{v}_{1}(\mathbf{j}),\ldots,\mathbf{v}_{n}(\mathbf{j})\right]=0
	\end{align*}
	whenever $\kernel{\mathbf{j}}$ contains at least one block of odd size, 
	and 
	\begin{align*}
	\mathfrak{c}_{n} \left[\mathbf{v}_{1}(\mathbf{j}),\ldots,\mathbf{v}_{n}(\mathbf{j})\right]
	= 
	\mathfrak{c}_{n} \left[\mathbf{v}_{1}(\mathbf{j'}),\ldots,\mathbf{v}_{n}(\mathbf{j'})\right]
	\end{align*}
	provided 
	a function $\mathbf{j'}:[\pm m] \rightarrow [N]$ satisfies 
	$\kernel{\mathbf{j'}}=\kernel{\mathbf{j}}$.
	Thus, 
	letting $\mathfrak{c}_n \left[ \pi \right]$ denote the common value
	$\mathfrak{c}_{n} \left[\mathbf{v}_{1}(\mathbf{j}),\ldots,\mathbf{v}_{n}(\mathbf{j})\right]$ 
	among all those functions $\mathbf{j}:[\pm m] \rightarrow [N]$ satisfying 
	$\kernel{\mathbf{j}}=\pi$, 	Equation (\ref{cumm.lemma.1}) becomes
	\begin{align}\label{cumm.lemma.2}
	\mathfrak{c}_n \left[ \Tr{Y_1}, \ldots, \Tr{Y_n}	\right] 
	=&
	\sum_{\pi \in P_{\text{even}}(\pm m)}
	\mathfrak{c}_n \left[ \pi \right]			
	\sum_{\substack{ \mathbf{j}: [\pm m]\rightarrow [N]  \\ \mathrm{ker}(\mathbf{j})=\pi }}
	\mathbf{a}(\mathbf{j})  .
	\end{align}
	Moreover, the M\"{o}bius inversion formula in (\ref{mobius.inversion.formula.0}) implies
	\[
	\sum_{\substack{ 
			\mathbf{j}: [\pm m]\rightarrow [N]  \\ \mathrm{ker}(\mathbf{j}) = \pi }}
	\mathbf{a}(\mathbf{j})
	=
	\sum_{\substack{ 
			\theta \in P(\pm m)  \\ 
			\theta \geq \pi }}
	\mu(\pi,\theta)
	\sum_{\substack{ 
			\mathbf{j}: [\pm m]\rightarrow [N]  \\ \mathrm{ker}(\mathbf{j}) \geq \theta }}
	\mathbf{a}(\mathbf{j})
	\]
	since for all partitions $\theta \in P(\pm m)$ we have the relation
	\[
	\sum_{\substack{ 
			\mathbf{j}: [\pm m]\rightarrow [N]  \\ \mathrm{ker}(\mathbf{j}) \geq \theta }}
	\mathbf{a}(\mathbf{j})
	= 
	\sum_{\substack{ 
			\pi \in P(\pm m)  \\ 
			\pi \geq \theta }}
	\sum_{\substack{ 
			\mathbf{j}: [\pm m]\rightarrow [N]  \\ \mathrm{ker}(\mathbf{j}) = \pi }}
	\mathbf{a}(\mathbf{j}).
	\]
	Hence, we get
	\begin{align}\label{cumm.lemma.3}
	\mathfrak{c}_n \left[ \Tr{Y_1}, \ldots, \Tr{Y_n}	\right] 
	=&
	\sum_{\pi \in P_{\text{even}}(\pm m)}
	\sum_{\substack{
			\theta \in P_{\text{even}}(\pm m) \\ 
			\theta \geq \pi}} 
	\mathfrak{c}_n \left[ \pi \right]	
	\mu(\pi, \theta) 
	\sum_{\substack{
			\mathbf{j}:[\pm m]\rightarrow[N] \\ \mathrm{ker}(\mathbf{j}) \geq \theta}}
	\mathbf{a}(\mathbf{j})
	\end{align}
	Note that if a partition $\theta \in P(\pm m)$ has a block of the form $\{k,-k\}$, then 
	\[
	\sum_{\substack{
			\mathbf{j}:[\pm m]\rightarrow[N] \\ \mathrm{ker}(\mathbf{j}) \geq \theta}}
	\mathbf{a}(\mathbf{j})
	=
	\Tr{A_k} (\star) 
	\]
	where $ (\star) $ is a sum excluding the entries of $A_{k}$. Therefore, since each $A_{k}$ is assumed to be of trace zero, we have  
	\begin{align}\label{cumm.lemma.4}
	\mathfrak{c}_n \left[ \Tr{Y_1}, \ldots, \Tr{Y_n}	\right] 
	=&
	\sum_{\pi \in P_{\text{even}}(\pm m)}
	\sum_{\substack{
			\theta \in P_{\chi}(\pm m) \\ 
			\theta \geq \pi}} 
	\mathfrak{c}_n \left[ \pi \right]	
	\mu(\pi, \theta) 
	\sum_{\substack{
			\mathbf{j}:[\pm m]\rightarrow[N] \\ \mathrm{ker}(\mathbf{j}) \geq \theta}}
	\mathbf{a}(\mathbf{j})
	\end{align}
	where $P_{\chi}(\pm m )$ denotes the set of all partitions in $P_{\text{even}}( \pm m)$ with no blocks of the form $\{k,-k\}$.

	Now, for a partition $\theta \in P_{\chi}(\pm m)$, each connected component of the graph $\mathcal{G}_{\theta}$, constructed as in Section \ref{sec.general.graph.sums}, has at least two edges, and hence $\mathcal{G}_{\theta}$ has at most  $\frac{m}{2}$ connected components. 
	Thus, from Theorem \ref{mingo-speicher}, Proposition \ref{prop.graph.sum.exp.bound1}, and the equality in (\ref{cumm.lemma.4}), we get
	\begin{align*}\label{cumm.lemma.5}
	\abs*{\mathfrak{c}_n \left[ \Tr{Y_1}, \ldots, \Tr{Y_n}	\right] }
	\leq &	
	\sum_{\pi \in P_{\text{even}}(\pm m)}
	\sum_{\substack{
			\theta \in P_{\chi}(\pm m) \\ 
			\theta \geq \pi}} 
	\abs*{ \mathfrak{c}_n \left[ \pi \right] }	
	\abs*{ \mu(\pi, \theta) }
	N^{\frac{m}{2}} \prod_{k=1}^{m} \norm*{A_{k}} . 
	\end{align*}
	Since the sums above are over the finite sets $P_{\text{even}}(\pm m)$ and $P_{\chi}(\pm m)$, our proof will be complete if we show that there is a constant $C_{n}$ independent from $N$ such that 
	\[
	\abs*{\mathfrak{c}_n\left[
		\mathbf{v}_{1}(\mathbf{j}),\ldots,\mathbf{v}_{n}(\mathbf{j})
		\right]	}
	\leq 
	C_{n} N^{-\frac{m}{2}}
	\]
	for all functions $\mathbf{j}:[\pm m] \rightarrow [N]$. 
	Let $\mathbf{j}:[\pm m] \rightarrow [N]$ be arbitrary. 
	By H\"{o}lder's inequality, letting $m_B := \sum_{k \in B} m_k $ for any given subset $B$ of $[n]$, we have 
	\begin{align*}
	\norm*{\prod_{k \in B} \mathbf{v}_{k}(\mathbf{j})}_{1} 
	=  
	\norm*{\prod_{k \in B} \prod_{l=m'_{k}+1}^{m'_{k}+m_{k}} V_{l}(j_{l},j_{-\gamma(l)})}_{1} 
	\leq 
	\prod_{k \in B} \prod_{l=m'_{k}+1}^{m'_{k}+m_{k}} \norm*{V_{l}(j_{l},j_{-\gamma(l)})}_{m_{B}}	.
	\end{align*}
	But, by hypothesis, the $p$-norms of the entries of $ \sqrt{N} V_{l}$ are uniformly bounded, i.e., there are constants $C_1, C_2, \ldots, C_m$ such that
	\begin{align*}
	\norm*{V_l (j_{l},j_{-\gamma(l)})}_p  
	\leq 
	C_{p} N^{-\frac{1}{2}} 
	\end{align*}
	for all integers $1 \leq j_{l},j_{-\gamma(l)} \leq N$ and $p,l=1,2,\ldots, m$, and hence, we get 
	\begin{align*}
	\norm*{\prod_{k \in B} \mathbf{v}_{k}(\mathbf{j})}_{1} 
	\leq
	\prod_{k \in B} \prod_{l=m'_{k}+1}^{m'_{k}+m_{k}}
	C_{m_{B}} N^{-\frac{1}{2}}
	=
	\left( C_{m_{B}} N^{-\frac{1}{2}} \right)^{m_{B}} . 
	\end{align*}
	Now, the moment-cumulants relation in  (\ref{eqn.moment.cumulants.relation}) implies 
	\begin{align*}
	\abs*{\mathfrak{c}_n\left[
		\mathbf{v}_{1}(\mathbf{j}),\ldots,\mathbf{v}_{n}(\mathbf{j})
		\right]	}
	\leq &
	\sum_{\substack{ \pi \in P(n) } }
	\abs*{\mu(\pi, 1_n)}
	\prod_{B \in \pi} 
	\norm*{
		\prod_{k \in B} 
		\mathbf{v}_{k}(\mathbf{j}) }_{1},
	\end{align*}
	so it follows that 
	\begin{align*}
	\abs*{\mathfrak{c}_n\left[
		\mathbf{v}_{1}(\mathbf{j}),\ldots,\mathbf{v}_{n}(\mathbf{j})
		\right]	} 
	\leq &	
	\sum_{\substack{ \pi \in P(n) } }
	\abs*{\mu(\pi, 1_n)} 
	\prod_{B \in \pi} \left( C_{m_{B}} N^{-\frac{1}{2}} \right)^{m_{B}} \\
	= &
	N^{-\frac{m}{2}}	
	\sum_{\substack{ \pi \in P(n) } }
	\abs*{\mu(\pi, 1_n)} 
	\prod_{B \in \pi} \left(C_{m_{B}} \right)^{m_{B}} .
	\end{align*}
	And the proof of Lemma \ref{lemma.bounded.cumulant.property} is now complete. 
\end{proof}
%
%
%
%
%
%
%
%
%
Now, asymptotic free independence and the bounded cumulants property can be stated in terms of some linear functionals, and, by doing so, we can show the bounded cumulants property is actually equivalent to a condition that is a consequence of Lemma \ref{lemma.bounded.cumulant.property}, see Proposition \ref{prop.asymptotic.centering} and Corollary \ref{corollary.asymptotic.centering} below. 
Thus, to prove Theorem \ref{thm.bounded.cumulants.property}, we first examine the relations between the notion of asymptotic free independence, the bounded cumulants property, and linear functionals on an algebra of non-commutative polynomials first. 
\subsection*{Multi-linear functionals on non-commutative polynomials and notions from free probability}\hspace*{\fill} \newline

Let $I$ be a non-empty set. 
Let $\mathcal{A}$ denote the algebra of non-commutative polynomials $\mathbb{C}\left\langle \mathrm{x}_{i} \mid i \in I \right\rangle$ and let  $\mathcal{A}_{i} \subset \mathcal{A}$ denote the algebra of polynomials $\mathbb{C}\left[ x_{i} \right] $ for each index $i \in I$.
Suppose we are given random matrix ensembles $\{X_{N,i}\}_{N=1}^{\infty}$ with $i\in I$ where each $X_{N,i}$ is a $N$-by-$N$ random matrix and consider the sequence of unital linear functional $\{\varphi_{N}: \mathcal{A}\rightarrow \mathbb{C}\}_{N=1}^{\infty}$ where each $\varphi_{N}$ is defined by 
\begin{align}\label{eqn.unital.funtionals} 
\varphi_{N} \left[  \mathrm{p} \right] 
:= \exptr{ \tr{ \mathrm{p} \left( \{X_{N,i} \}_{i \in I} \right) }} 
\quad \forall \mathrm{p} \in \mathcal{A} .
\end{align}
Note that the two conditions necessary for the random matrix ensembles $\{X_{N,i}\}_{N=1}^{\infty}$ with $i\in I$ to be asymptotically freely independent, namely,  \textit{({AF.}\ref{def.asymp.free.1.alternative})} and \textit{({AF.}\ref{def.asymp.free.2.alternative})} from Definition \ref{def.asymptotic.freenes.alternative}, can be stated in terms of the linear functionals $\varphi_{N}$ as
\begin{enumerate}[({AF.}1)]
	\item $\displaystyle  \lim_{N\rightarrow\infty}\varphi_{N} \left[  \mathrm{p}_{}\right] $ exists for every $\mathrm{p}_{} \in \mathcal{A}_{i}$ and every $i\in I$, and 
	\item  $\displaystyle \lim_{N \rightarrow \infty} \varphi_{N}\left[ (\mathrm{p}_{1} - \varphi_{N}\left[\mathrm{p}_{1}\right]   ) (\mathrm{p}_{2} - \varphi_{N}\left[\mathrm{p}_{2}\right]   ) \cdots (\mathrm{p}_{m} - \varphi_{N}\left[\mathrm{p}_{m}\right]   ) \right]  = 0 $ whenever $\mathrm{p}_{k} \in  \mathcal{A}_{i_{k}} $ with $i_{1}\neq i_{2}, i_{2} \neq i_{3}, \ldots, i_{m-1} \neq i_{m}$
\end{enumerate}
Moreover, assuming \textit{({AF.}\ref{def.asymp.free.1.alternative})} holds, we can replace each $\varphi_{N}\left[\mathrm{p}_{k}\right]$ appearing in \textit{({AF.}\ref{def.asymp.free.2.alternative})} by $\varphi_{}\left[  \mathrm{p}_{k}\right] :=\lim_{N\rightarrow\infty}\varphi_{N} \left[  \mathrm{p}_{k}\right] $; more concretely, we have the following.
%
%
\begin{proposition}\label{prop.asymptotic.centering.1}
Suppose each ensemble $\{X_{N,i}\}_{N=1}^{\infty}$ has a limiting distribution, namely, $\varphi_{}\left[  \mathrm{p}_{}\right] :=\lim_{N\rightarrow\infty}\varphi_{N} \left[  \mathrm{p}_{}\right] $ exists for every $\mathrm{p}_{} \in \mathcal{A}_{i}$ and every $i\in I$. 
Then the ensembles $\{X_{N,i}\}_{N=1}^{\infty}$ with $i\in I$  are asymptotically freely independent if and only if the following holds:
\begin{enumerate}[({AF.2.}1)]
	\item[({AF.2}')]\label{asymptotic.centering.on.the.go.2}  $\displaystyle \lim_{N \rightarrow \infty} \varphi_{N}\left[  (\mathrm{p}_{1} - \varphi_{}\left[	\mathrm{p}_{1}\right]  ) (\mathrm{p}_{2} - \varphi_{}\left[	\mathrm{p}_{2}\right]  ) \cdots (\mathrm{p}_{m} - \varphi_{}\left[	\mathrm{p}_{m}\right]  ) \right]  = 0 $ whenever $\mathrm{p}_{k} \in  \mathcal{A}_{i_{k}} $ with $i_{1}\neq i_{2}, i_{2} \neq i_{3}, \ldots, i_{m-1} \neq i_{m}$
\end{enumerate}
Moreover, if either {({AF.}\ref{def.asymp.free.2.alternative})} or \textit{({AF.2}')} holds, then $\lim_{N\rightarrow\infty}\varphi_{N} \left[  \mathrm{p}_{}\right]$ exists for every $ \mathrm{p}_{} \in \mathcal{A}$ . 
\end{proposition}
%
%
%
%
\begin{proof} 
	Let $J$ be the set of all positive integers $m$ satisfying the following property: 
	if $\mathrm{p}^{}_{1} \in \mathcal{A}_{i^{}_{1}}$, $\mathrm{p}^{}_{2} \in \mathcal{A}_{i^{}_{2}}$, $\ldots$, $\mathrm{p}^{}_{m} \in \mathcal{A}_{i^{}_{m}}$ with $i_{1},i_{2},\ldots,i_{m}\in I$ and $i^{}_{1} \neq i^{}_{2},  i^{}_{2} \neq i^{}_{3}, \ldots , i^{}_{m_{}-1} \neq i^{}_{m_{}} $, 
	then
	$\lim_{N \rightarrow \infty} \varphi_{N} \left[  \mathrm{p}_{1}   \mathrm{p}_{2} \cdots   \mathrm{p}_{m} \right]$ exists. 
	Since the algebras $\mathcal{A}_{i}$ with $i \in I$ generate $\mathcal{A}$ and each $\varphi_{N}$ is linear, $\lim_{N\rightarrow\infty}\varphi_{N} \left[  \mathrm{p}_{}\right]$ exists for every $ \mathrm{p}_{} \in \mathcal{A}$ if the set $J$ contains every positive integer. 
	Now, by hypothesis, $1$ belongs to $J$, so let us assume $1, 2, \ldots, m-1$ belong to $J$ and suppose $\mathrm{p}^{}_{1}, \mathrm{p}^{}_{2}, \ldots, \mathrm{p}^{}_{m}$ are as above. 
	Thus, if $S=\{ k_{1} < k_{2} < \cdots < k_{\abs{S}} \}$ is a strict subset of $[m]$, the limits
	\begin{align*}
	\lim_{N\rightarrow \infty} (-1)^{\vert S^c \vert}
	\prod_{k \in S^c} \varphi_{N}\left[ \mathrm{p}_{k} \right]
	\varphi_N \left[	\vec{\prod_{k \in S }} \mathrm{p}_{k}	\right] 
	\qquad \text{and} \qquad 
	\lim_{N\rightarrow \infty} (-1)^{\vert S^c \vert}
	\prod_{k \in S^c} \varphi_{}\left[ \mathrm{p}_{k} \right]
	\varphi_N \left[\vec{\prod_{k \in S }} \mathrm{p}_{k}\right] , 
	\end{align*}
	where $S^{c}$ denotes the complement of $S$ in the set $[m]$ and ${\vert S^{c} \vert}$ denotes the cardinality of $S^{c}$, exist. 
	Moreover, if \textit{({AF.}\ref{def.asymp.free.2.alternative})}  holds, the equality 
	\begin{align*}
	\varphi_{N}\left[ 
	(\mathrm{p}_{1} - \varphi_{N}\left[\mathrm{p}_{1}\right]   ) \cdots 
	(\mathrm{p}_{m} - \varphi_{N}\left[\mathrm{p}_{m}\right]   ) \right]  	
	& =  		
	\varphi_N \left[ \mathrm{p}_{1}  \cdots \mathrm{p}_{m}  \right] +
	\sum_{ S \varsubsetneq [m]} (-1)^{\vert S^c \vert}
	\prod_{k \in S^c} \varphi_{N}\left[ \mathrm{p}_{k} \right]
	\varphi_N \left[
	\vec{\prod_{k \in S }} \mathrm{p}_{k}
	\right] ,
	\end{align*}
	implies $\lim_{N\rightarrow \infty}\varphi_N \left[ \mathrm{p}_{1}  \cdots \mathrm{p}_{m}  \right] $ exists. 
	And therefore, $J$ contains every positive integer by induction on $m$. 
	Similarly, if \textit{({AF.2}')} holds, then $\lim_{N\rightarrow \infty}\varphi_N \left[ \mathrm{p}_{1}  \cdots \mathrm{p}_{m}  \right] $ exists, and hence $J$ contains every positive integer, since each $\varphi_{N}\left[ \mathrm{p}_{k} \right]$ in the equality above can be replaced by $\varphi_{}\left[ \mathrm{p}_{k} \right]$. 
	%


Let us now show that \textit{({AF.}\ref{def.asymp.free.2.alternative})} and  \textit{({AF.2}')} are equivalent. 
Suppose $\mathrm{p}^{}_{1} \in \mathcal{A}_{i^{}_{1}}, \mathrm{p}^{}_{2} \in \mathcal{A}_{i^{}_{2}}, \ldots, \mathrm{p}^{}_{m} \in \mathcal{A}_{i^{}_{m}}$ with $i_{1},i_{2},\ldots,i_{m}\in I$ and $i^{}_{1} \neq i^{}_{2},  i^{}_{2} \neq i^{}_{3}, \ldots , i^{}_{m_{}-1} \neq i^{}_{m_{}} $ and  take $\mathrm{q}_k=\mathrm{p}_k-\varphi \left[ \mathrm{p}_k \right] $ for $k=1,2,\ldots, m$. 
We then have $\lim_{N\rightarrow \infty} \varphi_{N}\left[ \mathrm{q}_{k} \right] = 0 $ and the equality
\begin{align*}
\varphi_N \left[ (\mathrm{p}_{1}-\varphi_N \left[ \mathrm{p}_{1}\right] ) \cdots(\mathrm{p}_{m}-\varphi_N \left[ \mathrm{p}_{m}\right] ) \right]
=&
\varphi_N \left[ \mathrm{q}_{1} \cdots \mathrm{q}_{m}  \right] +
\sum_{ S \varsubsetneq [m]} (-1)^{\vert S^c \vert}
\prod_{k \in S^c} \varphi_N \left[ \mathrm{q}_{k}\right] 
\varphi_N \left[ \vec{\prod_{k \in S}} \mathrm{q}_{k} \right] \\ 
=& 	
\varphi_N \left[ (\mathrm{q}_{1}-\varphi_N \left[ \mathrm{q}_{1}\right] ) \cdots(\mathrm{q}_{m}-\varphi_N \left[ \mathrm{q}_{m}\right] ) \right]
\end{align*}
Thus, if condition \textit{({AF.2}')} holds, then $\lim_{N\rightarrow\infty}\varphi_{N} \left[  \mathrm{p}_{}\right]$ exists for every $ \mathrm{p}_{} \in \mathcal{A}$, and hence
\begin{align*}
0 = 
\lim_{N\rightarrow \infty} (-1)^{\vert S^c \vert}
\prod_{k \in S^c} \varphi_{N}\left[ \mathrm{q}_{k} \right]
\varphi_N \left[	\vec{\prod_{k \in S }} \mathrm{q}_{k}	\right]
\qquad \forall S \varsubsetneq [m];   
\end{align*}
additionally, we have $
\lim_{N\rightarrow\infty}
\varphi_N \left[ (\mathrm{p}_{1}-\varphi\left[ \mathrm{p}_{1} \right]) \cdots(\mathrm{p}_{m}-\varphi\left[\mathrm{p}_{m} \right]) \right]  =  \lim_{N\rightarrow\infty}\varphi_{N} \left[  \mathrm{q}_{1} \cdots \mathrm{q}_{m} \right] = 0 $, and thus \[\lim_{N\rightarrow\infty}
\varphi_N \left[ (\mathrm{p}_{1}-\varphi_N \left[ \mathrm{p}_{1}\right] ) \cdots(\mathrm{p}_{m}-\varphi_N \left[ \mathrm{p}_{m}\right] ) \right] = 0 .\]
This shows that  \textit{({AF.2}')} implies \textit{({AF.}\ref{def.asymp.free.2.alternative})}.  
Similarly,  \textit{({AF.}\ref{def.asymp.free.2.alternative})} implies \textit{({AF.2}')}. 
\end{proof}
%
%
%
%
In the literature, however, the most common definition of asymptotic free independence for random matrix ensembles in terms of the linear functionals $\varphi_{N}$ defined by (\ref{eqn.unital.funtionals}) goes as follows: $\{X_{N,i}\}_{N=1}^{\infty}$ with $i\in I$ are asymptotically freely independent if they have a joint limiting (algebraic) distribution, i.e., $ \lim_{N \rightarrow \infty} \varphi_{N}\left[  \mathrm{p} \right]$ exist for every polynomial $\mathrm{p}\in \mathcal{A}$, and letting $\varphi_{} := \lim_{N \rightarrow \infty} \varphi_{N}$, we have 
\[
\varphi_{}\big[  (\mathrm{p}_{1} - \varphi_{}\left[	\mathrm{p}_{1}\right]  ) (\mathrm{p}_{2} - \varphi_{}\left[	\mathrm{p}_{2}\right]  ) \cdots (\mathrm{p}_{m} - \varphi_{}\left[	\mathrm{p}_{m}\right]  ) \big]  = 0 
\]
whenever $\mathrm{p}_{k} \in  \mathcal{A}_{i_{k}} $ with $i_{1}\neq i_{2}, i_{2} \neq i_{3}, \ldots, i_{m-1} \neq i_{m}$. 
The previous proposition shows equivalence between the common definition of asymptotic free independence and the one given in the introduction of this paper.

Now, the bounded cumulants property for the random matrix ensemble $\{ \{X_{N,i}\}_{i\in I}\}_{N=1}^{\infty}$ can also be established in terms of multi-linear functionals.
If for each integer $n\geq 1$, we consider the $n$-linear map $\rho_{N}: \mathcal{A}\times \cdots \times \mathcal{A} \rightarrow \mathbb{C}$ defined by 
\begin{align}\label{eqn.n-linear.maps}
\rho_{N}\left[ \mathrm{p}_{1},\mathrm{p}_{2}, \ldots, \mathrm{p}_{n}\right]
=
\mathfrak{c}_{n}\left[ \Tr{\mathrm{p}_{1}(\{X^{}_{N,i} \}_{i \in I}) } , \ldots, \Tr{\mathrm{p}_{n}(\{X^{}_{N,i} \}_{i \in I})} \right]
\end{align}
for all  $\mathrm{p}_{1}, \mathrm{p}_{2},\ldots, \mathrm{p}_{n} \in \mathcal{A}$ and where $\mathfrak{c}_{n}[\cdot,\ldots,\cdot]$ denotes the classical cumulant, from Section \ref{sec.classical.cumulants},
the random matrix ensemble  $\{ \{ X^{}_{N,i} \}_{i \in I} \}_{N=1}^{\infty}$ has then the bounded cumulants property if only if
\begin{align}\label{ineq.bnd.cum.propty.poly}
\sup_{N} \abs*{	\rho_{N}\left[\mathrm{p}_{1},\mathrm{p}_{2}, \ldots, \mathrm{p}_{n}\right]} < \infty  
\end{align}
for all $\mathrm{p}_{1}, \mathrm{p}_{2},\ldots, \mathrm{p}_{n} \in \mathcal{A}$ and all integers $n\geq 1$. 
Moreover, under some mild assumptions, each polynomial $\mathrm{p}_{k}$ appearing in (\ref{ineq.bnd.cum.propty.poly}) can be replaced by 
\[ (\mathrm{p}^{(k)}_1 - \varphi_N[\mathrm{p}^{(k)}_1] )
(\mathrm{p}^{(k)}_2 - \varphi_N[\mathrm{p}^{(k)}_2] )
\cdots
(\mathrm{p}^{(k)}_{m_k}-\varphi_N [\mathrm{p}^{(k)}_{m_{k}} ] ) 
\]
for some polynomials $\mathrm{p}^{(k)}_{1} \in \mathcal{A}_{i^{(k)}_{1}}, \mathrm{p}^{(k)}_{2} \in \mathcal{A}_{i^{(k)}_{2}}, \ldots ,\mathrm{p}^{(k)}_{m_{k}} \in \mathcal{A}_{i^{(k)}_{m_{k}}}$ and still get the bounded cumulants property. 
%
%
%
\begin{proposition}\label{prop.asymptotic.centering}
	Suppose $\varphi_N: \mathcal{A} \rightarrow \mathbb{C}$ is a unital linear functional and $\rho_N: \mathcal{A}\times \cdots \times \mathcal{A} \rightarrow \mathbb{C}$ is an $n$-linear functional for integer each $N \geq 1$. 
	If the limits $ 
	\lim_{N \rightarrow \infty}\varphi_N\left[\mathrm{p}\right]$ and $\lim_{N \rightarrow \infty}\rho_N\left[\mathrm{p}_1,\mathrm{p}_2,\ldots,\mathrm{p}_n\right]$ exist for all $\mathrm{p} \in \mathcal{A},  \mathrm{p}_1 \in \mathcal{A}_{{i}_1},\mathrm{p}_2 \in \mathcal{A}_{{i}_2},\ldots, \mathrm{p}_n \in \mathcal{A}_{{i}_n} $ with ${i}_{1},{i}_{2},\ldots,{i}_{n} \in I$,
	then the following are equivalent:
	\begin{enumerate}[(1)]
		\item\label{asymptotic.no.centering.1}
		$ \displaystyle\sup_{N} \abs*{	\rho_N\left[\mathrm{p}_1,\mathrm{p}_2,\ldots,\mathrm{p}_n\right]} < \infty$ for all $\mathrm{p}_1,\mathrm{p}_2,\ldots,\mathrm{p}_n \in \mathcal{A}$
		\item\label{asymptotic.no.centering.2}
		$\displaystyle\sup_{N} \abs*{ \rho_N\left[\mathrm{q}_1,\mathrm{q}_2,\ldots,\mathrm{q}_n\right]}  < \infty$	if each $\mathrm{q}_k$ is of the form
		\[\mathrm{q}_k=\mathrm{p}^{(k)}_1 \mathrm{p}^{(k)}_2 \cdots \mathrm{p}^{(k)}_{m_k}\]
		with $\mathrm{p}^{(k)}_{j} \in \mathcal{A}_{i^{(k)}_{j}}$ and $i^{(k)}_{1} \neq i^{(k)}_{2}, i^{(k)}_{2} \neq  i^{(k)}_{3}, \ldots, i^{(k)}_{m_{k}-1} \neq i^{(k)}_{m_{k}} $
		\item\label{asymptotic.centering.on.the.go}
		$\displaystyle\sup_{N} \abs*{ 
			\rho_N \left[\mathrm{q}_{N,1} , \mathrm{q}_{N,2} , \ldots , \mathrm{q}_{N,n} \right]} < \infty$
		if each $\mathrm{q}_{N,k}$ is of the form
		\[\mathrm{q}_{N,k}= 	(
		\mathrm{p}^{(k)}_1 - \varphi_N[\mathrm{p}^{(k)}_1]  
		)
		(
		\mathrm{p}^{(k)}_2 - \varphi_N[\mathrm{p}^{(k)}_2] 
		)
		\cdots
		(
		\mathrm{p}^{(k)}_{m_k}-\varphi_N [\mathrm{p}^{(k)}_{m_{k}} ] 
		) 
		\]
		with $\mathrm{p}^{(k)}_{j} \in \mathcal{A}_{i^{(k)}_{j}}$ and  $i^{(k)}_{1} \neq i^{(k)}_{2}, i^{(k)}_{2} \neq  i^{(k)}_{3}, \ldots, i^{(k)}_{m_{k}-1} \neq i^{(k)}_{m_{k}} $
	\end{enumerate}
\end{proposition}
%
%
%
\begin{proof} 
Conditions \textit{(\ref{asymptotic.no.centering.1})} and \textit{(\ref{asymptotic.no.centering.2})} are equivalent since each $\rho_{N}$ is $n$-linear and the algebra $\mathcal{A}$ is generated by the sub-algebras $\{\mathcal{A}_{i}\}_{i \in I}$.
We only need to prove that \textit{(\ref{asymptotic.no.centering.1})} implies \textit{(\ref{asymptotic.centering.on.the.go})} and \textit{(\ref{asymptotic.centering.on.the.go})} implies \textit{(\ref{asymptotic.no.centering.2})}.

Suppose \textit{(\ref{asymptotic.no.centering.1})} holds and let $\mathrm{q}_{N,k}$ is as in \textit{(\ref{asymptotic.centering.on.the.go})} for $k=1,2,\ldots,n$. 
Then, by multi-linearity, we have
\begin{align*} 
& 
\rho_N \left[ 	\mathrm{q}_{N,1} ,  \ldots , \mathrm{q}_{N,n} \right]  \\
& 
= \sum_{ J_{1} \subset [m_{1}], \ldots, J_{n} \subset [m_{n}] 	}
\left(
\prod_{k =1}^{n} 
\prod_{\substack{ j \in J_{k}^c   } } 
(-1)^{\vert J_{k}^c \vert} 
\varphi_{N} [ \mathrm{p}^{(k)}_{j} ]
\right)
\cdot
\rho_N \left[
\vec{\prod_{j \in J_{1} }} \mathrm{p}^{(1)}_{j},
\ldots,
\vec{\prod_{j \in J_{n} }} \mathrm{p}^{(k)}_{j}
\right] 
\end{align*}
The sum above is a finite sum and,  by hypothesis, each of its elements is uniformly bounded with respect to $N$.
Hence, \textit{(\ref{asymptotic.centering.on.the.go})} follows.

Let us assume now  \textit{(\ref{asymptotic.centering.on.the.go})} holds and let $J$ be the set of all positive integers $m$ satisfying the following property: 
if $m=m_{1}+m_{2}+\cdots + m_{n}$ for some positive integers $m_{1},m_{2},\ldots,m_{n}$,  and  $\mathrm{p}^{(k)}_{j} \in \mathcal{A}_{i^{(k)}_{j}}$ for $j=1,2,\ldots,m_k$ and $i^{(k)}_{1} \neq i^{(k)}_{2},  i^{(k)}_{2} \neq i^{(k)}_{3}, \ldots , i^{(k)}_{m_{k}-1} \neq i^{(k)}_{m_{k}} $ for $k=1,2,\ldots,n$, then
$\displaystyle{
\sup_{N} \abs*{
\rho_N (\mathrm{q}_{1} , \mathrm{q}_{2} , \ldots , \mathrm{q}_{n} ) } } < \infty$  where each $\mathrm{q}_k$ is given by $\mathrm{q}_{k}=\mathrm{p}^{(k)}_1 \mathrm{p}^{(k)}_2 \cdots \mathrm{p}^{(k)}_{m_k}$.
Note that we are done if we show that $J=\{n,n+1,n+2,\ldots\}$.
By hypothesis, $n$ belongs to $J$, so let us assume $n, n+1, \ldots, m-1$ belong to $J$ and let $\mathrm{p}^{(k)}_j$, $\mathrm{q}_{N,k}$ and $\mathrm{q}_k$ as above. Consider the equality
\begin{align*} 
& 
\rho_N \left[	\mathrm{q}_{N,1} ,  \ldots , \mathrm{q}_{N,n} \right]
-	\rho_N \left[	\mathrm{q}_{1} ,  \ldots , \mathrm{q}_{n} \right]  \\
& 
= 
\sum_{ \substack{ 
J_{1} \subset [m_{1}], \ldots, J_{n} \subset [m_{n}] 	\\
\cup^{n}_{k} J^{c}_{k}   \neq \emptyset
}}
\left(
\prod_{k =1}^{n} 
\prod_{\substack{ j \in J_{k}^c   } } 
(-1)^{\vert J_{k}^c \vert} 
\varphi_{N} [ \mathrm{p}^{(k)}_i ]
\right)
\cdot
\rho_N \left[
\vec{\prod_{j \in J_{1} }} \mathrm{p}^{(1)}_{j},
\ldots,
\vec{\prod_{j \in J_{n} }} \mathrm{p}^{(k)}_{j}
\right].
\end{align*}
given by $n$-linearity of $\rho_N$. Now, since  \textit{(\ref{asymptotic.centering.on.the.go})} holds, $\rho_N (	\mathrm{q}_{N,1} ,  \ldots , \mathrm{q}_{N,n} )$ is uniformly bounded with respect to $N$ and, by induction hypothesis, so is
$\rho_N [\vec{\prod_{j \in J_{1} }} \mathrm{p}^{(k)}_{j},
\ldots,
\vec{\prod_{j \in J_{n} }} \mathrm{p}^{(k)}_{j}] $  
if at least one $J_{k}$ is not $[m_{k}]$. 
Thus, $	\rho_N (	\mathrm{q}_{1} ,  \ldots , \mathrm{q}_{n} )$ is also uniformly bounded with respect to $N$, and hence $m$ belongs to $J$.
\end{proof}
%
%
%
%

The multi-linear functionals $\rho_{N}: \mathcal{A}\times \cdots \times \mathcal{A} \rightarrow \mathbb{C}$ given by (\ref{eqn.n-linear.maps}) are \textit{tracial in each entry}, i.e., for ever integer $k \in [n]$ and polynomials $\mathrm{q}_{0},\mathrm{q}_{1},\ldots,\mathrm{q}_{n} \in \mathcal{A}$, we have 
\[
\rho_{N}[ \mathrm{q}_{1},\ldots,\mathrm{q}_{k-1},\mathrm{q}_{0} \mathrm{q}_{k},\mathrm{q}_{k+1},\ldots,\mathrm{q}_{n}]
=
\rho_{N}[\mathrm{q}_{1},\ldots,\mathrm{q}_{k-1},\mathrm{q}_{k}\mathrm{q}_{0},\mathrm{q}_{k+1},\ldots,\mathrm{q}_{n}] .
\]
This traciality allows us to impose the condition that $i^{(k)}_{m_{k}} \neq i^{(k)}_{1}$ in \textit{(\ref{asymptotic.no.centering.2})} and \textit{(\ref{asymptotic.centering.on.the.go})} from Proposition \ref{prop.asymptotic.centering} and still get uniform boundedness of $\rho_N\left[\mathrm{p}_1,\mathrm{p}_2,\ldots,\mathrm{p}_n\right]$ with respect to $N$. 
%
%

%
%
\begin{corollary}\label{corollary.asymptotic.centering}
	Suppose $\mathcal{A}$, $\mathcal{A}_{i}$, $\varphi_{N}$, and $\rho_{N}$ are as in Proposition \ref{prop.asymptotic.centering}. 
	If $\rho_{N}$ is tracial in each entry,
	then condition \textit{(\ref{asymptotic.no.centering.1})} from Proposition \ref{prop.asymptotic.centering} is equivalent to any of the following: 
	\begin{enumerate}
		\item[\textit{(\ref{asymptotic.no.centering.2}')}]
		$\displaystyle\sup_{N} \abs*{\left[\mathrm{q}_1,\mathrm{q}_2,\ldots,\mathrm{q}_n\right]}  < \infty$	if each $\mathrm{q}_k$ is of the form
		$\mathrm{q}_k=\mathrm{p}^{(k)}_1 \mathrm{p}^{(k)}_2 \cdots \mathrm{p}^{(k)}_{m_k}$
		with $\mathrm{p}^{(k)}_{j} \in \mathcal{A}_{i^{(k)}_{j}}$ and $i^{(k)}_{1} \neq i^{(k)}_{2}, i^{(k)}_{2} \neq  i^{(k)}_{3}, \ldots, i^{(k)}_{m_{k}-1} \neq i^{(k)}_{m_{k}}$, and $i^{(k)}_{m_{k}} \neq i^{(k)}_{1} $
		\item[\textit{(\ref{asymptotic.centering.on.the.go}')}]
		$\displaystyle\sup_{N} \abs*{\rho_N \left[\mathrm{q}_{N,1} , \mathrm{q}_{N,2} , \ldots , \mathrm{q}_{N,n} \right] } < \infty$
		if each $\mathrm{q}_{N,k}$ is of the form
		\[\mathrm{q}_{N,k}= 	(
		\mathrm{p}^{(k)}_1 - \varphi_N[\mathrm{p}^{(k)}_1]  
		)
		(
		\mathrm{p}^{(k)}_2 - \varphi_N[\mathrm{p}^{(k)}_2] 
		)
		\cdots
		(
		\mathrm{p}^{(k)}_{m_k}-\varphi_N [\mathrm{p}^{(k)}_{m_{k}} ] 
		) 
		\]
		with $\mathrm{p}^{(k)}_{j} \in \mathcal{A}_{i^{(k)}_{j}}$ and  $i^{(k)}_{1} \neq i^{(k)}_{2}, i^{(k)}_{2} \neq  i^{(k)}_{3}, \ldots, i^{(k)}_{m_{k}-1} \neq i^{(k)}_{m_{k}}$, and $i^{(k)}_{m_{k}} \neq i^{(k)}_{1}$
	\end{enumerate}
\end{corollary}
%
%
%
%
\begin{proof} 
Note that while condition \textit{(\ref{asymptotic.no.centering.2})} from Proposition \ref{prop.asymptotic.centering} allows the indexes $i^{k}_{1}$ and $i^{k}_{m_{k}}$ to be possibly the same, condition \textit{(\ref{asymptotic.no.centering.2}')} above explicitly prohibits this. 
Thus, by traciality of $\rho_{N}$ in each entry, we have that \textit{(\ref{asymptotic.no.centering.2}'}) and  \textit{(\ref{asymptotic.no.centering.2})} are equivalent, and hence, it only remains to show that \textit{(\ref{asymptotic.centering.on.the.go}')} above implies  \textit{(\ref{asymptotic.centering.on.the.go})} from Proposition \ref{prop.asymptotic.centering}. 
Assume \textit{(\ref{asymptotic.centering.on.the.go}')} holds and let $J$ be the set of all positive integers $m$ satisfying the following property: 
if $m=m_{1}+m_{2}+\cdots + m_{n}$ for some positive integers $m_{1},m_{2},\ldots,m_{n}$,  and  $\mathrm{p}^{(k)}_{j} \in \mathcal{A}_{i^{(k)}_{j}}$ for $j=1,2,\ldots,m_k$ and $i^{(k)}_{1} \neq i^{(k)}_{2},  i^{(k)}_{2} \neq i^{(k)}_{3}, \ldots , i^{(k)}_{m_{k}-1} \neq i^{(k)}_{m_{k}} $ for $k=1,2,\ldots,n$, then
\begin{align}\label{eqn.bounded.n.linear.qnk}
\sup_{N} \abs*{\rho_N \left[\mathrm{q}_{N,1} , \mathrm{q}_{N,2} , \ldots , \mathrm{q}_{N,n} \right] }  < \infty 
\end{align}
where each $\mathrm{q}_{N,k}$ as above.
%
%
By hypothesis, $n$ belongs to $J$, so let us assume $n, n+1, \ldots, m-1$ belong to $J$ and let $\mathrm{p}^{(k)}_j$, $\mathrm{q}_{N,k}$ and $\mathrm{q}_k$ as above. 
Thus, if $i^{(k)}_{m_{k}} \neq i^{(k)}_{1}$ for $k=1,2,\ldots,n$, Inequality (\ref{eqn.bounded.n.linear.qnk}) holds. 
On the other hand, if $i^{(k)}_{m_{k}} = i^{(k)}_{1}$ for some $k \in \{1,2,\ldots,n\}$, let us consider polynomials $\widetilde{\mathrm{p}}^{}_{N,k}$ and $\widetilde{\mathrm{r}}_{N,k}$ given by 
\begin{align*}
\widetilde{\mathrm{p}}^{}_{N,k} 
= & 
(
\mathrm{p}^{(k)}_{m_k}-\varphi_N [\mathrm{p}^{(k)}_{m_k}]
) 				
(
\mathrm{p}^{(k)}_{1} - \varphi_N[\mathrm{p}^{(k)}_{1} ]
) \\
\widetilde{\mathrm{r}}_{N,k} 
= &
(
\mathrm{p}^{(k)}_{2} - \varphi_N[\mathrm{p}^{(k)}_{2}] 
)
(
\mathrm{p}^{(k)}_{3} - \varphi_N[\mathrm{p}^{(k)}_{3}]
)
\cdots
(
\mathrm{p}^{(k)}_{m_k-1}-\varphi_N[\mathrm{p}^{(k)}_{m_k-1}]
) 
\end{align*}
By traciality of $\rho_{N}$ in the $k$-th entry, we have
\begin{align*}
\rho_N \left[ \mathrm{q}_{N,1} , \ldots ,\mathrm{q}_{N,k} , \ldots , \mathrm{q}_{N,n} \right] 
= 	&
\rho_N \left[ \mathrm{q}_{N,1} , \ldots ,  \widetilde{\mathrm{p}}_{N,k}  \widetilde{\mathrm{r}}_{N,k} , \ldots , \mathrm{q}_{N,n} \right] ;
\end{align*}
moreover, from the relation  
\begin{align*}
\widetilde{\mathrm{p}}^{}_{N,k} 
= & 	
(\mathrm{p}^{(k)}_{m_k} \mathrm{p}^{(k)}_{1} 
-	\varphi_N [ \mathrm{p}^{(k)}_{m_k} \mathrm{p}^{(k)}_{1}	]
)
+	\varphi_N[	\mathrm{p}^{(k)}_{m_k} \mathrm{p}^{(k)}_{1}	]
-	\varphi_N[	\mathrm{p}^{(k)}_{m_k}	] 
\varphi_N[	\mathrm{p}^{(k)}_{1}	] \\
&	
-	\varphi_N[\mathrm{p}^{(k)}_{m_k} ]
( \mathrm{p}^{(k)}_{1} - \varphi_N[\mathrm{p}^{(k)}_{1}	 ] ) 
-	\varphi_N[ \mathrm{p}^{(k)}_{1}	]
( \mathrm{p}^{(k)}_{m_k}  - \varphi_N[\mathrm{p}^{(k)}_{m_k} ] )
\end{align*}
we get the equality
\begin{align*}
\rho_N [ \mathrm{q}_{N,1} , \ldots  , \mathrm{q}_{N,n} ]
=	&  
	\rho_N  [ \mathrm{q}_{N,1} 
				, \ldots ,
			(\mathrm{p}^{(k)}_{m_k} \mathrm{p}^{(k)}_{1} 
			-	\varphi_N [ \mathrm{p}^{(k)}_{m_k} \mathrm{p}^{(k)}_{1}	]
			)		\widetilde{\mathrm{r}}_{N,k}
				,\ldots, 
			\mathrm{q}_{N,n} ]   
\\ &
	+ 
		\varphi_N [ \mathrm{p}^{(k)}_{m_k} \mathrm{p}^{(k)}_{1}	]
		\rho_N [ \mathrm{q}_{N,1} , \ldots , \widetilde{\mathrm{r}}_{N,k} , \ldots , \mathrm{q}_{N,n} ] 
\\ &
	-
		\varphi_N[\mathrm{p}^{(k)}_{m_k} ] 
		\varphi_N[\mathrm{p}^{(k)}_{1}	]
		\rho_N [\mathrm{q}_{N,1} , \ldots , \widetilde{\mathrm{r}}_{N,k} , \ldots , \mathrm{q}_{N,n} ]	
\\ &
	-
		\varphi_N[\mathrm{p}^{(k)}_{1} ]
		\rho_N [\mathrm{q}_{N,1} 
				, \ldots ,	
		( \mathrm{p}^{(k)}_{m_k}  - \varphi_N[\mathrm{p}^{(k)}_{m_k} ] )
		\widetilde{\mathrm{r}}_{N,k}
		,\ldots, 
		\mathrm{q}_{N,n} ]  
\\ &
	-
		\varphi_N[\mathrm{p}^{(k)}_{m_k} ]
		\rho_N [\mathrm{q}_{N,1} 
		, \ldots ,	
		( \mathrm{p}^{(k)}_{1}  - \varphi_N[\mathrm{p}^{(k)}_{1} ] )
		\widetilde{\mathrm{r}}_{N,k}
		,\ldots, 
		\mathrm{q}_{N,n} ] 
\end{align*}
by linearity of $\rho_{N}$ in the  $k$-th entry. 
But then, by induction hypothesis, every element in the right hand side of the equality above is uniformly bounded with respect to $N$, and therefore, so is $\rho_N [ \mathrm{q}_{N,1} , \ldots  , \mathrm{q}_{N,n} ]$. 
\end{proof}
%
%
%
Having proved Proposition \ref{prop.asymptotic.centering.1} and Corollary \ref{corollary.asymptotic.centering}, 
we can now show that, under the hypothesis of Theorem \ref{thm.bounded.cumulants.property}, the family of random matrix ensembles $\{ \{ U^{}_{N,i} D^{}_{N,i} U^{*}_{N,i} \}_{N=1}^{\infty} \}_{i \in I}$ has the bounded cumulants property.

\begin{proof}[\textbf{Proof of Theorem \ref{thm.bounded.cumulants.property}}]
Let $\mathcal{A}$ denote the algebra of non-commutative polynomials $\mathbb{C}\left\langle \mathrm{x}_{i} \mid i \in I \right\rangle$, and let $\mathcal{A}_{i} \subset \mathcal{A} $ denote the algebra $\mathbb{C}\left[ x_{i} \right] $ for each index $i \in I$. 
For each integer $N\geq 1$, take $X_{N,i} =  U^{}_{N,i} D^{}_{N,i} U^{*}_{N,i} $  for every index $i \in I$ and let  $\varphi_{N}: \mathcal{A}\rightarrow \mathbb{C}$ be the unital linear map defined by
 (\ref{eqn.unital.funtionals}).
Note that if $\mathrm{p} \in \mathcal{A}_{j}$ for some $j \in I$, then  $\mathrm{p}\left( \{X_{N,i}\}_{i \in I} \right) = U^{}_{N,j} \mathrm{p}( D^{}_{N,j} ) U^{*}_{N,j} $, and hence, the limit 
$\lim_{N\rightarrow \infty} \varphi_{N}[\mathrm{p}]$
exists for every $\mathrm{p} \in \mathcal{A}_{i}$ and every $i \in I$. 
Now, suppose we are given polynomials $\mathrm{p}^{}_{1} \in \mathcal{A}_{i^{}_{1}}, \mathrm{p}^{}_{2} \in \mathcal{A}_{i^{}_{2}}, \ldots, \mathrm{p}^{}_{m} \in \mathcal{A}_{i^{}_{m}}$ with $i_{1},i_{2},\ldots,i_{m}\in I$ and $i^{}_{1} \neq i^{}_{2},  i^{}_{2} \neq i^{}_{3}, \ldots , i^{}_{m_{}-1} \neq i^{}_{m_{}} $, and $i^{}_{m} \neq i^{}_{1}$. 
Note that
\begin{align*}
N \varphi_{N}\left[ (\mathrm{p}_{1} - \varphi_{N}\left[\mathrm{p}_{1}\right]   ) (\mathrm{p}_{2} - \varphi_{N}\left[\mathrm{p}_{2}\right]   ) \cdots (\mathrm{p}_{m} - \varphi_{N}\left[\mathrm{p}_{m}\right]   ) \right]  
=
\exptr{\Tr{Y_{N}}}
\end{align*}
where
\begin{align*}
Y_{N} 
=
\big( U^{}_{N,i^{}_{1}} A^{}_{N,i^{}_{1}} U^{*}_{N,i^{}_{1}}  \big)	
\big( U^{}_{N,i^{}_{2}} A^{}_{N,i^{}_{2}} U^{*}_{N,i^{}_{2}}  \big)
\cdots
\big( U^{}_{N,i^{}_{m}} A^{}_{N,i^{}_{m}} U^{*}_{N,i^{}_{m}}  \big)
\end{align*}
and each $A^{}_{N,i^{}_{j}}$ is of trace zero and given by 
\begin{align*}
A^{}_{N,i^{}_{j}}
&	=
U^{*}_{N,i^{}_{j}} \left( 
\mathrm{p}^{}_{j} ( X_{N,i^{}_{j}} ) - \mathbb{E}[ \mathrm{tr} ( \mathrm{p}^{}_{j} ( X_{N,i^{}_{j}} )  ) ] I_{N} \right)
U^{}_{N,i^{}_{j}}
=
\mathrm{p}^{}_{j} ( D_{N,i^{}_{j}} ) -  \mathrm{tr} ( \mathrm{p}^{}_{j} ( D_{N,i^{}_{j}} )  ) I_{N}   .  
\end{align*}
Thus, by Lemma \ref{lemma.bounded.cumulant.property}, there is a constant $C$ depending only on the indexes $i^{}_{j}$ such that 
\begin{align*}
\abs*{	N^{}_{} \varphi_{N}\left[ (\mathrm{p}_{1} - \varphi_{N}\left[\mathrm{p}_{1}\right]   ) (\mathrm{p}_{2} - \varphi_{N}\left[\mathrm{p}_{2}\right]   ) \cdots (\mathrm{p}_{m} - \varphi_{N}\left[\mathrm{p}_{m}\right]   ) \right]  } 
\leq 
C  \prod_{j=1}^{m} \norm*{A^{}_{N,i_{j}}} . 
\end{align*}
But then, since $\sup_{N} \norm*{D_{N,i}} < \infty$ and $\lim_{N\rightarrow \infty}\mathrm{tr}( D^{k}_{N,i} )$ exists for every  $k\geq 1$ and any $i \in I$, we have $\sup_{N} \norm{A^{}_{N,i^{}_{j}}} < \infty$, and therefore
\begin{align}\label{eqn.no.centred.asymp.free.lim}
\lim_{N\rightarrow \infty}
\varphi_{N}\left[ (\mathrm{p}_{1} - \varphi_{N}\left[\mathrm{p}_{1}\right]   ) (\mathrm{p}_{2} - \varphi_{N}\left[\mathrm{p}_{2}\right]   ) \cdots (\mathrm{p}_{m} - \varphi_{N}\left[\mathrm{p}_{m}\right]   ) \right]
= 0 . 
\end{align}
Each linear functional $\varphi_{N}$ is tracial, i.e., $\varphi_{N}[pq] = \varphi_{N}[qp]$ for all $p,q \in \mathcal{A}$, and thus, following similar arguments to those in the proof Corollary \ref{corollary.asymptotic.centering}, we can remove the condition $i_{m} \neq i_{1}$ and still get (\ref{eqn.no.centred.asymp.free.lim}). 
Therefore, by Proposition \ref{prop.asymptotic.centering.1}, the random matrix ensembles $\{ U^{}_{N,i} D^{}_{N,i} U^{*}_{N,i} \}_{N=1}^{\infty}$ with $i \in I$ are asymptotically free, $ \lim_{N\rightarrow \infty} \varphi_{N}[ \mathrm{p}  ] $ exists for every $\mathrm{p} \in \mathcal{A} $, and (\ref{eqn.bnd.cumulants.prpty}) holds for $n = 1$. 
%
%
%


Fix now an arbitrary integer $n\geq 2$ and let $\rho_{N}: \mathcal{A}\times \cdots \times \mathcal{A} \rightarrow \mathbb{C}$ be the $n$-linear map given by 
 (\ref{eqn.n-linear.maps})  for every integer $N \geq 1$.   
Note that 
\begin{align*}
\rho_{N}(\mathrm{p}_{1},\mathrm{p}_{2}, \ldots, \mathrm{p}_{n})
&	=
\mathfrak{c}_{n}\left[ 
\Tr{\mathrm{p}_{1}( D^{}_{N,i_{1}} ) } , \ldots, 
\Tr{\mathrm{p}_{n}( D^{}_{N,i_{n}} ) } \right] 	
=
0
\end{align*}
%
if $\mathrm{p}_{1} \in \mathcal{A}_{i_{1}}, \mathrm{p}_{2} \in \mathcal{A}_{i_{2}}, \ldots, \mathrm{p}_{n} \in \mathcal{A}_{i_{n}}$ for some $i_{1},i_{2},\ldots,i_{n} \in I$. 
Thus, since $n$ is arbitrary, the family of ensembles $\{ \{ U^{}_{N,i} D^{}_{N,i} U^{*}_{N,i} \}_{N=1}^{\infty} \}_{i \in I}$ has the bounded cumulants property if the multi-linear functional $\rho_{N}$  satisfies \textit{(\ref{asymptotic.centering.on.the.go}')} from Corollary  \ref{corollary.asymptotic.centering}, namely, 
\[ \sup_{N} \abs*{\rho_N \left[\mathrm{q}_{N,1} , \mathrm{q}_{N,2} , \ldots , \mathrm{q}_{N,n} \right] } < \infty\]
whenever each $\mathrm{q}_{N,k}$ is of the form
\[\mathrm{q}_{N,k}
= 	
(\mathrm{p}^{(k)}_1 - \varphi_N[\mathrm{p}^{(k)}_1] )
(\mathrm{p}^{(k)}_2 - \varphi_N[\mathrm{p}^{(k)}_2]	)
\cdots	
(\mathrm{p}^{(k)}_{m_k}-\varphi_N [\mathrm{p}^{(k)}_{m_{k}}]) 
\]
with $\mathrm{p}^{(k)}_{j} \in \mathcal{A}_{i^{(k)}_{j}}$ and  $i^{(k)}_{1} \neq i^{(k)}_{2}, i^{(k)}_{2} \neq  i^{(k)}_{3}, \ldots, i^{(k)}_{m_{k}-1} \neq i^{(k)}_{m_{k}}$, and $i^{(k)}_{m_{k}} \neq i^{(k)}_{1}$.
Suppose $i^{(k)}_{j}$, $\mathrm{p}^{(k)}_{j}$, and $\mathrm{q}_{N,k}$ are as above and take $Y_{N,k}=\mathrm{q}_{N,k}\left( \{X_{N,i}\}_{i \in I} \right)$ for $k=1,2,\ldots, n$. 
Then, we have
\[
\rho_{N}  \left[ \mathrm{q}_{N,1} , \mathrm{q}_{N,2} , \ldots , \mathrm{q}_{N,n} \right]
=
\mathfrak{c}_{n}\left[ 
\Tr{Y^{}_{N,1}} , \Tr{Y^{}_{N,2}} ,\ldots,\Tr{Y^{}_{N,n}} \right]  . 
\]
Moreover, letting $A_{N,i^{(k)}_{j}} = 
\mathrm{p}^{(k)}_{j} ( D_{N,i^{(k)}_{j}} ) -  \mathrm{tr} ( \mathrm{p}^{(k)}_{j} ( D_{N,i^{(k)}_{j}} )  ) I_{N} $ for each $i^{(k)}_{j}$ and every $N \geq 1$, we get $A_{N,i^{(k)}_{j}}$ is of trace zero, $\sup_{N} \norm{A^{}_{N,i^{(k)}_{j}}} < \infty$, and 
\begin{align*}
Y_{N,k}
& =
\big(U^{}_{N,i^{(k)}_{1}} A^{}_{N,i^{(k)}_{1}} U^{*}_{N,i^{(k)}_{1}} \big)  
\big(
U^{}_{N,i^{(k)}_{2}} A^{}_{N,i^{(k)}_{2}} U^{*}_{N,i^{(k)}_{2}}   
\big)
\cdots
\big( U^{}_{N,i^{(k)}_{m_{k}}} A^{}_{N,i^{(k)}_{m_{k}}} U^{*}_{N,i^{(k)}_{m_{k}}}  \big) . 
\end{align*}
Therefore, by Lemma \ref{lemma.bounded.cumulant.property}, there is a constant $C$ depending only on the indexes $i^{(k)}_{j}$ such that 
\[
\abs*{\rho_{N}  \left[ \mathrm{q}_{N,1} , \mathrm{q}_{N,2} , \ldots , \mathrm{q}_{N,n} \right] } 
\leq 
C \prod_{k=1}^{n} \prod_{j=1}^{m_{k}} \norm*{A_{N,i^{(k)}_{j}}} 
< 
\infty . 
\]
\end{proof}
%
%
%

\section{Fluctuation moments}\label{sec.proof.main.thm.2}
%
%
The proofs of Theorem \ref{thm.second.order.asymptotics} and Theorem \ref{thm.second.order.freeness} are very similar, and thus, in order to avoid redundancies, this section is devoted to prove only Theorem \ref{thm.second.order.asymptotics}; nonetheless, what has to be modified to obtain the conclusions from Theorem \ref{thm.second.order.freeness} is pointed out in the next section. 

Let $X_{N,1}$ and $X_{N,2}$ be as in Theorem \ref{thm.second.order.asymptotics}. 
Assume $Y_{N}=Y_{N,1} Y_{N,2} \cdots Y_{N,2m_{1}}$ and $Z_{N} = Z_{N,1} Z_{N,2} \cdots Z_{N,2m_{2}}$ where $Y_{N,k}$ and $Z_{N,l}$ are given by \eqref{yz.centered.matrices} for some polynomials $\mathrm{p}_{1}, \mathrm{p}_{2},\ldots, \mathrm{p}_{2m_{1}}, \mathrm{q}_{1},   \mathrm{q}_{2}, \ldots , \mathrm{q}_{2m_{2}}\in \mathbb{C}[\mathrm{x}]$ and some indexes $i_1, i_2, \ldots, \allowbreak i_{2m_{_1}}, j_{1},j_{2},\ldots,j_{2m_{_2}} \in \{1,2\}$ satisfying $i_{1}=j_{1}$ and (\ref{cyclic.alternating.indexes}). 
Note that 
\[
Y_{N} = \big( U^{}_{N,i_{1}}A^{}_{N,1}U^{*}_{N,i_{1}} \big)
\big( U^{}_{N,i_{2}}A^{}_{N,2}U^{*}_{N,i_{2}} \big) \cdots
\big( U^{}_{N,i_{2m_{1}}}A^{}_{N,2m_{1}}U^{*}_{N,i_{2m_{1}}} \big)
\]
and
\[
Z_{N} =\big( U^{}_{N,j_{1}}B^{}_{N,1}U^{*}_{N,j_{1}} \big)
\big( U^{}_{N,j_{2}}B^{}_{N,2}U^{*}_{N,j_{2}} \big) \cdots
\big( U^{}_{N,j_{2m_{1}}}B^{}_{N,2m_{1}}U^{*}_{N,j_{2m_{1}}} \big)
\]
with $A_{N,k}$ and $B_{N,l}$ defined as in \eqref{ab.centered.matrices.alternative}; moreover, we have $ i_{2k-1} = j_{2l-1} = i_{1}  \neq i_{2} = i_{2k} = j_{2l}$. 
Thus, following similar arguments to those in the proof of Lemma \ref{lemma.bounded.cumulant.property}, we obtain
\begin{align}\label{eqn.cov.YN.ZN.1}
\covt{Y_{N}}{Z_{N}} 
=& 
\sum_{\theta \in P_{\chi}(\pm {2m})}
\left(\sum_{\substack{
		\pi \in P_{\text{even}}(\pm {2m}) \\ 
		\pi \leq \theta}} 
\mathfrak{c}_{2} \left[ \pi \right]	
\mu(\pi, \theta) \right) 
\sum_{\substack{
		\mathbf{j}:[\pm {2m}]\rightarrow[N] \\ \mathrm{ker}(\mathbf{j}) \geq \theta}}
\mathbf{a}(\mathbf{j}) 
\end{align}
where $P_{\chi}(\pm {2m})$ denotes the set of all even partitions of $[\pm {2m}]$ with no blocks of the form $\{k,-k\}$,  
$\mu:P(\pm {2m} ) \times P( \pm {2m})\rightarrow \mathbb{C}$ is the M\"{o}bius inversion function, 
$\mathbf{a(j)}$ is given by 
\begin{align*}
\mathbf{a(j)} = \prod_{k=1}^{2m_{1}} A_{N,k}(j_{-k},j_{+k}) \cdot 
\prod_{l=1}^{2m_{2}} B_{N,l}(j_{-2m_{1}-l},j_{+2m_{1}+l}), 
\end{align*}
for function each $\mathbf{j} : [\pm {2m} ] \rightarrow [N]$, and if $\mathbf{j} : [\pm {2m} ] \rightarrow [N]$ satisfies $\kernel{\mathbf{j}} = \pi $, then  
\begin{align}\label{eqn.normalized.fluct.c[pi]}
\mathfrak{c}_{2} \left[ \pi \right] & =
\cov{ \prod_{k =1}^{2m_{1}}	
	V_{k}(j_{\sigma(-k)},j_{\sigma(k)})}{
	\prod_{k = 	2m_{1}+1}^{2m_{1}+2m_{2}}						V_{k}(j_{\sigma(-k)},j_{\sigma(k)})}  
\end{align}
with $V^{}_{2k-1} =  V^{*}_{2k} = U^{*}_{N,{i}_{1}} U^{}_{N,{i}_{2}}  $ for $k = 1,2,\ldots, m$ and $\sigma:[\pm 2m] \rightarrow [\pm 2m]$ is the cyclic permutation given by 
\[
\sigma = (-1,1,-2,2,\ldots,-2m_{1},2m_{1})(-2m_{1}-1,2m_{1}+1,\ldots,-2m_{1}-2m_{2} ,2m_{1}+2m_{2}).
\]
%
%
%
%
%
%
%
It turns out that (\ref{eqn.cov.YN.ZN.1}) becomes 
\begin{align}\label{eqn.fluct.moments.c2.chi.chi.1}
\covt{Y_{N} }{Z_{N} } 
=& 
\sum_{\theta \in P_{\chi\chi}(\pm {2m})}
\left(\sum_{\substack{
		\pi \in P_{\text{even}}(\pm {2m}) \\ 
		\pi \leq \theta}} 
	\mathfrak{c}_{2} \left[ \pi \right]	
	\mu(\pi, \theta) \right)
\sum_{\substack{
		\mathbf{j}:[\pm {2m}]\rightarrow[N] \\ \mathrm{ker}(\mathbf{j}) \geq \theta}}
\mathbf{a}(\mathbf{j}) 
+ O(N^{-1}) 
\end{align}
where $P_{\chi\chi}(\pm {2m})$ denotes the set of all partitions $\theta \in P_{\chi}(\pm {2m})$ such that the graph sum exponent $\tau_{\theta}$, defined in Section \ref{sec.general.graph.sums}, equals $m$. 
Indeed, if we are given partitions $\pi \in P_{\text{even}}(\pm {2m})$ and $\theta \in P_{\chi}(\pm {2m}) $ satisfying $\theta \geq \pi$, then Theorem \ref{mingo-speicher} and Proposition \ref{prop.graph.sum.exp.bound1} imply 
\begin{align}\label{eqn.bound.fluct.moments.c2.chi.chi.graph.sum}
\abs*{
	\mathfrak{c}_{2} \left[ \pi \right]	
	\mu(\pi, \theta) 	
	\sum_{\substack{
			\mathbf{j}:[\pm 2m]\rightarrow[N] \\ \mathrm{ker}(\mathbf{j}) \geq \theta}}
	\mathbf{a}(\mathbf{j})  }
\leq 
\abs*{	\mathfrak{c}_{2} \left[ \pi \right]	}
\abs*{	\mu(\pi, \theta) }
N^{\tau_{\theta}} 
\prod_{k=1}^{2m_{1}} \norm*{ A_{N,k} }
\prod_{l=1}^{2m_{2}} \norm*{ B_{N,l} }
\end{align}
where $\tau_{\theta}$ is the number of connected components of the graph $\mathcal{G}_{\theta}$. 
Now, by hypothesis, $\sup_{N} \norm*{D_{N,i}} < \infty$ and $\lim_{N\rightarrow \infty}\mathrm{tr}( D^{k}_{N,i} )$ exists for every  $k\geq 1$ and any $i \in \{1,2\}$, so we have 
\[\sup_{N} \prod_{k=1}^{2m_{1}} \norm*{ A_{N,k} }
\prod_{l=1}^{2m_{2}} \norm*{ B_{N,l} } < \infty . \]
Moreover, every connected component of $\mathcal{G}_{\theta}$ contains at least two edges since $\theta$ is even and has no blocks of the form $\{-k,k\}$, and hence, the graph sum exponent $\tau_{\theta}$ satisfies 
\[
\tau_{\theta} \leq m .
\]
Additionally, since the unitary ensemble $\{\{U_{N,1},U_{N,2}\}\}_{N=1}^{\infty}$ satisfies 
\textit{(\ref{pnorms-uniformly-bounded})} from Lemma \ref{lemma.bounded.cumulant.property}, it follows from the proof of Lemma \ref{lemma.bounded.cumulant.property} that there is a constant $C_{2}$ independent from $N$ satisfying
\[
\abs*{\mathfrak{c}_{2} \left[ \pi \right] }  \leq C_{2} N^{-m}. 
\]
Therefore, from (\ref{eqn.bound.fluct.moments.c2.chi.chi.graph.sum}) we obtain  
\begin{align}
	\mathfrak{c}_{2} \left[ \pi \right]	
	\mu(\pi, \theta) 	
	\sum_{\substack{
			\mathbf{j}:[\pm 2m]\rightarrow[N] \\ \mathrm{ker}(\mathbf{j}) \geq \theta}}
	\mathbf{a}(\mathbf{j})  
=
O(N^{-1})
\end{align}
unless the graph sum exponent $\tau_{\theta} = m$, and, consequently, we get (\ref{eqn.fluct.moments.c2.chi.chi.1}). 

Note that the condition $\tau_{\theta} = m$, for an even partition $\theta\in P(\pm 2m)$ with no blocks of the form $\{+k,-k\}$, forces each component of the undirected graph ${\mathcal{G}}_{\theta}$ to have exactly two edges.
Thus, for any partition $\theta \in P_{\chi\chi}(\pm {2m})$, each component of the directed graph $\vec{\mathcal{G}}_{\theta}$ has one of the following forms:  
\begin{center}\includegraphics[scale=1]{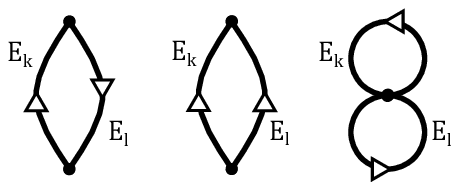}\end{center}
And therefore, as illustrated it at the end of Section \ref{sec.general.graph.sums}, each graph sum $\sum_{\substack{ \mathrm{ker}(\mathbf{j}) \geq \theta}}\mathbf{a}(\mathbf{j})$ appearing in (\ref{eqn.fluct.moments.c2.chi.chi.1})  can be written as a product of traces of matrices where each trace is of the followings forms: $\Trn{C_{N,k}C_{N,l}}$,  
$\Trn{C^{}_{N,k}C^{T}_{N,l}}$, or  $\Trn{C_{N,k} \circ C_{N,l}}$ where $C_{N,k}$ and $C_{N,l}$ belong to the set $\{A_{N,1},  \ldots, \allowbreak A_{N,2m_{1}},B_{N,1},\ldots,  B_{N,2m_{2}}\}$.
%
Hence, letting $\mathfrak{c}[\pi]$ be given by (\ref{eqn.normalized.fluct.c[pi]}) for each partition $\pi \in P(\pm 2m)$, the conclusions in Theorem \ref{thm.second.order.asymptotics} will follow from (\ref{eqn.fluct.moments.c2.chi.chi.1}) once we determine the order of 
\begin{align}\label{sum.c2.pi.mu.pi.theta}
\sum_{\substack{
		\pi \in P_{\text{even}}(\pm 2m)	\\
		\theta \geq \pi}} 
\mathfrak{c}_{2} \left[ \pi \right]	
\mu(\pi, \theta)  . 
\end{align}
Now, recall the values of M\"{o}bius inversion function are determined by (\ref{eqn.defining.mobius.func.r}), and given explicitly by (\ref{mobius.inversion.function}). 
Thus, to determine the order of (\ref{sum.c2.pi.mu.pi.theta}), it is enough to compute $\mathfrak{c}_{2}[\pi]$ for even partitions $\pi\in P(\pm 2m)$ satisfying  $\pi  \leq \theta$ for some another partition $\theta$ in the set $P_{\chi\chi}(\pm {2m} )$. 
%
%
%
%
\begin{proposition}\label{prop.fluc} 
	Suppose $\theta$ is a partition in $P_{\chi\chi}(\pm {2m} )$. 
	If  $\pi$ is an even partition such that $\pi \leq \theta$ and $\mathfrak{c}_{2}[\pi]$  is given by (\ref{eqn.normalized.fluct.c[pi]}), 
	then the following holds:
	\begin{enumerate}
		\item\label{prop.fluc.h} for $U^{*}_{N,{i}_{1}} U^{}_{N,{i}_{2}} = \frac{1}{\sqrt{N}} W^{*} H^{} W$, we have 
		\begin{align*}
		N^{m}\mathfrak{c}_{2}[ \pi ]  
		& =
		\left\{\begin{array}{cl}
		1 +  O\left( N^{-1}\right)  
		& \text{if there symmetric pairing partition } \hat\theta \leq \pi \text{ satisfying either (\ref{prop.minimal.special.partitions.1}) or 
			(\ref{prop.minimal.special.partitions.2})} 
		\\ & \text{from Proposition \ref{prop.minimal.special.partitions}} \\ 
		O\left( N^{-1/2}\right) 
		& \text{otherwise. } 
		\end{array} \right.
		\end{align*}	

		\item\label{prop.fluc.xh} for $U^{*}_{N,{i}_{1}} U^{}_{N,{i}_{2}} = \frac{1}{\sqrt{N}} W^{*} X H^{} W$, we obtain
		\begin{align*}
		N^{m}\mathfrak{c}_{2}[ \pi ]  
		& =
		\left\{\begin{array}{cl}
		1 +  O\left( N^{-1}\right)  
		& \text{if there symmetric pairing partition } \hat\theta \leq \pi \text{ satisfying (\ref{prop.minimal.special.partitions.1}) from Proposi-} 
		\\ & \text{tion \ref{prop.minimal.special.partitions}} \\ 
		1 +  O\left( N^{-1}\right)  
		& \text{if there symmetric pairing partition } \hat\theta \leq \pi \text{ satisfying (\ref{prop.minimal.special.partitions.2}) from Proposi-} 
		\\ & \text{tion \ref{prop.minimal.special.partitions} and the graph } \mathcal{G}_{\pi} \text{ has only double-loops as components} \\
		O\left( N^{-1/2}\right) 
		& \text{otherwise. } 
		\end{array} \right.
		\end{align*}	
		\item\label{prop.fluc.xhx} for $U^{*}_{N,{i}_{1}} U^{}_{N,{i}_{2}} = \frac{1}{N} W^{*} H^{*}X H^{} W$, we get
		\begin{align*}
		N^{m}\mathfrak{c}_{2}[ \pi ]  
		& =
		\left\{\begin{array}{cl}
		1 +  O\left( N^{-1}\right)  
		& \text{if } m_{1} = m_{2} \text{ and there are integers } 1 \leq k \leq 2m_{1} < l \leq 2m_{1} + 2m_{2}	\text{ so} 	\\ &  				
		\text{that } \sigma^{t}(-k) \sim_{\pi} \sigma^{-t}(l) \text{ for every integer } t \geq 0	\\ 	
		2 +  O\left( N^{-1}\right)  
		&\text{if there are integers } 1 \leq k \leq 2m_{1} < l \leq 2m_{1} + 2m_{2} \text{ so that for each} \\ & 
		\text{integer } t \geq 0  \text{ we have } \sigma^{-t}(k)	\sim_{\pi} 	\sigma^{t+1}(k)	\text{ and } 
		\sigma^{-t}(l)	\sim_{\pi}	\sigma^{t+1}(l)  \\
		O\left( N^{-1/2}\right) 
		& \text{otherwise. } 
		\end{array} \right.
		\end{align*}
	\end{enumerate}
\end{proposition}
%
%
%
The proof of Proposition \ref{prop.fluc} is based on the expected value of products of entries from $X$ and $W$, see relations (\ref{dist.entries.signature.matrix}) and (\ref{dist.entries.signed.perm.matrix}), and  the results on graph sums of the Discrete Fourier Transform from Section \ref{subsec.graph.sums.dft}, however, it requires some technical intermediate steps, so we will omit it for now and leave it to the end of this section. 
Nonetheless, the computation of (\ref{sum.c2.pi.mu.pi.theta}) up to a term of order $N^{-m-1/2}$ is quite simple assuming Proposition \ref{prop.fluc} holds. 
%
%
%
\begin{lemma}\label{lemma.fluc.mobius}
	Suppose $\theta$ is a partition in $P_{\chi\chi}(\pm {2m} )$ and let $\mathfrak{c}[\pi]$  be given by (\ref{eqn.normalized.fluct.c[pi]}) for each partition $\pi \in P(\pm 2m)$.    
	Then the following holds: 
	\begin{enumerate}
		\item for $ U^{*}_{N,{i}_{1}} U^{}_{N,{i}_{2}}  = \frac{1}{\sqrt{N}} W^{*} H W$, we have 	
		\begin{align*}
		\sum_{\substack{
				\pi \in P_{\text{even}}(\pm {2m})	\\
				\pi \leq \theta }} 
		N^{m}
		\mathfrak{c}_{2} \left[ \pi \right]	
		\mu(\pi, \theta)
		=&
		\left\{\begin{array}{cl}
	1  + 	O\left( N^{-1/2} \right)& 			
		\text{if  } \theta \text{ is a pairing partition satisfying either  (\ref{prop.minimal.special.partitions.1}) or} \\
		& \text{(\ref{prop.minimal.special.partitions.2}) from  Proposition \ref{prop.minimal.special.partitions}}, \\
	O\left( N^{-1/2}\right)	& 
		\text{ otherwise } 
		\end{array}\right.
		\end{align*}

		\item for $ U^{*}_{N,{i}_{1}} U^{}_{N,{i}_{2}}  = \frac{1}{\sqrt{N}} W^{*} XH W^{*}$, we obtain 
		\begin{align*}
		\sum_{\substack{
				\pi \in P_{\text{even}}(\pm 2m)	\\
				\pi \leq \theta }} 
		N^{m}
		\mathfrak{c}_{2} \left[ \pi \right]	
		\mu(\pi, \theta)
		=&
	\left\{\begin{array}{cl}
	1	+	O\left( N^{-1/2} \right)  & 			
		\text{if  }  
		\theta \text{ is a pairing partition satisfying 			(\ref{prop.minimal.special.partitions.1})  from Pro-} \\ & \text{position \ref{prop.minimal.special.partitions}}, \\
	1	+	O\left( N^{-1/2} \right) & 
		\text{if there exists a pairing partition } \hat\theta \leq \theta  \text{ satisfying} \\ 
		& \text{(\ref{prop.minimal.special.partitions.2}) from  Proposition \ref{prop.minimal.special.partitions} and } \theta  \text{ has only blocks of} \\ 
		& \text{the form } \{k,-k,l,-l\}, \\ 
	O\left( N^{-1/2}\right)	& 
		\text{otherwise } 
	\end{array}\right.
		\end{align*}

		\item for $ U^{*}_{N,{i}_{1}} U^{}_{N,{i}_{2}}  = \frac{1}{N}  W^{*} H^{*}XH  W^{}$, we get 
		\begin{align*}
		\sum_{\substack{
				\pi \in P_{\text{even}}(\pm 2m)	\\
				\pi \leq \theta }} 
		N^{m}
		\mathfrak{c}_{2} \left[ \pi \right]	
		\mu(\pi, \theta) 
		& =	\left\{\begin{array}{cl}
1	+	O\left( N^{-1/2} \right) 
		& \text{if } m_{1} = m_{2} \text{ and there is an integer }  \\ 
		& 2m_{1}+1 \leq l \leq  2m_{1} + 2m_{2} \text{ such that } \\
		& \theta= \{ \sigma^{t}(-1) , \sigma^{-t-1}(-l) \} \mid t \geq 0 
		\} ,  	\\  	
2	+	O\left( N^{-1/2} \right)  
		& \text{if there are integers }  1 \leq l_{1} \leq 2m_{1} \text{ and } \\
		& 2m_{1}+1 \leq l_{2} \leq  2m_{1} + 2m_{2} \text{ so that }  \theta= \theta_{1} \sqcup \theta_{2} \\ 
		& \text{where } \theta_{1}=\{ \{ \sigma^{t}(-l_{1}) , \sigma^{-t-1}(-l_{1}) \}  \mid t \geq 0 \}  \\ 
		&  \theta_{2}=\{ \{ \sigma^{t}(-l_{2}) , \sigma^{-t-1}(-l_{2}) \} \mid t \geq 0 \} ,	\\
	O\left( N^{-1/2}\right)  
		& \text{otherwise. } 
		\end{array} \right.
		\end{align*}
	\end{enumerate}
\end{lemma}
%
%
%
%
\begin{proof} 
	Suppose $U^{*}_{N,{i}_{1}} U^{}_{N,{i}_{2}}=W^{*}XHW/\sqrt{N}$. Proposition \ref{prop.minimal.special.partitions} and Proposition \ref{prop.fluc} imply that
	\begin{align*}
	\sum_{\substack{	\pi \in P_{\text{even}}(\pm 2m)	\\
			\pi \leq \theta }} 
	N^{m}\mathfrak{c}_{2} \left[ \pi \right]	
	\mu(\pi, \theta)
	=
	O\left( N^{-1/2}\right)					
	\end{align*}
	unless there are integers $ 1\leq k \leq 2m_{1} < l \leq 2m_{1} + 2m_{2}$ satisfying one of the following: 
	\begin{enumerate}[(i)]
		\item\label{lemma.fluc.mobius.1.case1} $k+l$ is even and $\sigma^{t}(-k) \sim_{\theta} \sigma^{-t}(l)$  for all integers  $t \geq 0$ or
		\item\label{lemma.fluc.mobius.1.case2} $k+l$ is odd, $\sigma^{t}(-k) \sim_{\theta} \sigma^{t}(-l)$  for all integers  $t \geq 0$ and $\mathcal{G}_{\theta}$ has only double-loops as	 components
	\end{enumerate} 
	Assuming (\ref{lemma.fluc.mobius.1.case1}) above holds, consider the pairing partition 
	$\hat\theta = \{ \{ \sigma^{t}(-k) , \sigma^{-t}(l) \}\mid t \geq 0\}$ 
	and note that Proposition \ref{prop.fluc} implies that 
	\begin{align}\label{cumulant.reduction}
	\sum_{\substack{
			\pi \in P_{\text{even}}(\pm 2m)	\\	\pi \leq \theta }} 
	N^{m}\mathfrak{c}_{2} \left[ \pi \right]	\mu(\pi, \theta)
	& =
	\sum_{\substack{	\pi \in P_{\text{even}}(\pm 2m)	\\
			\hat\theta \leq \pi \leq \theta }} 
	N^{m}\mathfrak{c}_{2} \left[ \pi \right]	
	\mu(\pi, \theta)
	+
	O\left( N^{-1/2}\right) .
	\end{align}
	Moreover, since $N^{m}\mathfrak{c}_{2} \left[ \pi \right] = 1 + O(N^{-1})$ for any partition $\pi$ satisfying $\hat\theta \leq \pi \leq \theta$, we get
	\begin{align*}
	\sum_{\substack{
			\pi \in P_{\text{even}}(\pm 2m)	\\	\pi \leq \theta }} 
	N^{m}\mathfrak{c}_{2} \left[ \pi \right]	\mu(\pi, \theta)
	& =
	\sum_{\substack{	\pi \in P_{\text{even}}(\pm 2m)	\\
			\hat\theta \leq \pi \leq \theta }} 
	\mu(\pi, \theta)
	+
	O ( N^{-1/2} )
	=
	\left\{\begin{array}{cl}
	1 + O\left( N^{-1/2}\right) & \text{if } \hat\theta = \theta \\
	O\left( N^{-1/2}\right)  & \text{if } \hat\theta < \theta
	\end{array}\right.
	\end{align*}
	from equations in (\ref{eqn.defining.mobius.func.r}) defining the M\"{o}bius inversion function.  
	On the other hand, if (\ref{lemma.fluc.mobius.1.case2}) above holds, consider $\hat\theta = \{ \{ \sigma^{t}(-k) , \sigma^{t}(-l) \}\mid t \geq 0\}$ instead and note that (\ref{cumulant.reduction}) above holds also in this case. 
	Hence, since $N^{m}\mathfrak{c}_{2} \left[ \theta \right] = 1 + O(N^{-1})$ and $N^{m}\mathfrak{c}_{2} \left[ \pi \right] = O(N^{-1/2})$ for any partition $\hat\theta \leq \pi < \theta$, we obtain
	\begin{align*}
	\sum_{\substack{
			\pi \in P_{\text{even}}(\pm 2m)	\\	\pi \leq \theta }} 
	N^{m}\mathfrak{c}_{2} \left[ \pi \right]	\mu(\pi, \theta)
	=
	N^{m}\mathfrak{c}_{2} \left[ \theta \right]	
	\mu(\theta, \theta)	
	+
	\sum_{\substack{	\pi \in P_{\text{even}}(\pm 2m)	\\
			\hat\theta \leq \pi < \theta }} 
	N^{m}\mathfrak{c}_{2} \left[ \pi \right]	
	\mu(\pi, \theta)
	=
	1 +	O( N^{-1/2})	
	\end{align*}
	The other cases, namely, $U^{*}_{N,{i}_{1}} U^{}_{N,{i}_{2}} = W^{*}HW /\sqrt{N} $ and $U^{*}_{N,{i}_{1}} U^{}_{N,{i}_{2}} = W^{*} H^{*}XH W/N$, are proved in the same way, one chooses a suitable pairing partition $\hat\theta$ such that  (\ref{cumulant.reduction}) holds, and then the corresponding conclusion follows from Proposition \ref{prop.fluc} and the equations in  (\ref{eqn.defining.mobius.func.r}) defining the M\"{o}bius function. 
\end{proof}
%
%
%
%
%
%
%
As mentioned earlier, the proof of Theorem \ref{thm.second.order.asymptotics} is complete once we apply Lemma \ref{lemma.fluc.mobius} to the relation  (\ref{eqn.fluct.moments.c2.chi.chi.1}). 
For instance, suppose $ U^{*}_{N,{i}_{1}} U^{}_{N,{i}_{2}}  = W^{*} H^{*}XH  W^{} / N$. 
Then, Theorem \ref{mingo-speicher}, Lemma \ref{lemma.fluc.mobius}, and (\ref{eqn.fluct.moments.c2.chi.chi.1}) imply that 
\begin{align}\label{rel.cov.YN.ZN.last}
\covt{Y_{N} }{Z_{N} } 
= &  
\delta_{m_{1},m_{2}}
\sum_{ \theta \in \mathcal{P}_{1} }
\sum_{\substack{
		\mathbf{j}:[\pm {2m}]\rightarrow[N] 
		\\ \mathrm{ker}(\mathbf{j}) \geq \theta}}
\mathbf{a}(\mathbf{j}) 
 + 
\sum_{\theta \in \mathcal{P}_{2} }
2 \sum_{\substack{
		\mathbf{j}:[\pm {2m}]\rightarrow[N] 
		\\ \mathrm{ker}(\mathbf{j}) \geq \theta}} 
\mathbf{a}(\mathbf{j})
%
+ 
O\left(N^{-\frac{1}{2}}\right) 
\end{align}
where $\mathcal{P}_{1}$ and $\mathcal{P}_{2}$ are subsets of $P_{\chi\chi}(\pm {2m})$ given by 
\[
\mathcal{P}_{1} = \{
\{ \{ \sigma^{t}(-1) , \sigma^{-t-1}(-l) \} \mid t = 0,1, 2, \ldots, 4m_{}  \}  \mid  l \in [2m] \setminus [2m_{1}]  \}
\]  
and 
\[
\mathcal{P}_{2} = \{	\{ \{ \sigma^{t}(-l_{1}) , \sigma^{-t-1}(-l_{1}) \} , 
\{ \sigma^{t}(-l_{2}) , \sigma^{-t-1}(l_{2}) \} \mid t \geq 0 \}	\mid  l_{1} \in [2m_{1}] ,  l_{2} \in [2m] \setminus [2m_{1}] \}. 
\]  
Now, note the set $\mathcal{P}_{1}$ has cardinality $2m_{1}$ provided $m_{1} = m_{2}$. Moreover,  $m_{1} = m_{2}$ implies that a partition $\theta \in P(\pm 2m)$ belongs to the set $\mathcal{P}_{1}$ if and only if for some integer $1\leq l \leq 2m_{1}$ the directed graph $\vec{\mathcal{G}}_{\theta}$ can be represented as 
\begin{center}\includegraphics[scale=1]{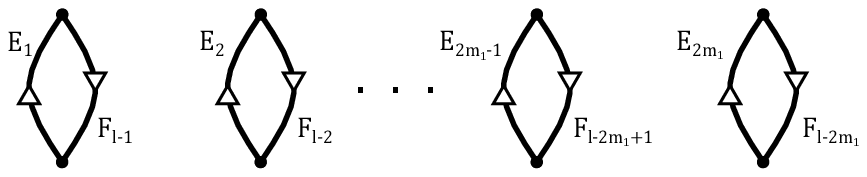}\end{center}
where $F_{k}$ denotes the edge $	E_{2m_{1}+k}$ and $l-k$ is taken module $2m_{1}$ for $k=1,2,\ldots, 2m_{1}$. 
Thus, for each integer $1\leq l \leq 2m_{1}$, there exists a unique $\theta  \in  \mathcal{P}_{1}$ so that  
\begin{align*}
\sum_{\substack{\mathbf{j}:[\pm {m}]\rightarrow[N] \\ 
		\mathrm{ker}(\mathbf{j}) \geq \theta}}
\mathbf{a}(\mathbf{j}) 
%
=\prod_{k=1}^{2m_{1}} \Tr{A_{N,k} B_{N,l-k}}
, 
\end{align*}
and hence, we obtain 
\begin{align*}
\covt{ Y_{N} }{ Z_{N} }  
=& 
 \delta_{m_{1},m_{2}} \sum_{l=1}^{2m_{1}} 
\left( \prod_{k=1}^{2m_{1}} \tr{A_{N,k}B_{N,l-k}} \right)
+ 
\sum_{\theta \in \mathcal{P}_{2} }
2 \sum_{\substack{
		\mathbf{j}:[\pm {2m}]\rightarrow[N] 
		\\ \mathrm{ker}(\mathbf{j}) \geq \theta}} 
\mathbf{a}(\mathbf{j})
+
	O \left( N^{-\frac{1}{2	}}\right) . 
\end{align*}
On the other hand, the set $\mathcal{P}_{2}$ has cardinality $m_{1} \cdot m_{2}$ since 
\[ 	
		\{ \{ \sigma^{t}(-l_{1}) , \sigma^{-t-1}(-l_{1}) \} \mid t \geq 0 \}  
	=
		\{ \{ \sigma^{t}(-m_{1}-l_{1}) , \sigma^{-t-1}(-m_{1}-l_{1}) \} \mid t \geq 0 \} 	
\]
and
\[ 	
\{ \{ \sigma^{t}(-l_{2}) , \sigma^{-t-1}(-l_{2}) \} \mid t \geq 0 \}  
=
\{ \{ \sigma^{t}(-m_{2}-l_{2}) , \sigma^{-t-1}(-m_{2}-l_{2}) \} \mid t \geq 0 \} 	
\]
for $l_{1}=1,2,\ldots, m_{1}$ and $l_{2}=1,2,\ldots, m_{2}$. Moreover, a partition $\theta \in P(\pm 2m)$ belongs to the set $\mathcal{P}_{2}$ if and only if for some integers $1\leq l \leq m_{1}$ and  $1\leq l_{2} \leq m_{2}$ the directed graph $\vec{\mathcal{G}}_{\theta}$ can be represented as 
\begin{center}\includegraphics[scale=1]{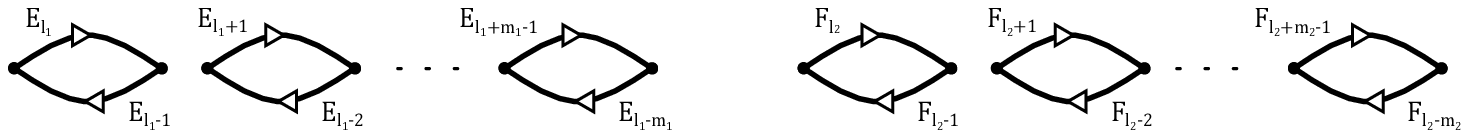}\end{center}
where $l_{1}-k_{1}$ and $l_{2}-k_{2}$ are taken modulo $2m_{1}$ and $2m_{2}$, respectively, for $k_{1}=1,2,\ldots,m_{1}$ and $k_{2} = 1,2,\ldots,m_{2}$. 
Thus, for each partition $\theta  \in  \mathcal{P}_{2}$ there are integers $1 \leq l_{1} \leq m_{1}$ and $1 \leq l_{2} \leq m_{2}$ satisfying   
\begin{align*}
\sum_{\substack{\mathbf{j}:[\pm {m}]\rightarrow[N] \\ 
		\mathrm{ker}(\mathbf{j}) \geq \theta}}
\mathbf{a}(\mathbf{j})
%
=& \prod_{k_{1}=1}^{m_{1}} \Tr{ A_{N,l_{1}+k_{1}-1} A_{N,l_{1}-k_{1}}} \cdot \prod_{k_{2}=1}^{m_{2}} \Tr{B_{N,l_{2}+k_{2}-1} B_{N,l_{2}-k_{2}}}	.
\end{align*}
Therefore, we have 
\begin{align*}
\covt{ Y_{N} }{ Z_{N} }  
=& 
2
\sum_{l_{1}=1}^{m_{1}} \sum_{l_{2}=1}^{m_{2}} 
\prod_{k_{1}=1}^{m_{1}} 
\tr{A_{N,l_{1}+k_{1}-1} A_{N,l_{1}-k_{1}}}
\prod_{k_{2}=1}^{m_{2}}		
\tr{B_{N,l_{2}+k_{2}-1} B_{N,l_{2}-k_{2}}}
\\	&		
+ \delta_{m_{1},m_{2}} \sum_{l=1}^{2m_{1}} 
\left( \prod_{k=1}^{2m_{1}} \tr{A_{N,k}B_{N,l-k}} \right)	
+
O ( N^{ -1/2 } ) . 
\end{align*}
The other cases, namely, $U^{*}_{N,{i}_{1}} U^{}_{N,{i}_{2}} = W^{*}HW /\sqrt{N} $ and $U^{*}_{N,{i}_{1}} U^{}_{N,{i}_{2}} = W^{*} XH W/N$, are proved in the same way, applying Lemma \ref{lemma.fluc.mobius} to the relation  (\ref{eqn.fluct.moments.c2.chi.chi.1}) we obtain similar relations to that in (\ref{rel.cov.YN.ZN.last}) that lead to (\ref{thm.second.order.asymptotics.1}) and (\ref{thm.second.order.asymptotics.2}) in Theorem \ref{thm.second.order.asymptotics}.

%
%
%
%
The remaining of this section is devoted to the proof of Proposition \ref{prop.fluc}. 
For clarity, we have considered two cases: 	$ U^{*}_{N,{i}_{1}} U^{}_{N,{i}_{2}} = W^{*} H^{*} X H^{}  W / N $  and $ U^{*}_{N,{i}_{1}} U^{}_{N,{i}_{2}} =  W^{*} Y H^{}  W /\sqrt{N} $ where $Y$ is either the identity matrix $I_{N}$ or an $N$-by-$N$ uniformly distributed signature matrix $X$. 
But first, let us introduce some more notation for partitions.  

Given a partition $\pi \in P(\pm 2m)$, we let $\pi_{\text{even}}$ and $\pi_{\text{odd}}$ denote the restriction of $\pi$ to the sets  
$\{ k \in [\pm 2m] \mid k \text{ is even}  \}$ and $\{ k \in [\pm 2m] \mid k \text{ is odd}  \}$, respectively. 
Moreover, we let $\pi^{\text{even}}$ and $\pi^{\text{odd}}$ denote the partitions of $\{ k \in [\pm 4m] \mid k \text{ is even }  \} $ and $\{ k \in [\pm 4m] \mid k \text{ is odd} \}$, respectively, given by $\pi^{\text{even}} = \{\{ 2k \mid k \in B \} \mid B \in \pi \}$ and 
$\pi^{\text{odd}} = \{\{ 2k-\sign{(k)} \mid k \in B \} \mid B \in \pi \}$ where $\sign(k) = 1$, if $k$ is positive, and $\sign(k) = -1$,  otherwise. 
For instance, if $\pi$ is the partition in $P(\pm 6)$ given by 
\[\pi=\{\{-1,4\},\{1,-3\},\{-2,3\},\{2,-4\}  \}\]
then 
\begin{align*}
\pi_{\text{even}} & =  \{ \{4\}, \{-2\},\{2,-4\} \}, 
	& \pi_{\text{odd}} =& \{\{-1\}, \{1,-3\}, \{3\} \},\\ 
\pi^{\text{even}}  & =  \{\{-2,8\},\{2,-6\},\{-4,6\},\{4,-8\} \},
\quad \text{and} \quad 
		& \pi^{\text{odd}} = &\{\{-1,7\},\{1,-5\},\{-3,5\},\{3,-7\} \}. 
\end{align*}

\subsection*{Case $\bm{ U^{*}_{N,{i}_{1}} U^{}_{N,{i}_{2}}=\frac{1}{\sqrt{N}} W^{*} Y H^{}  W}$ } 
%
%
%
%
%
Let $Y$ be an $N$-by-$N$ diagonal random matrix independent from $W$.  
Given a function $\mathbf{i}: [\pm 2m] \rightarrow [N]$, we let 
$$\mathbf{h}(\mathbf{i}) = \mathbf{h}_{1}(\mathbf{i}) \mathbf{h}_{2}(\mathbf{i}) \quad  \text{and} \quad 
\mathbf{y}(\mathbf{i}) = \mathbf{y}_{1}(\mathbf{i}) \mathbf{y}_{2}(\mathbf{i})
$$ where  $\mathbf{h_{1}}(\mathbf{i})$, $\mathbf{h_{2}}(\mathbf{i})$, $\mathbf{y_{1}}(\mathbf{i})$, and $\mathbf{y_{2}}(\mathbf{i})$ are given by 
\begin{align*}
%
%
%
\mathbf{h_{1}}(\mathbf{i}) &=
\prod_{k=1}^{m_{1}} 
H(i_{-2k+1},i_{2k-1}) H^{*}(i_{-2k},i_{2k}) , 
&\mathbf{h}_{2}(\mathbf{i})=&
\prod_{k=m_{1}+1}^{m_{1}+m_{2}}
H(i_{-2k+1},i_{2k-1}) H^{*}(i_{-2k},i_{2k}) , \\ 
%
%
%
%
\mathbf{y}_{1}(\mathbf{i}) &=
\prod_{k=1}^{m_{1}} 
Y(i_{-2k+1},i_{-2k+1}) Y(i_{2k-1},i_{2k-1}), 
&\mathbf{y}_{2}(\mathbf{i})=&
\prod_{k=m_{1}+1}^{m_{1}+m_{2}} 
Y(i_{-2k+1},i_{-2k+1}) Y(i_{2k-1},i_{2k-1});
\end{align*}
additionally, if we are given a function $\mathbf{j}: [\pm 2m] \rightarrow [N]$,  we put $$\mathbf{w}(\mathbf{i,j}) = \mathbf{w}_{1}(\mathbf{i,j}) \mathbf{w}_{2}(\mathbf{i,j})$$ where $\mathbf{w}_{1}(\mathbf{i,j})$ and  $\mathbf{w}_{2}(\mathbf{i,j})$ are given by 
%
%
\begin{align*}
\mathbf{w}_{1}(\mathbf{i,j}) &=
\prod_{k=1}^{2m_{1}} 
W(i_{k},j_{k}) W(i_{-k},j_{-k})
\quad \text{ and }
&\mathbf{w}_{2}(\mathbf{i,j})=&
\prod_{k=2m_{1}+1}^{2m_{1}+2m_{2}} 
W(i_{k},j_{k}) W(i_{-k},j_{-k}). 
\end{align*}
Now, for every partition $\pi \in P_{}(\pm {2m})$ and any function $\mathbf{j}:[\pm {2m}] \rightarrow [N]$ satisfying $ \kernel{\mathbf{j}} = \pi$, we define $\mathfrak{C}_{2} \left[ \pi \right]$ by
%
%
\begin{align}\label{eqn.unnormalized.flucts}
& 
\sum_{\substack{ \mathbf{i} : [\pm 2m] \rightarrow [N] \\ \kernel{\mathbf{i}}=\pi}} 
\mathbf{h}(\mathbf{i} \circ \bm{\sigma} ) 
\exptr{\mathbf{w}(\mathbf{i,j})}\exptr{\mathbf{y}(\mathbf{i})}
-
\sum_{\substack{ \mathbf{i} : [\pm 2m] \rightarrow [N] \\  \kernel{\mathbf{i}}=\pi_{1}\sqcup\pi_{2}  }} 
\mathbf{h}(\mathbf{i}  \circ \bm{\sigma}  ) 
\exptr{\mathbf{w}_{1}(\mathbf{i,j})}
\exptr{\mathbf{w}_{2}(\mathbf{i,j})}
\exptr{\mathbf{y}_{1}(\mathbf{i})}
\exptr{\mathbf{y}_{2}(\mathbf{i})}  \nonumber \\
& = 
\mathfrak{C}_{2} \left[ \pi \right] 
\end{align}
where ${\pi}_1$ and ${\pi}_2$ denote the restrictions of ${\pi}$ to $[\pm 2m_1]$ and $[\pm 2m ] \setminus [\pm 2m_1]$, respectively, and $\sigma$ is the cycle permutation given by 
\[
\sigma = (-1,1,-2,2,\ldots,-2m_{1},2m_{1})(-2m_{1}-1,2m_{1}+1,\ldots,-2m_{1}-2m_{2}, 2m_{1}+2m_{2}).
\]
%
%
%
%
%
\begin{proposition}\label{prop.cpi-and-Cpi}
	Let $\pi$ be an even partition in $P(\pm {2m})$. 
	Suppose $U^{*}_{N,{i}_{1}} U^{}_{N,{i}_{2}}  = W^{*} Y H^{} W / \sqrt{N}$ where $Y$ is an $N$-by-$N$ diagonal matrix independent from $W$ so that each entry $Y(i,i)$ takes values in the set $\{-1,1\}$.   
	If $\mathfrak{c}_{2}[\pi]$ and $\mathfrak{C}_{2} \left[ \pi \right]$ are given by (\ref{eqn.normalized.fluct.c[pi]}) and (\ref{eqn.unnormalized.flucts}), respectively, then 
	\begin{align}\label{eqn.relating.c[pi].C[pi]}
	N^{m}\mathfrak{c}_{2} \left[ \pi \right]  
	=
	\mathfrak{C}_{2} \left[ \pi \right] 	+	O(N^{-1})
	\end{align}
\end{proposition}
%
%
\begin{proof}
	Fix a function $\mathbf{j}:[\pm {2m}] \rightarrow [N]$ satisfying  $\kernel{\mathbf{j}} = \pi$.
	Note that the $(j_{-2k+1},j_{2k-1})$-entry of $U^{*}_{N,{i}_{1}} U^{}_{N,{i}_{2}}$ and the $(j_{-2k},j_{2k})$-entry of $U^{*}_{N,{i}_{2}} U^{}_{N,{i}_{1}}$ are given by
	\begin{align*}
	\sum_{i_{-2k+1},i_{2k+1}=1}^{N} 
	\frac{1}{\sqrt{N}}  W^{*}(j_{-2k+1},i_{-2k+1}) Y (i_{-2k+1},i_{-2k+1}) H^{}(i_{-2k+1},i_{2k-1})  W(i_{2k-1},j_{2k-1})  
	\end{align*}
	and
	\begin{align*}
	\sum_{i_{-2k},i_{2k}=1}^{N} 
	\frac{1}{\sqrt{N}}	W^{*}(j_{-2k},i_{-2k}) H^{*} (i_{-2k},i_{2k}) Y(i_{2k},i_{2k}),  W(i_{2k},j_{2k}),  
	\end{align*}
	respectively. 
	Thus, from (\ref{eqn.normalized.fluct.c[pi]}) and the linearity of the covariance, we have that
	\begin{align*}
	N^{m} \mathfrak{c}_{2} \left[ \pi \right]
	\quad =	
	\sum_{\mathbf{i}:[\pm {2m}]\rightarrow[N]}
	{\mathbf{h}}(\mathbf{i}) \cdot
	\text{cov} \big[ {\mathbf{w}}_{1}\mathbf{(i, j \circ \bm{{\sigma}})} \cdot \mathbf{y}_{1}\mathbf{(i)} , \mathbf{w}_{2}\mathbf{(i, j \circ \bm{{\sigma}})} \cdot \mathbf{y}_{2}\mathbf{(i )} \big]
	\end{align*} 
	where ${\mathbf{h}}(\mathbf{i})$, ${\mathbf{w}}_{1}\mathbf{(i,j)}$, ${\mathbf{w}}_{2}\mathbf{(i,j)}$, $\mathbf{y}_{1}\mathbf{(i)}$, $\mathbf{y}_{2}\mathbf{(i)}$, and $\bm{\sigma}$ are defined as above. 
	But, for every function $\mathbf{i}:[\pm {2m}]\rightarrow[N]$, we have that
	\[
	{\mathbf{w}}_{k}\mathbf{(i,j)} =
	{\mathbf{w}}_{k}\mathbf{(i\circ \bm{ {\sigma}}, j \circ \bm{{\sigma}})}  
	\quad \text{ and } \quad
	\mathbf{y}_{k}\mathbf{(i)} =	\mathbf{y}_{k}\mathbf{(i\circ \bm{{\sigma}})}
	\quad \text{ for }	\quad k=1,2,
	\]
	so we obtain
	\begin{align}\label{eqn.c2[pi].1}
	N^{m}\mathfrak{c}_{2} \left[ \pi \right]
	\quad =	
	\sum_{\mathbf{i}:[\pm {2m}]\rightarrow[N]}
	{\mathbf{h}}(\mathbf{i}\circ \bm{{\sigma}}) \cdot
	\cov{ {\mathbf{w}}_{1}\mathbf{(i, j )} \cdot \mathbf{y}_{1}\mathbf{(i)} }{\mathbf{w}_{2}\mathbf{(i, j )} \cdot \mathbf{y}_{2}\mathbf{(i )} } .
	\end{align}
	Moreover, from (\ref{dist.entries.signed.perm.matrix}) we get that 
	\[
	\exptr{{\mathbf{w}}_{1}(\mathbf{i},\mathbf{j}) {\mathbf{w}}_{2}(\mathbf{i},\mathbf{j})}= 0
	\quad  \text{and} \quad
	\exptr{{\mathbf{w}}_{1}(\mathbf{i},\mathbf{j})}\exptr{{\mathbf{w}}_{2}(\mathbf{i},\mathbf{j})}=0
	\] 
	provided  $\kernel{\mathbf{i}} \neq \pi$ and $\kernel{\mathbf{i}} \not\geq \pi_{1} \sqcup \pi_{2} $, respectively. 
	And hence, equality in (\ref{eqn.c2[pi].1}) becomes 
	\begin{align*}
	N^{m}\mathfrak{c}_{2} \left[ \pi \right]
	\quad = & 
	\sum_{\substack{ 	\mathbf{i}:[\pm {2m}]\rightarrow[N] \\ 
			\mathrm{ker}(\mathbf{i}) = \pi}}
	\mathbf{h}(\mathbf{i \circ \bm{\sigma}})
	\exptr{\mathbf{w}_{1}(\mathbf{i,j}) \mathbf{w}_{2}(\mathbf{i,j})} \exptr{\mathbf{y}_{1}(\mathbf{i})\mathbf{y}_{2}(\mathbf{i})} \\ 
	& -
	\sum_{\substack{	\theta \in P(\pm {2m})  \\ 
			{\theta} \geq \pi_{1} \sqcup \pi_{2}}}
	\sum_{\substack{	\mathbf{i}:[\pm {2m}]\rightarrow[N] \\ 
			\mathrm{ker}(\mathbf{i}) = \theta}}
	\mathbf{h}(\mathbf{i \circ \bm{\sigma}})
	\exptr{\mathbf{w}_{1}(\mathbf{i,j})} 	\exptr{\mathbf{w}_{2}(\mathbf{i,j})}
	\exptr{\mathbf{y}_{1}(\mathbf{i})}		\exptr{\mathbf{y}_{2}(\mathbf{i})}. 
	\end{align*}
	To obtain (\ref{eqn.relating.c[pi].C[pi]}), it only remains to show that for ${\theta} \gneq \pi_{1} \sqcup \pi_{2} $, i.e., $\theta \geq \pi_{1} \sqcup \pi_{2}$ but $\theta \neq \pi_{1} \sqcup \pi_{2}$, implies 
	\begin{align*}
	& 
	\sum_{\substack{  \mathbf{i}:[\pm {2m}]\rightarrow[N] \\ \kernel{\mathbf{i}}={\theta}   }} 
	{\mathbf{h}}(\mathbf{i}\circ \bm{{\sigma}})
	\exptr{{\mathbf{w}}_{1}(\mathbf{i},\mathbf{j})}
	\exptr{{\mathbf{w}}_{2}(\mathbf{i},\mathbf{j})}
	\exptr{{\mathbf{y}}_{1}(\mathbf{i})}
	\exptr{{\mathbf{y}}_{2}(\mathbf{i})}				 
	= 
	O\left(N^{ -1  }\right) . 
	\end{align*}
	Suppose ${\theta} \in P(\pm {2m})$ satisfies ${\theta} \gneq \pi_{1} \sqcup \pi_{2} $. 
	Then, we must have $\#(\theta) < \#(\pi_{1} \sqcup \pi_{2})  = \#(\pi_{1}) +  \#(\pi_{2})$, or, equivalently,  
	\[
	\#(\theta) - \#(\pi_{1}) -  \#(\pi_{2}) \leq -1.
	\]
	Now, $\mathbf{h}(\mathbf{i}\circ \bm{{\sigma}})$ has absolute value $1$ for any function $\mathbf{i}:[\pm {2m}]\rightarrow[N]$ and $\abs*{\exptr{{\mathbf{y}}_{1}(\mathbf{i})}
		\exptr{{\mathbf{y}}_{2}(\mathbf{i})}}\leq 1$, so	(\ref{dist.entries.signed.perm.matrix}) implies 
	\begin{align*}
	& \abs*{
		\sum_{\substack{  \mathbf{i}:[\pm {2m}]\rightarrow[N] \\ \kernel{\mathbf{i}}={\theta}   }} 
		{\mathbf{h}}(\mathbf{i}\circ \bm{{\sigma}})
		\exptr{{\mathbf{w}}_{1}(\mathbf{i},\mathbf{j})}
		\exptr{{\mathbf{w}}_{2}(\mathbf{i},\mathbf{j})}
		\exptr{{\mathbf{y}}_{1}(\mathbf{i})}
		\exptr{{\mathbf{y}}_{2}(\mathbf{i})}				} \\
	& \leq 
	\frac{(N-\#(\pi_{1}))!}{N!} \cdot \frac{(N-\#(\pi_{2}))!}{N!}  \cdot 
	\abs*{\sum_{\substack{  \mathbf{i}:[\pm {2m}]\rightarrow[N] \\ \kernel{\mathbf{i}}={\theta}   }} 
		{\mathbf{h}}(\mathbf{i}\circ \bm{{\sigma}})			} 
	= 
	O\left(N^{- \#(\pi_{1}) -\#(\pi_{2}) + \#({\theta})  }\right)
	= 
	O\left(N^{	-1	}\right).   
	\end{align*}
\end{proof}
%
%
%
%

\begin{proof}[\textbf{Proof of (\ref{prop.fluc.h}) from Proposition \ref{prop.fluc}}]
	Let $Y$ be the identity matrix $I_{N}$. 
	By Proposition \ref{prop.cpi-and-Cpi}, we only need to show that 
	\begin{align*}
	\mathfrak{C}_{2}[ \pi ]  
	& =
	\left\{\begin{array}{cl}
	1 +  O\left( N^{-1}\right)  
	& \text{if there symmetric pairing partition } \hat\theta \leq \pi \text{ satisfying either (\ref{prop.minimal.special.partitions.1}) or 
		(\ref{prop.minimal.special.partitions.2}) }
	\\ & \text{from Proposition \ref{prop.minimal.special.partitions}}, \\ 
	O\left( N^{-1/2}\right) 
	& \text{otherwise. } 
	\end{array} \right.
	\end{align*}	
	where $\mathfrak{C}_{2} \left[ \pi \right]$ is given by (\ref{eqn.unnormalized.flucts}).
	Note that from (\ref{dist.entries.signed.perm.matrix}) and (\ref{eqn.unnormalized.flucts}) we obtain the inequality 
	\begin{align*}
	\abs*{\mathfrak{C}_{2}[\pi]  }	
	\leq
	\frac{(N-\#(\pi))!}{N!}
	\abs*{ \sum_{\substack{
				\mathbf{i}:[\pm {2m} ] \rightarrow [N] \\ 
				\kernel{\mathbf{i}}=\pi}} 
		\mathbf{h}(\mathbf{i}\circ\bm{\sigma}) }
	+
	\frac{(N-\#(\pi_{1}))!}{N!}
	\frac{(N-\#(\pi_{2}))!}{N!}
	\abs*{ \sum_{\substack{
				\mathbf{i}:[\pm {2m} ] \rightarrow [N] \\ 
				\kernel{\mathbf{i}}=\pi_{1}\sqcup\pi_{2} }} 
		\mathbf{h}(\mathbf{i}\circ\bm{\sigma}) } . 
	\end{align*}	
	But then, if $\mathrm{p}_{\sigma^{-1} \circ \pi}$ is a non-zero polynomial, so is $\mathrm{p}_{\sigma^{-1} \circ \pi_{1} \sqcup \pi_{2} }$ by Proposition \ref{minimal.special.polynomials}, and therefore, the last inequality and Corollary \ref{cor.bound.gauss.sum} would imply  
	\(
	\mathfrak{C}_{2} \left[ \pi \right] = O(N^{-1/2}) . \text{ } 
	\)
	And so, we can assume $\mathrm{p}_{\sigma^{-1} \circ \pi}$ is the zero polynomial without loss of generality. 

	Now, for every function $\mathbf{i}:[\pm {2m} ] \rightarrow [N]$ satisfying $\kernel{\mathbf{i}}=\pi$ we have $\mathbf{h}(\mathbf{i}\circ\bm{\sigma}) = 1$ since $\mathrm{p}_{\sigma^{-1} \circ \pi}$ is the zero polynomial; additionally, (\ref{dist.entries.signed.perm.matrix}) gives
	\(
	\exptr{\mathbf{w}(\mathbf{i,j})} = 
	\frac{(N-\#(\pi_{}))!}{N!} 
	\)
	since $\pi$ is an even partition. 
	Thus, from (\ref{eqn.unnormalized.flucts}) we obtain 
	\begin{align}\label{C2pi.sigma-pi-zero.case.1}
	\mathfrak{C}_{2}[\pi] = 1 -	
	\sum_{\substack{ 
			\mathbf{i}:[\pm {2m} ] \rightarrow [N] \\ \kernel{\mathbf{i}}=\pi_1 \sqcup \pi_2 }} 
	\mathbf{h}(\mathbf{i}\circ\bm{\sigma})  \exptr{\mathbf{w}_1(\mathbf{i,j})} \exptr{\mathbf{w}_2(\mathbf{i,j})} .
	\end{align}
	Moreover, by Proposition \ref{minimal.special.polynomials}, there is a symmetric pairing partition $\hat\theta \leq \pi$ such that $\mathrm{p}_{\sigma^{-1} \circ \hat\theta}$ is also the zero polynomial, and hence,  
	%
	the partition $\hat\theta$ must satisfy one of the conditions \textit{(1)-(6)} from Proposition \ref{prop.minimal.special.partitions}.  
	Notice $\hat\theta = \hat\theta_{1} \sqcup \hat\theta_{2}$, where  $\hat\theta_1$ and $\hat\theta_2$ denote the restrictions of $\hat\theta$ to $[\pm 2m_1]$ and $[\pm 2m ] \setminus [\pm 2m_1]$, respectively, implies 
	\begin{align}\label{C2pi.pi=pi1_pi2.case.1}
	\mathfrak{C}_{2}[\pi] 
	= 
	1 -	
	\frac{(N-\#(\pi_1))!}{N!}	\cdot \frac{(N-\#(\pi_2))!}{N!} \cdot \frac{N!}{(N-\#(\pi_1\sqcup\pi_2))!} 
	= 
	O\left( N^{-1}\right).
	\end{align}
	Indeed, if $\hat\theta = \hat\theta_{1} \sqcup \hat\theta_{2}$, then
	 $\hat\theta_{1}$ and  $\hat\theta_{2}$ must be even partitions, and so are $\pi_{1}$ and $\pi_{2}$ since  $\hat\theta \leq \pi$ implies $\hat\theta_{1} \leq \pi_{1}$ and  $\hat\theta_{2} \leq \pi_{2}$, so (\ref{dist.entries.signed.perm.matrix}) gives
	\begin{align*}
	\exptr{\mathbf{w}_{1}(\mathbf{i,j})} = 
	\frac{(N-\#(\pi_{1}))!}{N!} 
	\quad\text{and}\quad
	\exptr{\mathbf{w}_{2}(\mathbf{i,j})} = 
	\frac{(N-\#(\pi_{2}))!}{N!} 
	\end{align*}
	for every function $\mathbf{i}:[\pm {2m} ] \rightarrow [N]$ satisfying $\kernel{\mathbf{i}}=\pi_{1} \sqcup \pi_{2}$; moreover, Proposition \ref{minimal.special.polynomials} implies the polynomial  $\mathrm{p}_{\sigma^{-1} \circ \pi_1 \sqcup \pi_2}$ is also zero since $\hat\theta = \hat\theta_{1} \sqcup \hat\theta_{2}  \leq \pi_1 \sqcup \pi_2$, and thus, we obtain
	%
	%
	$$\mathbf{h}(\mathbf{i}\circ\bm{\sigma}) = 1 .$$
	%
	%
	%
	Hence, (\ref{C2pi.pi=pi1_pi2.case.1}) follows from (\ref{C2pi.sigma-pi-zero.case.1}) provided $\hat\theta$ satisfies either 
	\textit{(\ref{prop.minimal.special.partitions.3})},
	\textit{(\ref{prop.minimal.special.partitions.4})},
	\textit{(\ref{prop.minimal.special.partitions.5})}, or
	\textit{(\ref{prop.minimal.special.partitions.6})} from Proposition \ref{prop.minimal.special.partitions}. 

 	Assume now $\hat\theta$ satisfies either \textit{(\ref{prop.minimal.special.partitions.1})} or \textit{(\ref{prop.minimal.special.partitions.2})} from Proposition \ref{prop.minimal.special.partitions}. 
	Then, either $\pi_1 \sqcup \pi_2$ contains some singletons, if $\{k,l\} \in \pi$ or $\{k,-l\} \in \pi$ for some integers $1\leq k \leq 2m_{1} < l \leq 2m_{1}+2m_{2}$, or $\pi_1 \sqcup \pi_2 = \{\{-k,k\} \mid k \in [\pm 2m ]\}$, otherwise. 
	In any case, the graph $\vec{\mathcal{G}}_{\pi_1 \sqcup \pi_2}$ does not satisfy none of the conditions \textit{(1)-(6)} from Remark \ref{rmk.single.loops} since  $m_{1}+m_{2} > 2$, and hence,  the polynomial $\mathrm{p}_{\sigma^{-1} \circ \pi_1 \sqcup \pi_2}$ is non-zero. 
	Thus, by (\ref{dist.entries.signed.perm.matrix})  and Corollary \ref{cor.bound.gauss.sum}, we have 
	\begin{align}\label{ineq.h.w1.w2}	
	\abs*{
		\sum_{\substack{ 
				\mathbf{i}:[\pm {2m} ] \rightarrow [N] \\ \kernel{\mathbf{i}}=\pi_1 \sqcup \pi_2 }} 
		\mathbf{h}(\mathbf{i}\circ\bm{\sigma})  \exptr{\mathbf{w}_1(\mathbf{i,j})} \exptr{\mathbf{w}_2(\mathbf{i,j})}
	}
	\leq
	C N^{\#( \pi_1 \sqcup \pi_2 ) - \frac{1}{2} } \cdot	\frac{(N-\#(\pi_1))!}{N!}	\cdot \frac{(N-\#(\pi_2))!}{N!}
	\end{align}
	for some constant $C > 0$ independent from $N$. 
	Therefore,  from  (\ref{C2pi.sigma-pi-zero.case.1}) we get that
	\[
	\mathfrak{C}_{2}[\pi] 
	= 
	1 +	O( N^{-1/2}). 
	\]
\end{proof}
%
%
%
%
%
\begin{proof}[\textbf{Proof of (\ref{prop.fluc.xh}) from Proposition \ref{prop.fluc}}]
	Let $Y$ be a random $N$-by-$N$ signature matrix independent from $W$.
	Similar to the previous case, $U^{*}_{N,{i}_{1}} U^{}_{N,{i}_{2}} = W^{*} H^{}  W / \sqrt{N}$, 
	we can assume  $\mathrm{p}_{\sigma^{-1} \circ \pi}$ is the zero polynomial and it suffices to show that
	\begin{align*}
	\mathfrak{C}_{2}[ \pi ]  
	& =
	\left\{\begin{array}{cl}
	1 +  O\left( N^{-1}\right)  
	& \text{if there symmetric pairing partition } \hat\theta \leq \pi \text{ satisfying (\ref{prop.minimal.special.partitions.1}) from } 
	\\ & \text{Proposition \ref{prop.minimal.special.partitions}}, \\ 
	1 +  O\left( N^{-1}\right)  
	& \text{if there symmetric pairing partition } \hat\theta \leq \pi \text{ satisfying (\ref{prop.minimal.special.partitions.2}) from } 
	\\ & \text{Proposition \ref{prop.minimal.special.partitions} and the graph } \mathcal{G}_{\pi} \text{ has only double-loops as components}, \\
	O\left( N^{-1/2}\right) 
	& \text{otherwise. } 
	\end{array} \right.
	\end{align*}		

	Let $\pi_{\text{odd},1}$ and  $\pi_{\text{odd},2}$ denote the restrictions of $\pi_{\text{odd}}$ to $[\pm 2m_1]$ and $[\pm 2m ] \setminus [\pm 2m_1]$, respectively. 
	Note that if $\pi_{\text{odd}}$ is not an even partition, then either  $\pi_{\text{odd},1}$ or  $\pi_{\text{odd},2}$ is not even, and hence, we obtain $\mathfrak{C}_{2}[\pi]=0$ since (\ref{dist.entries.signature.matrix}) would imply   
	$\exptr{\mathbf{y}_1(\mathbf{i})} \exptr{\mathbf{y}_2(\mathbf{i})} =\exptr{\mathbf{y}(\mathbf{i})} 
	=0
	$
	for every function $\mathbf{i}:[\pm {2m} ] \rightarrow [N]$ satisfying $\kernel{\mathbf{i}}=\pi$. 
	Thus, we can further assume $\pi_{\text{odd}}$ is even.
	It then follows from (\ref{dist.entries.signature.matrix}) and (\ref{dist.entries.signed.perm.matrix}) that 
	\[
	\mathbf{h}(\mathbf{i}\circ\bm{\sigma}) = 1,  
	\quad 
	\exptr{\mathbf{y}(\mathbf{i})} =1
	\quad \text{and} \quad 
	\exptr{\mathbf{w}(\mathbf{i,j})} = \frac{(N-\#(\pi))!}{N!} 
	\] 
	for $\mathbf{i}:[\pm {2m} ] \rightarrow [N]$ satisfying $\kernel{\mathbf{i}}=\pi$, and hence, we obtain  
	\begin{align}\label{C2pi.sigma-pi-zero.case.2}
	\mathfrak{C}_{2}[\pi] = 1 -	
	\sum_{\substack{ 
			\mathbf{i}:[\pm {2m} ] \rightarrow [N] \\ \kernel{\mathbf{i}}=\pi_1 \sqcup \pi_2 }} 
	\mathbf{h}(\mathbf{i}\circ\bm{\sigma}) 
	\exptr{\mathbf{y}_1(\mathbf{i})} \exptr{\mathbf{y}_2(\mathbf{i})}  \exptr{\mathbf{w}_1(\mathbf{i,j})} \exptr{\mathbf{w}_2(\mathbf{i,j})} .  
	\end{align}
	By Proposition \ref{minimal.special.polynomials}, there is a symmetric pairing partition $\hat\theta \leq \pi$ such that $\mathrm{p}_{\sigma^{-1} \circ \hat\theta}$ is also the zero polynomial, and thus,  
	the partition $\hat{\theta}$ must satisfy one of the conditions \textit{(1)-(6)} from Proposition \ref{prop.minimal.special.partitions}. 
	However, if $\hat{\theta}$ satisfies either 
	\textit{(\ref{prop.minimal.special.partitions.3})},
	\textit{(\ref{prop.minimal.special.partitions.4})},
	\textit{(\ref{prop.minimal.special.partitions.5})}, or
	\textit{(\ref{prop.minimal.special.partitions.6})}, then 
	\begin{align}\label{C2pi.pi=pi1_pi2.case.2}
	\mathfrak{C}_{2}[\pi] 
	= 
	1 -	\frac{N!}{(N-\#(\pi_1\sqcup\pi_2))!} \cdot
	\frac{(N-\#(\pi_1))!}{N!}	\cdot
	\frac{(N-\#(\pi_2))!}{N!}
	= 
	O\left( N^{-1}\right). 
	\end{align}
	Indeed, suppose $\hat{\theta}$ satisfies either 
	\textit{(\ref{prop.minimal.special.partitions.3})},
	\textit{(\ref{prop.minimal.special.partitions.4})},
	\textit{(\ref{prop.minimal.special.partitions.5})}, or
	\textit{(\ref{prop.minimal.special.partitions.6})} from Proposition \ref{prop.minimal.special.partitions}, let $\hat{\theta}_1$ and $\hat{\theta}_2$ denote the restrictions of $\hat{\theta}$ to $[\pm 2m_1]$ and $[\pm 2m ] \setminus [\pm 2m_1]$, respectively, and let $\mathbf{i}:[\pm {2m} ] \rightarrow [N]$ be a function satisfying $\kernel{\mathbf{i}}=\pi_{1} \sqcup \pi_{2}$. 
	Note that $\hat{\theta} = \hat{\theta}_1 \sqcup \hat{\theta}_2  \leq  \pi_{1} \sqcup \pi_{2}$ since $\hat{\theta} \leq \pi$ implies $\hat{\theta}_{1} \leq \pi_{1}$ and  $\hat{\theta}_{2} \leq \pi_{2}$, and thus, by Proposition \ref{minimal.special.polynomials},  the polynomial $\mathrm{p}_{\sigma^{-1} \circ \pi_{1} \sqcup \pi_{2}}$ is zero, and hence, we get  
	\begin{align*}
	\mathbf{h}(\mathbf{i}\circ\bm{\sigma}) = 1. 
	\end{align*}
	Moreover, $\pi_{1}$ and $\pi_{2}$ are even partitions since  
	$\hat{\theta}$ is even and $\hat{\theta} = \hat{\theta}_1 \sqcup \hat{\theta}_2 \leq \pi$, so, from (\ref{dist.entries.signed.perm.matrix}),  we get 
	\begin{align*}
	\exptr{\mathbf{w}_{1}(\mathbf{i,j})} = 
	\frac{(N-\#(\pi_{1}))!}{N!} 
	\quad\text{and}\quad
	\exptr{\mathbf{w}_{2}(\mathbf{i,j})} = 
	\frac{(N-\#(\pi_{2}))!}{N!}  . 
	\end{align*}
	%
	%
	%
	%
	The partitions $\pi_{\text{odd},1}$ and $\pi_{\text{odd},2}$ are also even since $\hat{\theta}_{\text{odd}}$ is even and   $\hat{\theta} = \hat{\theta}_1 \sqcup \hat{\theta}_2 \leq \pi$ implies $\hat{\theta}_{\text{odd}} = \hat{\theta}_{\text{odd},1} \sqcup \hat{\theta}_{\text{odd},2}$, $\hat{\theta}_{\text{odd},1} \leq \pi_{\text{odd},1}$, and  $\hat{\theta}_{\text{odd},2} \leq \pi_{\text{odd},2}$ where  $\hat{\theta}_{\text{odd},1}$ and  $\hat{\theta}_{\text{odd},2}$ denote the restrictions of $\hat{\theta}_{\text{odd}}$ to $[\pm 2m_1]$ and $[\pm 2m ] \setminus [\pm 2m_1]$, respectively.
	Thus, from (\ref{dist.entries.signature.matrix}), we have   
	\[
	\exptr{\mathbf{y}_1(\mathbf{i})} \exptr{\mathbf{y}_2(\mathbf{i})} = 1  .
	\]
	Consequently, we obtain (\ref{C2pi.pi=pi1_pi2.case.2}) from (\ref{C2pi.sigma-pi-zero.case.2}). 
	Now, similar to the case $U^{*}_{N,{i}_{1}} U^{}_{N,{i}_{2}} = W^{*} H^{}  W / \sqrt{N}$, if $\hat{\theta}$ satisfies either  \textit{(\ref{prop.minimal.special.partitions.1})} or \textit{(\ref{prop.minimal.special.partitions.2})} from Proposition \ref{prop.minimal.special.partitions}, then $\mathrm{p}_{\sigma^{-1} \circ \pi_1 \sqcup \pi_2}$ is a non-zero polynomial and (\ref{ineq.h.w1.w2}) holds, so, from (\ref{C2pi.sigma-pi-zero.case.2}), we obtain  
	\[	
	\mathfrak{C}_{2}[\pi] 
	=
	1 +	O( N^{-1/2}) 
	\]
	since we have $\abs*{\exptr{\mathbf{y}_1(\mathbf{i})} \exptr{\mathbf{y}_2(\mathbf{i})} = 1 } \leq 1$ for any function $\mathbf{i}:[\pm {2m} ] \rightarrow [N]$. 

	It only remains to show that the undirected graph $\mathcal{G}_{\pi}$ must have only double-loops as connected components if  $\hat{\theta}$ satisfies \textit{(\ref{prop.minimal.special.partitions.2})} from Proposition \ref{prop.minimal.special.partitions}. 
	So, suppose $\hat{\theta}$ satisfies \textit{(\ref{prop.minimal.special.partitions.2})} from Proposition \ref{prop.minimal.special.partitions}. 
	Note that if $\pi$ has a block of the form $\{k,l\}$, then  $k+l$ is odd, and hence, we must have either $\{k\}$ or $\{l\}$ is a block of $\pi_{\text{odd}}$, contradicting the assumption that $\pi_{\text{odd}}$ is an even partition. 
	Thus, $\pi$ has only blocks of the form $\{k,l,-k,-l\}$, or, equivalently, the undirected graph $\mathcal{G}_{\pi}$ has only double loops as connected components. 
\end{proof}
%
%
%

%
%
\subsection*{Case $\bm{ U^{*}_{N,{i}_{1}} U^{}_{N,{i}_{2}} =\frac{1}{N} W^{*} H^{*} X H^{}  W}$ } 
%
%
%
%
For each function $\mathbf{i}:[\pm {4m}] \rightarrow [N]$, we let $\widehat{\mathbf{h}}_{}(\mathbf{i})$, $\widehat{\mathbf{g}}_{}(\mathbf{i})$, $\widehat{\mathbf{x}}_{1}(\mathbf{i})$, and $\widehat{\mathbf{x}}_{2}(\mathbf{i})$ be given by 
\begin{align*}
%
%
%
%
\widehat{\mathbf{h}}_{}(\mathbf{i})	=&
\prod_{k=1}^{2m}
H^{*}(i_{-2k+1},i_{2k-1}) H(i_{-2k},i_{2k}) , 
&\widehat{\mathbf{g}}_{}(\mathbf{i})=&
\prod_{k=1}^{2m}
H^{*}(i_{-2k+1},i_{-2k}) H(i_{2k},i_{2k-1})  ,  \\
%
%
%
%
\widehat{\mathbf{x}}_{1}(\mathbf{i}) = & 
\prod_{k=1}^{2m_{1}} 
X(i_{-2k},i_{2k}), \qquad \quad \text{and}
&\widehat{\mathbf{x}}_{2}(\mathbf{i}) =&
\prod_{k=2m_{1}+1}^{2m_{1}+2m_{2}} 
X(i_{-2k},i_{2k});
\end{align*}
additionally, if we are given a function $\mathbf{j}:[\pm {2m}] \rightarrow [N]$, we take  
%
%
%
%
\begin{align}\label{eqn.variable.t}
\mathbf{t}  =
(t_{-1},t_{1},t_{-3},t_{3},\ldots,t_{4m_{1}+4m_{2}-1}) 
=
(j_{-1},j_{1},j_{-2},j_{2},\ldots,j_{-2m},j_{2m})
\end{align}
%
%
%
%
%
and let $\widehat{\mathbf{w}}_{1}(\mathbf{i,t})$ and $\widehat{\mathbf{w}}_{2}(\mathbf{i,t})$, also denoted  $\widehat{\mathbf{w}}_{1}(\mathbf{i})$ and $\widehat{\mathbf{w}}_{2}(\mathbf{i})$, respectively, be defined by 
%
%
\begin{align*}
\widehat{\mathbf{w}}_{1}(\mathbf{i,t})
=&
\prod_{k=1}^{2m_{1}} 
W(i_{2k-1},t_{2k-1}) W(i_{-2k+1},t_{-2k+1})  \quad \text{and}\\ 
\widehat{\mathbf{w}}_{2}(\mathbf{i,t})
=&
\prod_{k=2m_{1}+1}^{2m_{1}+2m_{2}} 
W(i_{2k+1},t_{2k+1}) W(i_{-2k+1},t_{-2k+1}) . 
\end{align*}
Now, given partitions $\pi \in P(\pm 2m)$ and ${\alpha} \in P_{2}(2m)$ and a function $\mathbf{j}:[\pm {2m}] \rightarrow [N]$ satisfying $ \kernel{\mathbf{j}} = \pi$, we define $\mathfrak{C}_{2} \left[ \pi, {\alpha} \right]$ by 
\begin{align}\label{eqn.unnormalized.fluct.HXH}
&
\sum_{\substack{ \mathbf{i}:[\pm {4m} ] \rightarrow [N] \\  \kernel{\mathbf{i}}= \widehat{\alpha} \sqcup \pi^{\text{odd}} }} 
\widehat{\mathbf{h}}(\mathbf{i} \circ \bm{\widehat{\sigma}} ) 
\exptr{\widehat{\mathbf{w}}(\mathbf{i})}
\exptr{\widehat{\mathbf{x}}(\mathbf{i})}  \nonumber
-\sum_{\substack{ \mathbf{i}:[\pm {4m} ] \rightarrow [N] \\  \kernel{\mathbf{i}}=\widehat{\alpha} \sqcup \pi_{1}^{\text{odd}} \sqcup \pi_{2}^{\text{odd}}  }} 
\widehat{\mathbf{h}}(\mathbf{i} \circ \bm{\widehat{\sigma}}) 
\exptr{\widehat{\mathbf{w}}_{1}(\mathbf{i})}
\exptr{\widehat{\mathbf{w}}_{2}(\mathbf{i})}
\exptr{\widehat{\mathbf{x}}_{1}(\mathbf{i})}
\exptr{\widehat{\mathbf{x}}_{2}(\mathbf{i})}  \\
&= N^{m}\mathfrak{C}_{2} \left[ \pi, {\alpha} \right]
\end{align}
where $\pi^{\text{odd}}_{1}$ and $\pi^{\text{odd}}_{2}$ denote the restrictions of $\pi^{\text{odd}}$ to the sets $[\pm 4m_{1}]$ and $[\pm (4m_{1}+4m_{2}) ] \setminus [\pm 4m_{1}]$, respectively, $\widehat\alpha$ is the partition given by $\widehat\alpha=\{\{-2k,2k,-2l,2l\} \mid \{k,l\} \in \alpha \}$, and  $\widehat{\sigma} : [\pm {4m}] \rightarrow [\pm {4m}]$ is the permutation with cycle decomposition 
\begin{align*}
\widehat{\sigma} = 	(-1,1,-2,2,\ldots,-4m_{1},4m_{1})
(-4m_{1}-1,4m_{1}+1,-4m_{2}-2,\ldots,4m_{1}+4m_{2}). 
\end{align*}
%
%
%
%
%
\begin{proposition}\label{prop.cpi-and-Cpi-nu}
	Let $\pi$ be an even partition in $P(\pm {2m})$. 
	Suppose $U^{*}_{N,{i}_{1}} U^{}_{N,{i}_{2}} =  W^{*} H^{*}X H^{} W / N$. 
	If $\mathfrak{c}_{2} \left[ \pi \right] $ is given by (\ref{eqn.normalized.fluct.c[pi]}) and  $\mathfrak{C}_{2} \left[ \pi, {\alpha} \right] $ is given by (\ref{eqn.unnormalized.fluct.HXH}) for every pairing partition ${\alpha} \in P_{2}(2m)$, then 
	\begin{align}\label{eqn.c[pi].And.C[pi,eta]}
	N^{2m}\mathfrak{c}_{2} \left[ \pi \right] 
	=	&
	\left(\sum_{\substack{ {\alpha} \in {P}_{2}( 2m )   }}
	\mathfrak{C}_{2} \left[ \pi, {\alpha} \right]  \right)
	+
	O\left( N^{m-1} \right) 
	\end{align}
\end{proposition}
%
%
\begin{proof}
	Fix a function $\mathbf{j}:[\pm {2m}] \rightarrow [N]$ satisfying   $\kernel{\mathbf{j}} = \pi$ and let $\mathbf{t}$ be as in (\ref{eqn.variable.t}). 
	The $(j_{-k},j_{k})$-entry of $U^{*}_{N,{i}_{1}} U^{}_{N,{i}_{2}}$ is then given by the sum
	\begin{align*}
	\sum_{i_{-2k+1},i_{2k+1},i_{-2k},i_{2k}=1}^{N} 
	W^{*}(t_{-2k+1},i_{-2k+1}) H^{*}(i_{-2k+1},i_{-2k}) X(i_{-2k},i_{2k}) H^{}(i_{2k},i_{2k-1})  W(i_{2k-1},t_{2k-1}),  
	\end{align*}
	and hence, by Equation (\ref{eqn.normalized.fluct.c[pi]}) and the linearity of the covariance, we get 
	\begin{align*}
	N^{2m} \mathfrak{c}_{2} \left[ \pi \right] 
	& =	
	\sum_{\mathbf{i}:[\pm {4m}]\rightarrow[N]}
	\widehat{\mathbf{g}}(\mathbf{i}) \cdot
	\text{cov} \big[ \widehat{\mathbf{w}}_{1}\mathbf{(i, t \circ \bm{\widetilde{\sigma}})} \cdot \widehat{\mathbf{x}}_{1}\mathbf{(i)} , \mathbf{w}_{2}\mathbf{(i, t \circ \bm{\widetilde{\sigma}})} \cdot \widehat{\mathbf{x}}_{2}\mathbf{(i )} \big]
	\end{align*} 
	where $\widehat{\mathbf{g}}(\mathbf{i})$, $\widehat{\mathbf{w}}_{1}\mathbf{(i,t)}$, $\widehat{\mathbf{w}}_{2}\mathbf{(i,t)}$, $\widehat{\mathbf{x}}_{1}\mathbf{(i)}$, and $\widehat{\mathbf{x}}_{2}\mathbf{(i)}$ are defined as above and $\widetilde\sigma : [\pm 4m] \rightarrow [\pm 4m] $ is the permutation with cycle decomposition
	\[
	\widetilde\sigma = 	(-1,1,-3,3,\ldots,-4m_{1}+1,4m_{1}-1)
	(-4m_{1}-1,4m_{1}+1,\ldots,4m_{1}+4m_{2}-1). 
	\]	
	Note that for every function $\mathbf{i}:[\pm {4m}]\rightarrow[N]$ we have
	\[
	\widehat{\mathbf{h}}(\mathbf{i}\circ \bm{\widehat{\sigma}}) = 
	\widehat{\mathbf{g}}(\mathbf{i}\circ \bm{\widetilde{\sigma}}), \quad 
	\widehat{\mathbf{w}}_{k}\mathbf{(i,t)} =
	\widehat{\mathbf{w}}_{k}\mathbf{(i\circ \bm{\widetilde{\sigma}}, t \circ \bm{\widetilde{\sigma}})} ,\quad \text{ and } \quad
	\widehat{\mathbf{x}}_{k}\mathbf{(i)} =	\widehat{\mathbf{x}}_{k}\mathbf{(i\circ \bm{\widetilde{\sigma}})}
	\]
	for $k=1,2$, so we get
	\begin{align*}
	N^{2m}\mathfrak{c}_{2} \left[ \pi \right]
	&=	
	\sum_{\mathbf{i}:[\pm {4m}]\rightarrow[N]}
	\widehat{\mathbf{h}}(\mathbf{i}\circ \bm{\widehat{\sigma}}) \cdot
	\cov{ \widehat{\mathbf{w}}_{1}\mathbf{(i, t )} \cdot \widehat{\mathbf{x}}_{1}\mathbf{(i)} }{\mathbf{w}_{2}\mathbf{(i, t )} \cdot \widehat{\mathbf{x}}_{2}\mathbf{(i )} } .
	\end{align*}
	Now, suppose $\theta = \kernel{\mathbf{i}}$ for a function $\mathbf{i}:[\pm {4m}]\rightarrow[N]$. 
	Since $\pi^{\text{odd}} = \kernel{\mathbf{t}}$, from (\ref{dist.entries.signed.perm.matrix}) we have that
	\[
	\exptr{\widehat{\mathbf{w}}_{1}(\mathbf{i},\mathbf{t}) \widehat{\mathbf{w}}_{2}(\mathbf{i},\mathbf{t})}= 0
	\quad  \text{and} \quad
	\exptr{\widehat{\mathbf{w}}_{1}(\mathbf{i},\mathbf{t})}\exptr{\widehat{\mathbf{w}}_{2}(\mathbf{i},\mathbf{t})}=0
	\] 
	provided  ${\theta}_{\text{odd}} \neq \pi^{\text{odd}}$ and ${\theta}_{\text{odd}} \not\geq (\pi_{1} \sqcup \pi_{2})^{\text{odd}} = \pi_{1} ^{\text{odd}}  \sqcup \pi_{2}^{\text{odd}} $, respectively;
	%
	%
	moreover, (\ref{dist.entries.signature.matrix}) implies that
	\[
	\exptr{\widehat{\mathbf{x}}_{1}(\mathbf{i})\widehat{\mathbf{x}}_{2}(\mathbf{i})}=\exptr{\widehat{\mathbf{x}}_{1}(\mathbf{i})}\exptr{\widehat{\mathbf{x}}_{2}(\mathbf{i})}=0 
	\]
	if ${\theta}_{\text{even}}$ is not an even partition, ${\theta}_{\text{even}}$ has a block of the form $\{ 2k, -2k\}$, or $2k \not\sim_{{\theta}} -2k$ for some $k \in [2m]$. 
	Thus, we obtain
	\begin{align*}
	N^{2m} \mathfrak{c}_{2} \left[ \pi \right]
	= & 
	\sum_{\substack{ {\theta} \in \widetilde{P}_{\pi}( \pm 4m )   }}
	\sum_{\substack{  \mathbf{i}:[\pm {4m}]\rightarrow[N] \\ \kernel{\mathbf{i}}={\theta}   }} 
	\widehat{\mathbf{h}}(\mathbf{i}\circ \bm{\widehat{\sigma}})	\exptr{\widehat{\mathbf{w}}_{1}(\mathbf{i},\mathbf{t}) \widehat{\mathbf{w}}_{2}(\mathbf{i},\mathbf{t})}
	\exptr{\widehat{\mathbf{x}}_{1}(\mathbf{i})\widehat{\mathbf{x}}_{2}(\mathbf{i})}  \\ 	& -
	\sum_{\substack{ {\theta} \in \widetilde{P}_{\pi_{1} \sqcup \pi_{2}}( \pm 4m )   }}
	\sum_{\substack{  \mathbf{i}:[\pm {4m}]\rightarrow[N] \\ \kernel{\mathbf{i}}={\theta}   }}
	\widehat{\mathbf{h}}(\mathbf{i}\circ \bm{\widehat{\sigma}})
	\exptr{\widehat{\mathbf{w}}_{1}(\mathbf{i},\mathbf{t})}
	\exptr{\widehat{\mathbf{w}}_{2}(\mathbf{i},\mathbf{t})}
	\exptr{\widehat{\mathbf{x}}_{1}(\mathbf{i})}
	\exptr{\widehat{\mathbf{x}}_{2}(\mathbf{i})} 
	\end{align*}
	where $\widetilde{P}_{\beta}(\pm 4m)$ denotes the set of all partitions ${\theta} \in P(\pm 4m)$ such that ${\theta}_{\text{odd}} \geq \beta^{\text{odd}}$ and for every integer $k \in [2m]$ there exists  $l \in [2m]\setminus\{k\}$ such that $2k \sim_{{\theta}} -2k \sim_{{\theta}} -2l \sim_{{\theta}} 2l$.

	Now, letting $\widehat{P}_{\beta}(\pm 4m)$ denote the set of partitions ${\theta}  \in \widetilde{P}_{\beta}(\pm 4m)$ so that ${\theta} = {\theta}_{\text{even}} \sqcup {\theta}_{\text{odd}}$, ${\theta}_{\text{odd}} = \beta^{\text{odd}}$, and every block of ${\theta}_{\text{even}}$ is of the form $\{2k,-2k,2l,-2l\}$ with $k,l \in [2m]$ and $k\neq l$, 
	note the mapping $${\alpha} \mapsto \widehat{\alpha} \sqcup \beta^{\text{odd}}$$ with $\widehat{{\alpha}} = \left\{ \{2k,-2k,2l,-2l\} \mid \{k,l\} \in {\alpha} \right\}$ gives a bijection between the set of pairing partitions $P_{2}(2m)$ and the set $\widehat{P}_{\beta}(\pm 4m)$ for any partition $\beta \in P(\pm 2m)$. 
	Thus, to get (\ref{eqn.c[pi].And.C[pi,eta]}), it only remains to show that 
	\begin{align*}
	N^{2m} \mathfrak{c}_{2} \left[ \pi \right] 
	= & 
	\sum_{\substack{ {\theta} \in \widehat{P}_{\pi}( \pm 4m )   }}
	\sum_{\substack{  \mathbf{i}:[\pm {4m}]\rightarrow[N] \\ \kernel{\mathbf{i}}={\theta}   }} 
	\widehat{\mathbf{h}}(\mathbf{i}\circ \bm{\widehat{\sigma}})
	\exptr{\widehat{\mathbf{w}}(\mathbf{i},\mathbf{t})}
	\exptr{\widehat{\mathbf{x}}(\mathbf{i})}  \\
	& -
	\sum_{\substack{ 	{\theta} \in \widehat{P}_{\pi_{1} \sqcup \pi_{2} } ( \pm 4m ) }}
	\sum_{\substack{  \mathbf{i}:[\pm {4m}]\rightarrow[N] \\ \kernel{\mathbf{i}}={\theta}   }} 
	\widehat{\mathbf{h}}(\mathbf{i}\circ \bm{\widehat{\sigma}})
	\exptr{\widehat{\mathbf{w}}_{1}(\mathbf{i},\mathbf{t})}
	\exptr{\widehat{\mathbf{w}}_{2}(\mathbf{i},\mathbf{t})}
	\exptr{\widehat{\mathbf{x}}_{1}(\mathbf{i})}
	\exptr{\widehat{\mathbf{x}}_{2}(\mathbf{i})}
	\\ & 	+
	O\left( N^{m-1} \right).  
	\end{align*}
	%
	%
	%
	%
	%
	%
	%
	%
	%
	%
	%
	%
	%
	Suppose ${\theta} \in \widetilde{P}_{\pi_{1} \sqcup \pi_{2}}( \pm 4m )$.
	Then, since ${\theta}_{\text{odd}} \geq (\pi_{1} \sqcup \pi_{2})^{\text{odd}} = \pi_{1}^{\text{odd}} \sqcup  \pi_{2}^{\text{odd}}$ and each block of ${\theta}_{\text{even}}$ has at least 4 elements, we get the inequality 
	\[
	\#({\theta})  \leq 
	\#({\theta}_{\text{odd}})+\#({\theta}_{\text{even}}) 
	\leq \#(\pi_{1}) + \#(\pi_{2}) + m   
	\] 
	with equality only if ${\theta} = {\theta}_{\text{even}} \sqcup {\theta}_{\text{odd}}$, ${\theta}_{\text{odd}} =  \pi_{1}^{\text{odd}} \sqcup  \pi_{2}^{\text{odd}}$, 
	and each block of ${\theta}_{\text{even}}$ has exactly 4 elements, i.e.,  ${\theta} \in \widehat{P}_{\pi_{1} \sqcup  \pi_{2} }( \pm 4m )$;
	moreover, (\ref{dist.entries.signature.matrix}) and (\ref{dist.entries.signed.perm.matrix}) imply that 
	\begin{align*}
	& \abs*{
		\sum_{\substack{  \mathbf{i}:[\pm {4m}]\rightarrow[N] \\ \kernel{\mathbf{i}}={\theta}   }} 
		\widehat{\mathbf{h}}(\mathbf{i}\circ \bm{\widehat{\sigma}})
		\exptr{\widehat{\mathbf{w}}_{1}(\mathbf{i},\mathbf{t})}
		\exptr{\widehat{\mathbf{w}}_{2}(\mathbf{i},\mathbf{t})}
		\exptr{\widehat{\mathbf{x}}_{1}(\mathbf{i})}
		\exptr{\widehat{\mathbf{x}}_{2}(\mathbf{i})}				} \\
	& \leq 
	\frac{(N-\#(\pi_{1}))!}{N!} \cdot \frac{(N-\#(\pi_{2}))!}{N!}  \cdot 
	\abs*{\sum_{\substack{  \mathbf{i}:[\pm {4m}]\rightarrow[N] \\ \kernel{\mathbf{i}}={\theta}   }} 
		\widehat{\mathbf{h}}(\mathbf{i}\circ \bm{\widehat{\sigma}})			} 
	= 
	O\left(N^{- \#(\pi_{1}) -\#(\pi_{2}) + \#({\theta})  }\right).   
	\end{align*}
	Hence, if ${\theta} \in \widetilde{P}_{\pi_{1} \sqcup \pi_{2}}( \pm 4m ) \setminus \widehat{P}_{\pi_{1} \sqcup  \pi_{2} }( \pm 4m )$,  we have
	\begin{align*}
	& 
	\sum_{\substack{  \mathbf{i}:[\pm {4m}]\rightarrow[N] \\ \kernel{\mathbf{i}}={\theta}   }} 
	\widehat{\mathbf{h}}(\mathbf{i}\circ \bm{\widehat{\sigma}})
	\exptr{\widehat{\mathbf{w}}_{1}(\mathbf{i},\mathbf{t})}
	\exptr{\widehat{\mathbf{w}}_{2}(\mathbf{i},\mathbf{t})}
	\exptr{\widehat{\mathbf{x}}_{1}(\mathbf{i})}
	\exptr{\widehat{\mathbf{x}}_{2}(\mathbf{i})}				 
	= 
	O\left(N^{ m-1  }\right) . 
	\end{align*}
	Similar arguments show that $ \#({\theta}) \leq m + \#(\pi_{})$ for every partition ${\theta} \in \widetilde{P}_{\pi_{}}( \pm 4m )$ with equality only if ${\theta} \in \widehat{P}_{\pi}( \pm 4m )$, and hence, we get
	\[
	\sum_{\substack{  \mathbf{i}:[\pm {4m}]\rightarrow[N] \\ \kernel{\mathbf{i}}={\theta}   }} 
	\widehat{\mathbf{h}}(\mathbf{i}\circ \bm{\widehat{\sigma}})
	\exptr{\widehat{\mathbf{w}}_{}(\mathbf{i},\mathbf{t})}
	\exptr{\widehat{\mathbf{x}}_{}(\mathbf{i})}		
	= 
		O\left(N^{ m-1  }\right) 
	\]
	for any ${\theta} \in  \widetilde{P}_{\pi_{}}( \pm 4m ) \setminus \widehat{P}_{\pi}( \pm 4m )$. 
\end{proof}
%
%
%
%
%

\begin{proposition}\label{prop.fluc.xhx.C[p,a]}
	Suppose $\pi \in P_{\text{}}(\pm 2m )$ and ${\alpha} \in P_{2}(2m)$ and let  ${\alpha}_{1}$ and  ${\alpha}_{2}$ denote the restrictions of  ${\alpha}$ to the sets $[2m_{1}]$ and $[ 2m_{1} + 2m_{2} ] \setminus [2m_{1}]$, respectively. 
	If 	 $\mathfrak{C}_{2}[\pi,{\alpha}] $ is given by (\ref{eqn.unnormalized.fluct.HXH}), then
	%
	\[
	\mathfrak{C}_{2}[\pi,{\alpha}] 
	=
	\left\{\begin{array}{cl}
1 +  O\left( N^{-1}\right)  
	& \text{if there is a symmetric pairing } \eta \in P(\pm {4m}) \text{ satisfying (\ref{prop.minimal.special.partitions.1}) from Proposi-}  \\
	& \text{tion \ref{prop.minimal.special.partitions} and such that }   \eta \leq \widehat{\alpha} \sqcup \pi^{\text{odd}} , \\
1 +  O\left( N^{-1}\right)  
	& \text{if } {\alpha} \neq {\alpha}_{1} \sqcup {\alpha}_{2} \text{ and there is a symmetric pairing }   \eta \in P(\pm {4m}) \text{ satisfying}	\\ 
	&  \text{(\ref{prop.minimal.special.partitions.3}) from Proposition \ref{prop.minimal.special.partitions}} \text{ with } $k$ \text{ and } $l$ \text{ even and such that } \eta \leq \widehat{\alpha} \sqcup \pi^{\text{odd}} , \\
O( N^{-\frac{1}{2}} )  
	& \text{otherwise. } \\
	\end{array} \right.
	\]
\end{proposition}
%
%
\begin{proof}  
	%
	Note that if the polynomial $\mathrm{p}_{\widehat\sigma^{-1} \circ (\widehat{\alpha} \sqcup \pi^{\text{odd}})}$ is non-zero, 
	then $\mathfrak{C}_{2}[\pi,{\alpha}] =  O(N^{-1/2})$.  
	Indeed, if $\mathrm{p}_{\widehat\sigma^{-1} \circ (\widehat{\alpha} \sqcup \pi^{\text{odd}})}$ is a non-zero polynomial, so is  $\mathrm{p}_{\widehat\sigma^{-1}  \circ (\widehat{\alpha} \sqcup \pi_{1}^{\text{odd}} \sqcup \pi_{2}^{\text{odd}} ) }$ by Proposition \ref{minimal.special.polynomials},   
	and thus, Corollary \ref{cor.bound.gauss.sum} implies there is a constant $C$ independent from $N$ such that 
	\begin{align*}	
	\abs*{ \sum_{\substack{ 
			\mathbf{i}:[\pm {4m} ] \rightarrow [N] \\  
			\kernel{\mathbf{i}} = {\widehat{\sigma}}^{-1}  \circ (  \widehat{\alpha} \sqcup \beta^{\text{odd}}  ) } } 		
			\widehat{\mathbf{h}}(\mathbf{i}  ) }
	=
	\abs*{ \sum_{\substack{ \mathbf{i}:[\pm {4m} ] \rightarrow [N] \\  \kernel{\mathbf{i}}= \widehat{\alpha} \sqcup \beta^{\text{odd}} }} 		\widehat{\mathbf{h}}(\mathbf{i} \circ \bm{\widehat{\sigma}} ) }	
	\leq 
	C_{} N ^{ m + \#(\beta) - \frac{1}{2}} 	
	\end{align*}
	for $\beta = \pi$ and  $\beta =\pi_{1} \sqcup \pi_{2}$. 
	But then, we get that $\mathfrak{C}_{2}[\pi,{\alpha}] =  O(N^{-1/2})$ since from (\ref{dist.entries.signature.matrix}),  (\ref{dist.entries.signed.perm.matrix}), and (\ref{eqn.unnormalized.fluct.HXH}) we have 
	\begin{align*}
	\Big|\mathfrak{C}_{2}[\pi, {\alpha}]  \Big|	
	\leq
	C \cdot	\frac{(N-\#(\pi))!}{N!} \cdot N ^{ \#(\pi_{})  - \frac{1}{2}} 
	+ 
	C \cdot \frac{(N-\#(\pi_{1}))!}{N!} \cdot \frac{(N-\#(\pi_{2}))!}{N!}
	\cdot N ^{ \#(\pi_{1}) + \#(\pi_{2}) - \frac{1}{2}}.
	%
	\end{align*}	

	Assume  $\mathrm{p}_{\widehat\sigma^{-1} \circ (\widehat{\alpha} \sqcup \pi^{\text{odd}})}$ is the zero polynomial. 
	Then, by Proposition \ref{minimal.special.polynomials}, there is a symmetric pairing partition $\eta \leq \widehat{\alpha} \sqcup \pi^{\text{odd}}$ such that $\mathrm{p}_{\widehat\sigma^{-1} \circ \eta}$ is also the zero polynomial, and hence, the partition $\eta$ must satisfy one of the conditions \textit{(1)-(6)} from Proposition \ref{prop.minimal.special.partitions}.
	However, we have $\eta=\eta_{\text{odd}} \sqcup \eta_{\text{even}}$, since $\eta \leq \widehat{\alpha} \sqcup \pi^{\text{odd}}$, and neither $2m_1$ or $2m_2$ is odd, so conditions \textit{(\ref{prop.minimal.special.partitions.2})} and \textit{(4)-(6)} can not hold. 
	Now, note that if $\mathrm{p}_{\widehat\sigma^{-1} \circ (\widehat{\alpha} \sqcup \beta^{\text{odd}})}$ is zero polynomial for some partition $\beta \in P(\pm 2m)$, then
	\begin{align*}
	\sum_{\substack{ \mathbf{i}:[\pm {4m} ] \rightarrow [N] \\  \kernel{\mathbf{i}}= \widehat{\alpha} \sqcup \beta^{\text{odd}} }} 		\widehat{\mathbf{h}}(\mathbf{i} \circ \bm{\widehat{\sigma}} ) 	
	=
	\frac{N!}{	(	N	-	\#(\widehat{\alpha} \sqcup \beta^{\text{odd}})  + 1) !}
	=
	\frac{N!}{	(	N	-	m  - \#( \beta)  + 1) !} 
	\end{align*}
	since we would have $\widehat{\mathbf{h}}(\mathbf{i} \circ \bm{\widehat{\sigma}} ) =1$ for any function $ \mathbf{i}:[\pm {4m} ] \rightarrow [N]$ satisfying $\kernel{\mathbf{i}}= \widehat{\alpha} \sqcup \beta^{\text{odd}}$. 
	Hence, if ${\alpha}_{1}$, ${\alpha}_{2}$, $\pi_{1}$, and $\pi_{2} $ are all even partitions and the polynomial $\mathrm{p}_{\sigma^{-1} \circ \hat{\alpha} \sqcup \pi^{\text{odd}}_{1} \sqcup \pi^{\text{odd}}_{2} }$ is zero, from (\ref{dist.entries.signature.matrix}),  (\ref{dist.entries.signed.perm.matrix}), and (\ref{eqn.unnormalized.fluct.HXH}), we obtain 
	\begin{align}\label{eqn.vanishing.pi.eta}
	& 
	\frac{(N-\#(\pi))!}{ N!} \cdot 
	\sum_{\substack{ 
			\mathbf{i}:[\pm {4m} ] \rightarrow [N]  \\ \kernel{\mathbf{i}} = \hat{{\alpha}} \sqcup \pi^{\text{odd}}} }
	\widehat{\mathbf{h}}(\mathbf{i}\circ \bm{\widehat{\sigma}}) - 	
	\frac{(N-\#(\pi_{1}))!}{ N!} \cdot
	\frac{(N-\#(\pi_{2}))!}{ N!}	\cdot
	\sum_{\substack{ 
			\mathbf{i}:[\pm {4m} ] \rightarrow [N]  \\ \kernel{\mathbf{i}}= \hat{{\alpha}}\sqcup \pi^{\text{odd}}_1 \sqcup \pi^{\text{odd}}_2}}
	\widehat{\mathbf{h}}(\mathbf{i}\circ \bm{\widehat{\sigma}})  \nonumber \nonumber \\ 	
	& = N^{m}\mathfrak{C}_{2}[\pi,{\alpha}] = O\left( N^{m-1} \right)  
	\end{align}
	On the other hand, if either ${\alpha}_{1}$ or ${\alpha}_{2}$ is not an even partition, from (\ref{dist.entries.signature.matrix})  and (\ref{eqn.unnormalized.fluct.HXH}), 
	we get
	\begin{align}\label{eqn.non-vanishing.pi.eta}
	\mathfrak{C}_{2}[\pi,{\alpha}]
	=	&  
	\frac{1}{N^{ m } } \cdot \frac{(N-\#(\pi))!}{ N!} \cdot 
	\sum_{\substack{ 
			\mathbf{i}:[\pm {4m} ] \rightarrow [N]  \\ \kernel{\mathbf{i}} = \hat{{\alpha}} \sqcup \pi^{\text{odd}}} }
	\widehat{\mathbf{h}}(\mathbf{i}\circ \bm{\widehat{\sigma}}) 
	=
	1 + O\left(N^{-1}\right) .
	\end{align}
	Suppose ${\eta}$ satisfies \textit{(\ref{prop.minimal.special.partitions.3})}  from  Proposition \ref{prop.minimal.special.partitions} and let $1\leq k \leq 4m_{1}$ and $4m_{1}+1 \leq l \leq 4m_{1}+4m_{2}$ such that 
	\(
	\eta = 	\left\{
	\{\widehat\sigma^{t_{1}}(-k),\widehat\sigma^{-t_{1}}(k)\},
	\{\widehat\sigma^{t_{2}}(-l),\widehat\sigma^{-t_{2}}(l)\}
	\mid
	t_{2}, t_{2} \geq 0	\right\}. 
	\)
	We need to consider three cases: $k$ and $l$ are both odd, $k+l$ is odd, and $k$ and $l$ are both even. 
	First, if $k$ and $l$ are both odd, then $\mathrm{p}_{\sigma^{-1} \circ \hat{\alpha} \sqcup \pi^{\text{odd}}_{1} \sqcup \pi^{\text{odd}}_{2} }$ is also the zero polynomial, $\pi_{1}$ and $\pi_{2} $ are both even partitions, and ${\alpha} = {\alpha}_{1} \sqcup {\alpha}_{2}$, so  (\ref{eqn.vanishing.pi.eta}) holds. 
	Second, if $k+l$ is odd, then ${\alpha} = {\alpha}_{1} \sqcup {\alpha}_{2}$ and $\pi = \pi_{1} \sqcup \pi_{2} $, but then (\ref{eqn.vanishing.pi.eta}) holds too. 
	Third, if $k$ and $l$ are both even, then $\pi = \pi_{1} \sqcup \pi_{2} $ and either ${\alpha} = {\alpha}_{1} \sqcup {\alpha}_{2}$ or ${\alpha} \neq {\alpha}_{1} \sqcup {\alpha}_{2}$. 
	However, if  ${\alpha} = {\alpha}_{1} \sqcup {\alpha}_{2}$, we already know that  $\mathfrak{C}_{2}[\pi,{\alpha}]= O(N^{-1})$ from (\ref{eqn.vanishing.pi.eta}), and if ${\alpha} \neq {\alpha}_{1} \sqcup {\alpha}_{2}$, then  ${\alpha}_{1}$ and ${\alpha}_{2}$ are not even partitions, so (\ref{eqn.non-vanishing.pi.eta}) holds. 
	Finally, if ${\eta}$ satisfies  \textit{(\ref{prop.minimal.special.partitions.1})}  from  Proposition \ref{prop.minimal.special.partitions}, we must have ${\alpha} \neq {\alpha}_{1}\sqcup{\alpha}_{2}$, so we obtain  $\mathfrak{C}_{2}[\pi,{\alpha}] = 1 + O\left(N^{-1}\right) $. 
\end{proof}
%
%
%

\begin{proof}[\textbf{Proof of (\ref{prop.fluc.xhx}) from Proposition \ref{prop.fluc}}]  
	%
	%
	%
	%
	%
	Fix an even partition $\pi \in P(\pm 2m)$ such that $\pi \leq \theta$ for some partition $\theta \in P_{\chi\chi}(\pm 2m)$ and 
	%
	%
	let $\mathfrak{C}_{2}[\pi,\alpha] $ be  given by (\ref{eqn.unnormalized.fluct.HXH}) for each pairing partition $\alpha \in P_{2}(2m)$.  
	By Proposition \ref{prop.fluc.xhx.C[p,a]}, we have that
	\[
	\sum_{\alpha\in P_{2}(2m)} \mathfrak{C}_{2} \left[ \pi, \alpha \right] 
	= 
	\abs*{E_{\pi}} + \abs*{F_{\pi}} - \abs*{E_{\pi} \cap F_{\pi}} 
	+ O ( N^{-1/2}  ) 
	\]
	where $E_{\pi}$ and $F_{\pi}$ are the subsets of $P_{2}(2m)$ given by
	\[
	E_{\pi} = \left\{ \alpha\in P_{2}(2m) \left| \begin{array}{c}
	\eta \leq \widehat{\alpha} \sqcup \pi^{\text{odd}} \text{ for some symmetric pairing } \eta \in P(\pm 4m)   \\ 
	\text{satisfying (\ref{prop.minimal.special.partitions.1}) from Proposition \ref{prop.minimal.special.partitions} }
	\end{array} \right. \right\}
	\]
	and
	\[
	F_{\pi} = \left\{ \alpha\in P_{2}(2m) \left| \begin{array}{c}
	\alpha \neq \alpha_{1} \sqcup \alpha_{2} \text{ and }\eta \leq \widehat{\alpha} \sqcup \pi^{\text{odd}} \text{ for some symmetric pairing } \eta \in P(\pm 4m)   \\ 
	\text{satisfying (\ref{prop.minimal.special.partitions.3}) from Proposition \ref{prop.minimal.special.partitions} with } k \text { and } l \text{ even}
	\end{array} \right. \right\}. 
	\]  
	Thus, by Proposition \ref{prop.cpi-and-Cpi-nu}, we only need to show that  $E_{\pi} \neq \emptyset$ implies $\abs*{E_{\pi}} =1$, $F_{\pi} \neq \emptyset$ implies $\abs*{F_{\pi}} = 2$, and $E_{\pi} \cap F_{\pi}$ is empty. 
	%
	%
	%
	%
	%

	%
	%
	%
	%
	Let $\alpha$ and $\beta$ be pairing partitions in $P_{2}(2m)$ and suppose there are symmetric pairings $\eta_{\alpha}, \eta_{\beta} \in P(\pm {4m})$ satisfying (\ref{prop.minimal.special.partitions.1}) from Proposition \ref{prop.minimal.special.partitions}, $\eta_{\alpha} \leq \widehat{\alpha} \sqcup \pi^{\text{odd}}$,  and  $\eta_{\beta} \leq \widehat{\beta} \sqcup \pi^{\text{odd}}$. 
	Then, there are integers $ 2m_{1} + 1  \leq l_{\alpha} , l_{\beta} \leq  2m_{1} + 2m_{2}$ so that 
	\[
	\widehat{\sigma}^{-t}(2l_{\alpha}-1)  
	\sim_{\eta_{\alpha}}
	\widehat{\sigma}^{t}(-1) 
	\sim_{\eta_{\beta}}
	\widehat{\sigma}^{-t}(2l_{\beta}-1) 
	%
	\quad \forall  t\geq 0
	\]
	where $\widehat{\sigma}$ is the permutation given by
	\[
	\widehat{\sigma} = 	(-1,1,-2,2,\ldots,-4m_{1},4m_{1})
	(-4m_{1}-1,4m_{1}+1,-4m_{1}-2,\ldots,-4m_{1}-4m_{2},4m_{1}+4m_{2}).
	\]
	But then, we must have
	\begin{align}\label{rel.sigma.hat.alpha}
	\widehat{\sigma}^{-t}(2l_{\alpha}-1) 
	\sim_{\widehat\alpha \sqcup \pi^{\text{odd}}}
	\widehat{\sigma}^{t}(-1) 
	\sim_{\widehat\beta \sqcup \pi^{\text{odd}}} 
	\widehat{\sigma}^{-t}(2l_{\beta}-1) 
	%
	\quad \forall  t\geq 0
	\end{align}
	since $\eta_{\alpha} \leq \widehat{\alpha} \sqcup \pi^{\text{odd}}$ and  $\eta_{\beta} \leq \widehat{\beta} \sqcup \pi^{\text{odd}}$,  
	and thus, we get that
	\[
	{\sigma}^{-t}(l_{\alpha})
	\sim_{ \pi } 
	{\sigma}^{t}(-1) 
	\sim_{ \pi } 
	{\sigma}^{-t}(l_{\beta}) 
	%
	\quad \forall  t\geq 0
	\]	
	where ${\sigma}$ is the permutation given by
	\[{\sigma} = 	(-1,1,-2,2,\ldots,-2m_{1},2m_{1})
	(-2m_{1}-1,2m_{1}+1,-2m_{1}-2,\ldots,-2m_{1}-2m_{2},2m_{1}+2m_{2}) . \]
	In particular, for $t=0$, we obtain $l_{\alpha} \sim_{ \pi } -1 \sim_{ \pi } l_{\beta}$, and thus, we have  $l_{\alpha}=l_{\beta}$ since $\pi$ is an even partition with only blocks of the form  $\{-k,+k,-l,l\}$ and $\{+k,-l\}$. 
	Therefore, it follows from (\ref{rel.sigma.hat.alpha}) that  $\widehat{\alpha}=\widehat{\beta}$, or, equivalently, ${\alpha}={\beta}$. 
	This shows that  $E_{\pi} \neq \emptyset$ implies $\abs*{E_{\pi}} =1$.

	%
	%
	%
	%
	Suppose now there are symmetric pairings $\eta_{\alpha}, \eta_{\beta} \in P(\pm {4m})$ satisfying (\ref{prop.minimal.special.partitions.3}) from Proposition \ref{prop.minimal.special.partitions} with $k$ and $l$ even, $\eta_{\alpha} \leq \widehat{\alpha} \sqcup \pi^{\text{odd}}$,  and  $\eta_{\beta} \leq \widehat{\beta} \sqcup \pi^{\text{odd}}$. 
	Then, there exists integers $ 1  \leq k_{\alpha}  \leq  k_{\beta} \leq  2m_{1}$ so that %
	\[
	\widehat{\sigma}^{t}(-2 k_{\alpha}) \sim_{\eta_{\alpha}} \widehat{\sigma}^{-t}(2 k_{\alpha}) 
	\quad \text{and} \quad 
	\widehat{\sigma}^{t}(-2 k_{\beta}) \sim_{\eta_{\beta}} \widehat{\sigma}^{-t}(2 k_{\beta})
	\quad \forall  t\geq 0, 
	\]
	and hence, we get  
	\begin{align}\label{rel.sigma.hat.alpha.pi.odd}
	\widehat{\sigma}^{t}(-2 k_{\alpha}) \sim_{ \widehat{\alpha} \sqcup \pi^{\text{odd}}} \widehat{\sigma}^{-t}(2 k_{\alpha}) 
	\quad \text{and} \quad 
	\widehat{\sigma}^{t}(-2 k_{\beta}) \sim_{\widehat{\beta} \sqcup \pi^{\text{odd}}} \widehat{\sigma}^{-t}(2 k_{\beta})
	\quad \forall  t\geq 0
	\end{align}
	since $\eta_{\alpha} \leq \widehat{\alpha} \sqcup \pi^{\text{odd}}$ and  $\eta_{\beta} \leq \widehat{\beta} \sqcup \pi^{\text{odd}}$; 
	in particular, we must have 
	\[
	\widehat{\sigma}^{4t}(2 k_{\alpha}) \sim_{ \widehat{\alpha} } \widehat{\sigma}^{-4t}(-2 k_{\alpha}) 
	\quad \text{and} \quad 
	\widehat{\sigma}^{4t}(2 k_{\beta}) \sim_{\widehat{\beta} } \widehat{\sigma}^{-4t}(-2 k_{\beta}) 
	\quad \forall  t\geq 0. 
	\]
	Now, since ${\alpha}$ and ${\beta}$ are pairing partitions of $[2m_{1}+2m_{2}]$,  $\widehat{\alpha} = \{\{+2k,-2k,+2l,-2l\} \mid \{k,l\} \in \alpha \} $, and  $\widehat{\beta} = \{\{+2k,-2k,+2l,-2l\} \mid \{k,l\} \in \beta \} $, we get 
	\[
	{\sigma}^{2t}( k_{\alpha} )	\sim_{\alpha} 	{\sigma}^{-2t}( k_{\alpha} ) 
	\quad \text{and} \quad 
	{\sigma}^{2t}( k_{\beta} )	\sim_{\beta} 	{\sigma}^{-2t}( k_{\beta} )
	\quad \forall  t\geq 0, 
	\]
	where ${\sigma}$ is the permutation defined above; hence, we obtain
	\[
	\alpha_{1} 	= \{\{k_{\alpha}\},\{k_{\alpha}+1,k_{\alpha}-1\},\ldots,\{k_{\alpha}+m_{1}-1,k_{\alpha}-m_{1}+1\},\{k_{\alpha}+m_{1}\}\} 
	\]
	and
	\[
	\beta_{1} = \{\{k_{\beta}\},\{k_{\beta}+1,k_{\beta}-1\},\ldots,\{k_{\beta}+m_{1}-1,k_{\beta}-m_{1}+1\},\{k_{\beta}+m_{1}\}\} 
	\]
	where $\alpha_{1}$ and $\beta_{1}$ denote the restrictions of $\alpha$ and $\beta$, respectively, to the set $[\pm 2m_{1}]$. 
	Let us show that $\alpha_{1}=\beta_{1}$. 
	From (\ref{rel.sigma.hat.alpha.pi.odd}), we also have that 
	\begin{align*}
	\widehat{\sigma}^{4t}( - 2 k_{\alpha}-1) \sim_{ \pi^{\text{odd}} } \widehat{\sigma}^{-4t}( 2 k_{\alpha}-1), & & 
	\widehat{\sigma}^{4t}( 2 k_{\alpha}+1) \sim_{\pi^{\text{odd}} } \widehat{\sigma}^{-4t}( -2 k_{\alpha}+1), \\
	\widehat{\sigma}^{4t}( - 2 k_{\beta}-1) \sim_{ \pi^{\text{odd}} } \widehat{\sigma}^{-4t}( 2 k_{\beta}-1), &
	\quad \text{and}  &
	\widehat{\sigma}^{4t}( 2 k_{\beta}+1) \sim_{\pi^{\text{odd}} } \widehat{\sigma}^{-4t}( -2 k_{\beta}+1)
	\end{align*}
	for every integer $t \geq 0$, and thus, since $\pi^{\text{odd}} = \{\{ 2k-\sign{(k)} \mid k \in B \} \mid B \in \pi \}\}$, we obtain  
	\[
	{\sigma}^{t+1}( k_{\alpha} )	= 
	{\sigma}^{t}( - k_{\alpha} - 1  ) 
			\sim_{ \pi } 
	{\sigma}^{-t}( k_{\alpha} )	
	\quad \text{and} \quad 
	{\sigma}^{t+1}( k_{\beta} )		=
	{\sigma}^{t}( -k_{\beta} - 1 )  
			\sim_{ \pi } 
	{\sigma}^{-t}( k_{\beta} )	
	\quad \forall  t\geq 0. 
	\]
	Let $t=k_{\beta}-k_{\alpha}$ and note that 
	\[
	{\sigma}^{2t}(	k_{\alpha} )
	\sim_{\pi}
	{\sigma}^{-2t+1}(	k_{\alpha} )
	\sim_{\pi}
	{\sigma}^{2t+1}(	k_{\alpha} )  
	\]
	since 
	\[
	k_{\beta} = 
	{\sigma}^{2t}(	k_{\alpha} )
	\sim_{\pi}
	{\sigma}^{-2t+1}(	k_{\alpha} ) 
	\quad	\text{and}	\quad 	
	{\sigma}^{2t+1}(k_{ \alpha } ) = \sigma^{}(k_{\beta}) \sim_{\pi}  k_{\beta}; 
	 \] 
	moreover, since ${\sigma}^{2t}(k_{\alpha} ) > 0$, ${\sigma}^{-2t+1}(	k_{\alpha} ), {\sigma}^{2t+1}(	k_{\alpha} ) < 0$, and $\pi$ is a partition with only blocks of the form  $\{-k,+k,-l,l\}$ and $\{+k,-l\}$ with $k,l > 0$, we must have 
	\[
	{\sigma}^{2t}(	k_{\alpha} ) 	=	-{\sigma}^{-2t+1}(	k_{\alpha} ) ,
	\quad 
	{\sigma}^{2t}(	k_{\alpha} ) =	-{\sigma}^{2t+1}(	k_{\alpha} ), 
	\quad	\text{or}	\quad 
	{\sigma}^{-2t+1}(	k_{\alpha} )  = {\sigma}^{2t+1}(	k_{\alpha} ), 
	\]
	or, equivalently, 
	\begin{align*}
	{\sigma}^{4t-2}(	k_{\alpha} ) 	=		k_{\alpha} ,
	\quad 
	{\sigma}^{-2}(	k_{\alpha} )		=		k_{\alpha} ,
	\quad	\text{or}	\quad 
	{\sigma}^{-4t}(	k_{\alpha} )  		=		k_{\alpha}. 
	\end{align*}
	But the equality ${\sigma}^{s}(k_{\alpha}) = k_{\alpha}$ holds if only if $s \equiv 0 \mod 4m_{1}$, so only ${\sigma}^{-4t}( k_{\alpha} ) = k_{\alpha}$ can hold, and thus, we get 	$k_{\alpha} = k_{\beta}$ or $k_{\beta} = k_{\alpha} + m_{1}$ since $0 \leq t = k_{\beta} - k_{\alpha} \leq 2m_{1}-1$. 
	Therefore, $\alpha_{1} = \beta_{1}$ and there is an integer $1\leq k=k_{\alpha} \leq 2m_{1}$ so that 
	\begin{align*}
	\alpha_{1}
	& =
	\{ \{k\}, \{k+1,k-1\}, \ldots, \{k+m_{1}-1,k-m_{1}+1\}, \{k+m_{1}\} \}	 
	= 
	\beta_{1} \\ 
	& =
	\{
	\{ \sigma^{2t}(k) ,\sigma^{-2t}(k)    \} 
	\mid 0 \leq t \leq m_{1}  \} .
	\end{align*}
	Similarly, letting $\alpha_{2}$ and $\beta_{2}$ denote the restrictions of $\alpha$ and $\beta$, respectively, to the set $[\pm (2m_{1}+2m_{2})] \setminus [\pm 2m_{1}]$, we have $\alpha_{2} = \beta_{2}$ and there is an integer $2m_{1}+1 \leq \l_{} \leq 2m_{1}+2m_{2}$ so that 
	\begin{align*}
	\alpha_{2} 
			&= 
	\{ \{l\}, \{l+1,l-1\}, \ldots, \{l+m_{2}-1,l-m_{2}+1\}, \{l+m_{2}\} \}	
			=
	\beta_{2} \\  
			&=
	\{	\{ \sigma^{2t}(l) ,\sigma^{-2t}(l)    \} 
			\mid 0 \leq t \leq m_{2}  \}.  
	\end{align*}
	This shows that $\abs*{F_{\pi}} = 2$ provided $F_{\pi} \neq \emptyset$ since  $\gamma \in F_{\pi}$ implies 
	\[
	\gamma  = 
	\{\{k, l\} , \{k+m_{1},l+m_{2}\}\} 	\cup \widetilde{\alpha}
	\quad \text{ or } \quad 
	\gamma  = 
	\{\{k, l+m_{2}\}, \{k+m_{1},l\}\} \cup \widetilde{\alpha}
	\]
	where  $\widetilde{\alpha} =\{\{ \sigma^{2t}(k) ,\sigma^{-2t}(k)    \} \mid 1 \leq t \leq m_{1}-1  \} 
	\cup \{\{ \sigma^{2t}(l) ,\sigma^{-2t}(l)    \} \mid 1 \leq t \leq m_{2}-1  \} $. 
	%
	
	%
	%
	%
	
	%
	%
	%
	%
	Finally, $E_{\pi} \cap F_{\pi}$ is empty since $E_{\pi} \neq \emptyset$ implies $\pi \neq \pi_{1} \sqcup \pi_{2}$, and, on the other hand, $ F_{\pi} \neq \emptyset $ implies $\pi=\pi_{1} \sqcup \pi_{2}$. 
\end{proof}
%
%
%
%
%

\section{Concluding remarks}\label{sec.concluding}

\begin{enumerate}
	\item The random matrix ensemble $\{W_{N,1},H_{N}W_{N,2}\}_{N=1}^{\infty}$ satisfies the hypothesis in Lemma \ref{lemma.bounded.cumulant.property}, and hence, Theorem \ref{thm.second.order.freeness} is proved once we show \eqref{thm.second.order.freeness.moments} holds.
	To that end, we first define appropriate versions of the functions  $\mathbf{w}_{}(\mathbf{i,j})$, $\mathbf{w}_{1}(\mathbf{i,j})$, $\mathbf{w}_{2}(\mathbf{i,j})$, and show that \eqref{eqn.relating.c[pi].C[pi]} still holds in this case. 
	Then, following similar steps to those in the proof of Proposition \ref{prop.fluc} and Lemma \ref{lemma.fluc.mobius} and letting $U^{*}_{N,{i}_{1}} U^{}_{N,{i}_{2}} = \frac{1}{\sqrt{N}} W_{N,1}^{*} H^{}_{N} W_{N,2}$, we conclude 
	\begin{align*}
	\sum_{\substack{
			\pi \in P_{\text{even}}(\pm {2m})	\\
			\pi \leq \theta }} 
	N^{m}
	\mathfrak{c}_{2} \left[ \pi \right]	
	\mu(\pi, \theta)
	=&
	\left\{\begin{array}{cl}
	1  + 	O\left( N^{-1/2} \right)& 			
	\text{if  } \theta \text{ is a pairing partition satisfying (\ref{prop.minimal.special.partitions.1})} \\
	& \text{from Proposition \ref{prop.minimal.special.partitions}}, \\
	O\left( N^{-1/2}\right)	& 
	\text{otherwise.} 
	\end{array}\right.
	\end{align*}
	\item One can replace the Discrete Fourier transform $H_{N}$ in the unitary random matrix ensemble  $\{W_{N},  H_{N}W_{N} /\sqrt{N}, X_{N}H_{N}W_{N}/\sqrt{N} \}_{N=1}^{\infty}$  by any Hadamard matrix $H'_{N}$ and still get an asymptotically liberating ensemble, see \cite{anderson2014asymptotically}. 
	Moreover, key equations in this paper involving $H_{N}$ still holds when we replace $H_{N}$ by a general Hadamard matrix $H'_{N}$, for instance, \eqref{eqn.DFT.number.blocks}, \eqref{eqn.shifted.DFT.injective.graph.sums}, and \eqref{eqn.graph.sums.h.injective.vs.non-injective}. 
	Thus, to determine the corresponding induced fluctuations moments, one needs to compute graph sums of $H'_{N}$ and obtain similar results to those from Section \ref{subsec.graph.sums.dft}. 
	However, the results for graph sums of $H_{N}$ were possible thanks to the reciprocity theorem for generalized Gauss sums and it is not obvious what could be used for a general $H'_{N}$. 	 
	\item Although Proposition \ref{prop.asymptotic.centering} and Corollary \ref{corollary.asymptotic.centering} give equivalent conditions only for point-wise uniform boundedness, similar statements and proofs provide us with corresponding conditions for the point-wise convergence of a sequence of multi-linear functionals. 
	These conditions together with bounds for graph sums can be exploited to study higher order moments.
	In particular, the relations \eqref{cumm.lemma.4} and \eqref{eqn.fluct.moments.c2.chi.chi.1} can be used to determine the higher order moments induced by Haar-unitary and Haar-orthogonal via the Weingarten Calculus from \cite{MR1959915} and \cite{MR2217291}. 
\end{enumerate}

\bibliographystyle{acm}
\bibliography{bibliography}

\end{document}